\documentclass[a4paper,reqno]{amsart}%{article}
\pdfoutput=1

\usepackage{amsmath,amsfonts,amssymb,amsthm,stmaryrd,bbm,graphicx,mathtools,enumerate}
\usepackage{mathrsfs}
\usepackage[T1]{fontenc} % pour taper les lettres accentuées
\usepackage[latin1]{inputenc}
\usepackage[english]{babel}
\usepackage[tracking,spacing,kerning,babel]{microtype}
\usepackage{wasysym} %\smiley 
\usepackage{color}
\usepackage{xcolor}
    	\usepackage{framed} % Theorems
\usepackage{wrapfig} % To use the nice \begin{wrapfigure} command ! 
\usepackage[final]{hyperref}   % ICHANGED THIS ALL PARAGRAPH
\hypersetup{
    linktoc=page,
    linkcolor=red,          % color of internal links
    citecolor=blue,        % color of links to bibliography
    filecolor=blue,      % color of file links
    urlcolor=cyan,
   colorlinks=true           % color of external links
}

\usepackage{lipsum} % pour créer plein de textes !! (pour faire des tests !)

\usepackage{minitoc}
\usepackage{multicol}
\newtheoremstyle{mine}
{\baselineskip}
{\baselineskip}
{\itshape}
{
}
{\bfseries}
{.}
{.5em}%{\newline}%{.5em}
{#1 #2\ifx#3\relax\else~(#3)\fi}

\theoremstyle{mine}

\newtheorem{theorem}{Theorem}
\numberwithin{theorem}{section}
\newtheorem{corollary}[theorem]{Corollary}
\newtheorem{proposition}[theorem]{Proposition}
\newtheorem{lemma}[theorem]{Lemma}
\newtheorem{definition}[theorem]{Definition}
\newtheorem{conjecture}{Conjecture} 
\numberwithin{equation}{section}
\newtheorem{remark}[theorem]{Remark}

\colorlet{shadecolor}{blue!10}

\newcommand{\margin}[1]{\textcolor{magenta}{*}\marginpar[\textcolor{magenta} {  \raggedleft  \footnotesize  #1 }  ]{ \textcolor{magenta} { \raggedright  \footnotesize  #1 }  }}

\def\rm{\reversemarginpar}

\let\qed=\QED

\renewcommand{\epsilon}{\varepsilon}

\newcommand{\A}{\mathbb{A}}
\newcommand{\F}{\mathbb{F}}
\newcommand{\R}{\mathbb{R}}

\newcommand{\C}{\mathbb{C}}
\newcommand{\Z}{\mathbb{Z}}
\newcommand{\N}{\mathbb{N}}
\renewcommand{\S}{\mathbb{S}}
\newcommand{\rec}{\mathop{rec}}
\newcommand{\1}{\mathbf 1}

\newcommand{\V}{\mathcal V} 
\def\H{\mathbb{H}}

\def\calC{\mathcal{C}}

\def\calF{\mathcal{F}}

\def\calN{\mathcal{N}}
\def\calS{\mathcal{S}}

\def\SLE{\mathrm{SLE}}

\def\diam{\mathrm{diam}}

\def\var{\mathop{\mathrm{Var}}}
\def\Var#1{\mathrm{Var}\bigl[ #1\bigr]}

\def\Cov#1{\mathrm{Cov}\bigl[ #1\bigr]}

\def\Im{{\rm Im}\,}
\def\Re{{\rm Re}\,}

\def\P{\mathbb{P}} %{{\bf P}}
\def\E{\mathbb{E}} %{{\bf E}}
\def\md{\mid}

\def \eps {\epsilon}

\def\Bb#1#2{{\def\md{\bigm| }#1\bigl[#2\bigr]}}

\def\Pb{\Bb\P}
\def\Eb{\Bb\E}

\def\FK#1#2#3{{\def\md{\bigm| } \P_{#1}^{\,#2}  \bigl[  #3 \bigr]}}
\def\EFK#1#2#3{{\def\md{\bigm| } \E_{#1}^{\,#2}  \bigl[  #3 \bigr]}}

\def \p {{\partial}}

\def\<#1{\langle #1\rangle}
\newcommand{\red}[1]{{\color{red}#1}}
\newcommand{\blue}[1]{{\color{blue}#1}}
\newcommand{\purple}[1]{{\color{purple}#1}}

\def\Wick#1{\,\colon\!\! \, #1 \, \!\colon}

\def\nn{\nonumber}
\def\bi{\begin{itemize}}  %USE\bi[WHATEVER]
\def\ei{\end{itemize}}
\def\bnum{\begin{enumerate}} % USE \bnum[i)] if want i), ii) .. OR \bnum[{\bf (a)}] etc ..!
\def\enum{\end{enumerate}}
\def\ni{\noindent}
\def\bf{\bfseries}

\def\GFF{\mathrm{GFF}}

\def\ba{\mathbf{a}}
\def\GFF{\mathrm{GFF}}
\def\SG{\mathrm{SG}}
\def\IV{\mathrm{IV}}
\def\bm{\mathbf{m}}
\def\bz{\mathbf{z}}
\newcommand{\Mod}[1]{\ (\mathrm{mod}\ #1)}

\title[Statistical reconstruction of the GFF and KT transition]
{
Statistical reconstruction of the Gaussian free field and KT transition
}

\author{Christophe Garban \and Avelio Sepúlveda}

\address
{Université Claude Bernard Lyon 1, CNRS UMR 5208, Institut Camille Jordan, 69622 Villeurbanne, France}
\email{garban@math.univ-lyon1.fr; sepulveda@math.univ-lyon1.fr}

\begin{document}

\maketitle

\begin{abstract}
In this paper, we focus on the following question. Assume $\phi$ is a discrete Gaussian free field (GFF) on $\Lambda \subset  \frac 1 n \Z^2$ and that we are given $e^{iT \phi}$, or equivalently $\phi \pmod{\frac {2\pi} T}$. Can we recover the macroscopic observables of $\phi$ up to $o(1)$ precision? We prove that this statistical reconstruction problem undergoes the following Kosterlitz-Thouless type phase transition:

\bi
\item If $T<T_{rec}^-$ , one can fully recover  $\phi$ from the knowledge of $\phi \pmod{\frac {2\pi} T}$. In this regime our proof relies on a new type of Peierls argument which we call {\em annealed} Peierls argument and which allows us to deal with an unknown {\em quenched} groundstate. 
\item
If $T>T_{rec}^+$, it is impossible to fully recover the field $\phi$ from the knowledge of $\phi \pmod{\frac {2\pi} T}$.
To prove this result, we generalise the delocalisation theorem by Fröhlich-Spencer to the case of  integer-valued GFF in an inhomogeneous medium. This delocalisation result  is of independent interest and we give an application of our techniques to the {\em random-phase Sine-Gordon model} in Appendix \ref{a.SG}. Also, an interesting connection with Riemann-theta functions is drawn along the proof.
\ei 
\smallskip

This statistical reconstruction problem is motivated by the two-dimensional XY and Villain models. Indeed, at low-temperature $T$, the large scale fluctuations of these continuous spin systems are conjectured to be governed by a Gaussian free field. It is then natural to ask if one can recover the underlying macroscopic GFF from the observation of the spins of the XY or Villain model. 

Another motivation for this work is that it provides us with an ``integrable model'' (the GFF) that undergoes a KT transition. 
\end{abstract}

\section{Introduction}

\begin{figure}[!htp]
\begin{center}
\includegraphics[width=0.4\textwidth]{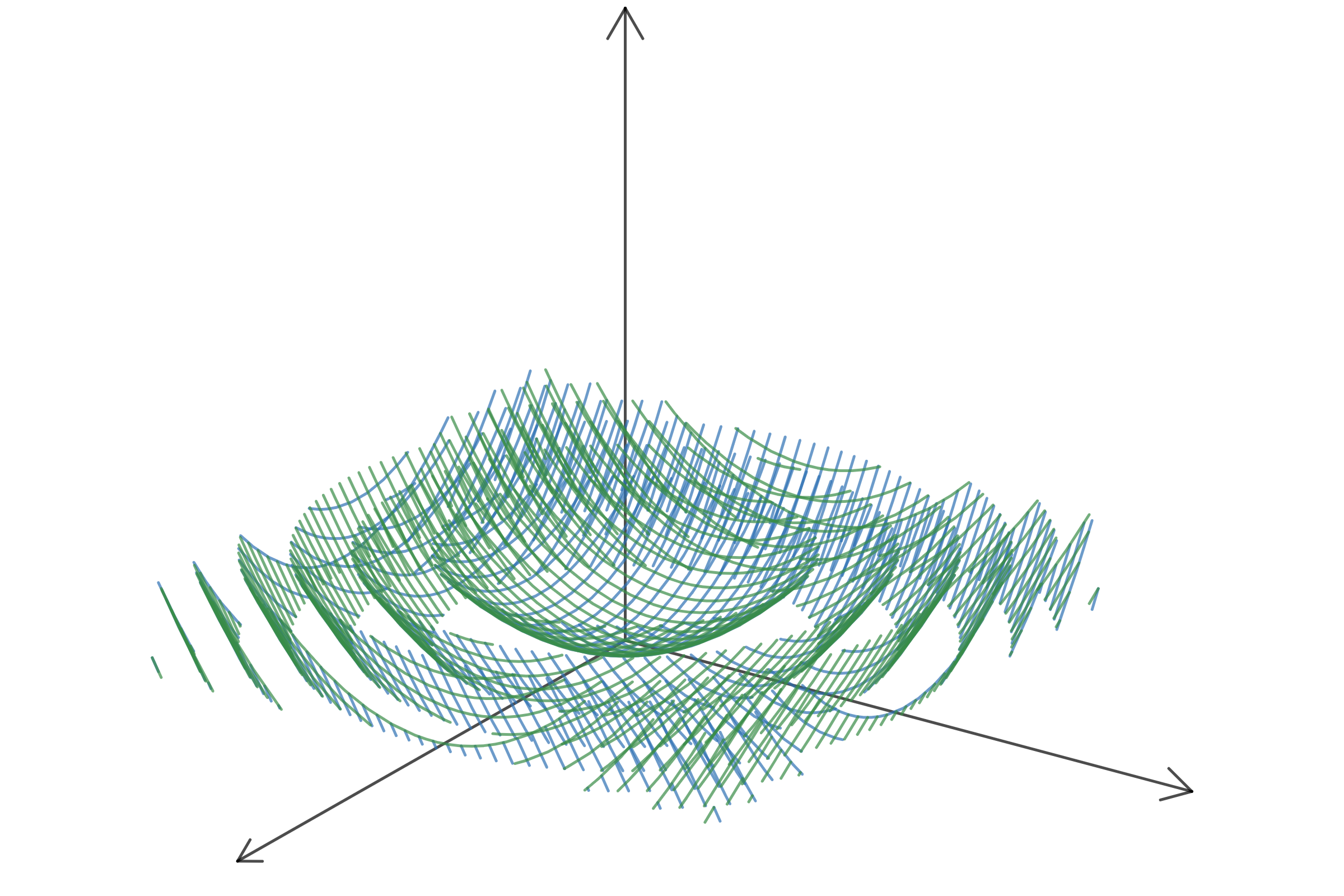}
\includegraphics[width=0.4\textwidth]{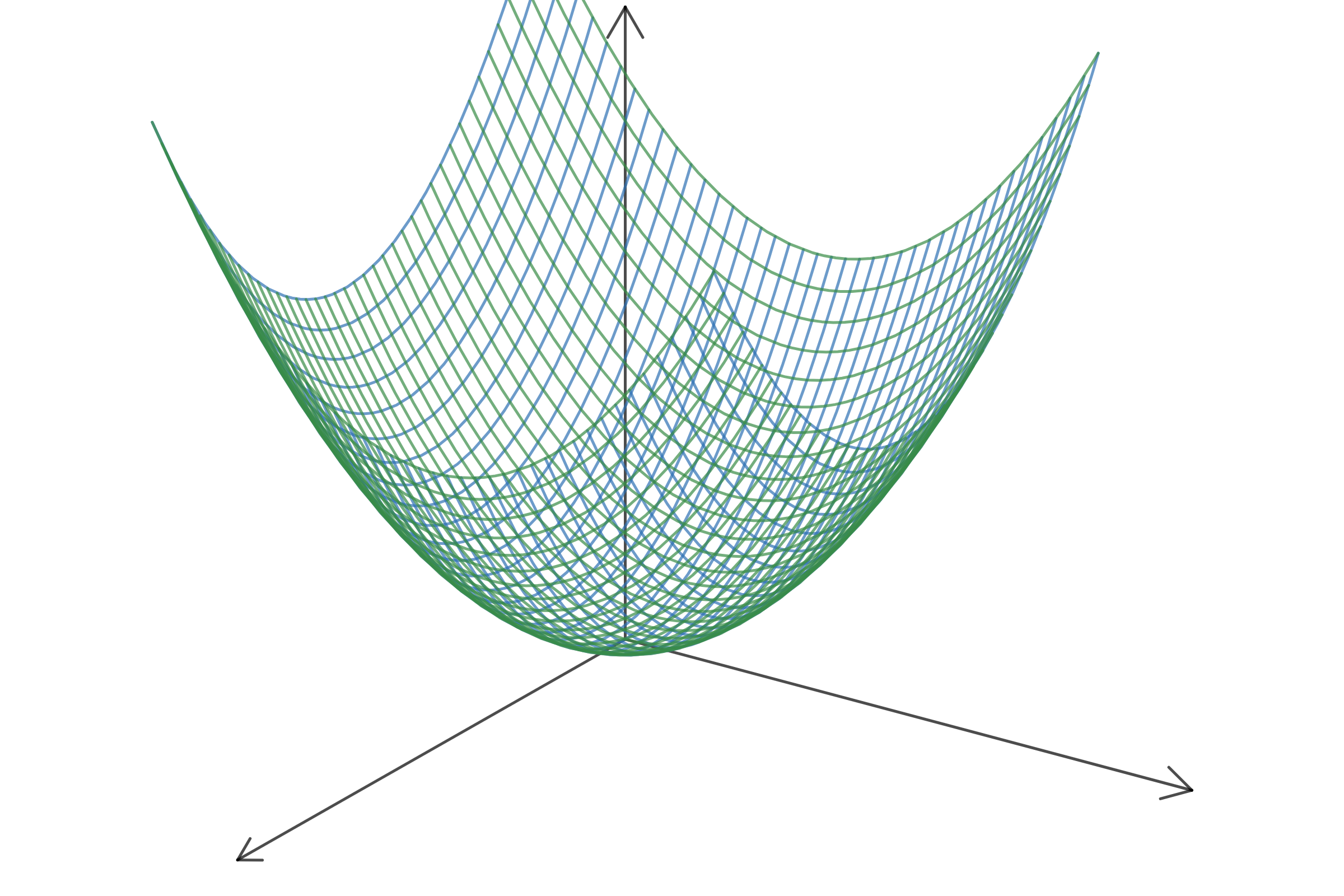}
\end{center}
\caption{If you are given the values of a function $f$ modulo 1 (left), can you reconstruct what $f$ is (right) ? If $f$ is smooth as in this example, sure you may. But what if $f$ is an instance of a $2d$ Gaussian free field? Analyzing this statistical reconstruction problem is the aim of this paper.}\label{}
\end{figure}

\subsection{Main result}
We work on the graph $\Lambda_n:= [-1,1]^2 \cap \frac{1}{n} \Z^2$ and for functions $f,g: \Lambda_n\mapsto \R$, we denote
\begin{equation}
\langle f,g \rangle:= \sum_{x\in \Lambda_n} f(x)g(x).
\end{equation}
For each $n\geq 1$,  $\phi_n$ will denote a GFF\footnote{With either free or Dirichlet boundary condition. We introduce all the relevant definitions in Section \ref{S.Preliminaries}.} on $\Lambda_n$. Recall that for any smooth function $f: [-1,1]^2\mapsto \R$
\begin{equation}\label{E.Convergence GFF}
\frac{1}{n^2} \langle \phi_n,f \rangle\to (\Phi, f) \ \ \ \ \ \ \ \ \text{ in law as } n\to \infty\,,
\end{equation}
where $\Phi$ is a continuous GFF in $[-1,1]^2$, and $(\Phi,f)``:=\int \Phi(x) f(x) dx"$. This tells us that the macroscopic observables related to $\phi_n$, are random variables of the form $n^{-2}\langle\phi_n,f\rangle$.

The main focus of this paper is to understand when we can recover the full macroscopic information of $\phi_n$ by just knowing $\exp(i T \phi_n)$, or equivalently, $\phi_n \pmod{\frac{2\pi}{T}}$. We will give several motivations which lead us to consider this problem later in Section \ref{ss.motiv}. We now state our main result which shows that this statistical reconstruction problem undergoes a phase transition as $T$ varies, which is reminiscent of the \textbf{Berezinskii-Kosterlitz-Thouless transition (BKT transition)} (see Section \ref{ss.BackKT}). 

\begin{theorem}\label{T.0-boundary theorem}
Let $\phi_n$ be a GFF on $\Lambda_n$ with Dirichlet boundary conditions. Then, there exists $0<T_{\rec}^-\leq T_{\rec}^+<\infty$ such that
\begin{enumerate}[(a)]
	\item If $T<T_{\rec}^-$, there exists a (deterministic) reconstruction function $F_T$ such that for any continuous function $f : [-1,1]^2 \to \R$ and any $\eps>0$,
	\[\P\left[|n^{-2}\langle F_T(\exp(iT \phi_n))-\phi_n, f \rangle|\geq \epsilon \right] \to 0, \ \ \  \text{as }n\to \infty. \]	
Furthermore, uniformly in $n$ there exists a constant $C>0$ s.t. for any point $x,y\in \Lambda_n \subset [-1,1]^2$
\begin{align}
&\label{e.1-point}\Eb{(F_{T}(\exp(i T \phi_n))(x) - \phi_n(x))^2} \leq C, \ \ \  \ \text{ and}\\
&\label{e.2-points}\Eb{(F_{T}(\exp(i T \phi_n))(x) - \phi_n(x))(F_{T}(\exp(i T \phi_n))(y) - \phi_n(y)) }\leq e^{-C\|x-y\|n}.
\end{align}	

	\item If $T>T_{\rec}^+$, for any (deterministic) function $F$ and any continuous non-zero function  $f$,
	there exists $\delta>0$ such that
	%\margin{C: Indeed non-zero was important! Should we be more explicit then and say this is $>\delta \Eb{\<{f,(-\Delta)^{-1}f}^{1/2}}$ or so  (with some scaling in $n$ ??) \avelio{It is tricky. Because it is true if $\Delta^{-1}$ is the continuous inverse, but not if it is discrete (there is a $n^{2}$ one has to take care of).} AhOk, indeed, the Peierls donne des sommes de i.i.d et donc on ne peut pas comparer avec ce terme! Notre preuve actuellement ne DONNE pas ca!}
	%\margin{Is it for ANY $\eps>0$ ?: It would be if we had the Laplace transform, I think we only have one $\epsilon$ with our technique}
	\[\liminf_{n\to \infty} \P\left[|n^{-2}\langle F(\exp(iT \phi_n))-\phi_n, f)\rangle |\geq \delta \right]>0. \] 
Also, for any $x\in (-1,1)^2$, there exists $c=c(T,x)>0$ s.t. for any $F$
\begin{align}\label{E.1-point zero boundary}
\liminf_{n\to \infty} \Eb{ (F(\exp(iT \phi_n))(x)-\phi_n(x))^2} \geq c(T,x) \log n  \,.
\end{align}
\end{enumerate}
\end{theorem}

\ni
The same result holds for a free boundary condition GFF.
\begin{theorem}\label{T.free-boundary theorem}
	Let $\phi_n$ be a GFF on $\Lambda_n:= [-1,1]^2 \cap \frac 1 n \Z^2$ with free boundary conditions and rooted  at a vertex $x_0\in \Lambda_n$. Then, there exists $0<T_{\rec}^-\leq T_{\rec}^+<\infty$ such that
	\begin{enumerate}[(a)]
		\item If $T<T_{\rec}^-$, there exists a reconstruction function $F_T$, s.t. for any smooth function $f$ with $0$-mean (i.e., $\int_{[-1,1]^2}f =0$) and any $\epsilon>0$,
		\[\P\left[|n^{-2}(F_T(\exp(iT \phi_n))-\phi_n, f))|\geq \epsilon \right] \to 0, \ \ \  \text{as }n\to \infty. \]
		\item If $T>T_{\rec}^+$, for any function $F$ and any smooth non-zero function $f$ with $0$-mean there exists $\delta>0$ such that 		\[\liminf_{n\to \infty} \P\left[|n^{-2}(F(\exp(iT \phi_n))-\phi_n, f))|\geq \delta \right]>0. \] 
	\end{enumerate}
\end{theorem}
\ni
In fact, in the case of free boundary condition, one also has the equivalent statement of \eqref{e.1-point} and \eqref{E.1-point zero boundary}. However, there is an important difference between both boundary conditions. We do not expect the equivalent of \eqref{e.2-points} to be true for the free case. The main reason is that the conditional law of $\phi_n$ given $e^{iT\phi_n}$ may be decomposed as a convex combination of the laws of $\phi^{k}_n$, for $k\in \Z$, where for each one of the fields $\phi^{k}_n$ one has \eqref{e.2-points}. However, the law of $\phi^{k}_n$ is not centered (see Remark \ref{r.differece_bc}).

\smallskip

We now state two corollaries of the above theorems. The first one rephrases this phase-transition in terms of the continuum GFF. The second one (which will give support to conjectures \ref{c.SLE4} and \ref{c.SLErho}) shows that one can recover macroscopic interfaces from $\phi\Mod{\frac {2\pi} T}$ when $T<T_{rec}^-$.   
\begin{corollary}\label{C.Convergence}
	Let $\phi_n$ be a sequence of GFF in $\Lambda_n$, such that, in probability, $\phi_n\to \Phi$ a continuum GFF in $[-1,1]^2$. Then, if $T<T_{\rec}^-$, the function $F_T(e^{iT\phi_n})$ converges in probability to $\Phi$. Furthermore, if $T>T_{\rec}^+$ there is no deterministic function $F$ such that $F(e^{iT\phi_n})\to \Phi$.
\end{corollary}

\begin{corollary}\label{c.Levelines}
	Let $\phi_n$ be a sequence of GFF in $\Lambda_n$, let $\eta^{(n)}$ be the Schramm-Sheffield level line\footnote{For the definition and the context used in this conjecture see Section \ref{ss.LL}} of $\Phi^{(n)}$. Then, there exists a deterministic function $L_T$, such that the Hausdorff distance between $L_T(\exp(iT\phi_n))$ and $\eta^{(n)}$ is $o(1)$. In particular, $L_T(\exp(iT\phi_n))$ converges in law to an $\SLE_4$. 
\end{corollary}

Our work naturally belongs to the class of \textbf{statistical reconstruction problems} which have been the subject of an intense activity recently. For example it shares similarities  with the statistical reconstruction problems analyzed in \cite{Holroyd,PeresSly,Abbe}. In particular in the later work, {\em Groups synchronization on grids}, the authors analyze the following problem: Imagine that each site $x\in \Z^d$ carries a spin or group element $\theta_x\in \mathfrak{S}$, a compact group (for example $\{\pm 1\}$ or $O(n)$). The question they are interested in is the following one: what macroscopic information on $\{ \theta_x\}_{x\in \Z^d}$ can be recovered from the knowledge of 
\[
\{ \theta_i \theta_j^{-1} + \mathrm{noise} \}_{i\sim j, \text{edges of }\Z^d}
\]
where observations of neighboring spins $\theta_i \theta_j^{-1}$ are subjected to a small {\em noise}. Our setting is very similar in flavour as we also have access to $\phi_n(i)-\phi_n(j)$ when $i\sim j$ except  the noise term is replaced in our case by the modulo operation $\pmod{\frac {2\pi} T}$. Similarly as adding a noise term, applying $\pmod{\frac {2\pi} T}$ is also reducing the information we have on $\phi_n(i)-\phi_n(j)$, except it cannot be analyzed as a convolution effect. The second difference with \cite{Abbe} is that our spins belong to $\R$ instead of a compact group $\mathfrak{S}$.

%Before listing some motivations that guided our work, we state 

\subsection{Fluctuations for integer-valued fields.}

Our present statistical reconstruction problem is intimately related to a generalization of the integer-valued Gaussian free field which plays a key role in the proof of the BKT transition for the Villain and XY models in \cite{FS}. Let us briefly recall the classical integer-valued GFF  before introducing its generalisation.

 For simplicity, in this subsection as well as in Sections \ref{ss.BackKT},\, \ref{s.deloc} and Appendix \ref{a}, we will consider an arbitrary finite subset $\Lambda\subset \Z^2$, instead of the scaled box $\Lambda_n=\frac 1 n \Z^2 \cap [-1,1]^2$. This way, it matches the setup in \cite{FS,RonFS}.

%n some sense 
%As we will explain in more details in Subsection \ref{??}, the statistical reconstruction of the GFF given its values modulo $\frac {2\pi} T$ is intimately related to the fluctuations of a generalization of the \textbf{integer valued GFF model} (IV-GFF).

\begin{definition}\label{d.IVgff}
Let $\Lambda\subset \Z^d$ be a finite domain\footnote{Again, the GFF as well as the graph notations $\p \Lambda$ etc. are defined in Section  \ref{S.Preliminaries}.}. The \textbf{integer-valued} GFF (IV-GFF) on $\Lambda$ with Dirichlet-boundary conditions, i.e. $0$ on $\p \Lambda$, and inverse temperature $\beta$ is the $\beta$-GFF $\{\phi(i)\}_{i\in \Lambda}$ conditioned on the singular event $\{\phi(i) \in \Z, \forall i\in \Lambda\}$.   
%$\{\phi(i) \in 2\pi \Z, \forall i\in \Lambda\}$.  
%($2\pi$ is only a convenient scaling factor here 
Equivalently, it can be defined as the probability measure $\P^{\IV}_{\beta,\Lambda}$ on $\Z^{\Lambda}$ defined as follows:
\begin{align}\label{}
\P_{\beta,\Lambda}^\IV(d\phi):=\frac 1 Z \sum_{\bm\in \Z^\Lambda : \bm_{\md \p \Lambda} =0} \delta_{\bm}(d \phi) 
\exp\left(-\frac \beta 2 \<{\nabla \phi, \nabla \phi} \right)
\end{align}
\ni
or also, to avoid any possible confusion, for any $\bm\in \Z^\Lambda$ with 0 boundary conditions, we have that $\P_{\beta,\Lambda}^\IV(\phi=\bm) = \frac 1 Z \exp(-\frac \beta 2 \<{\nabla \bm, \nabla \bm})$.

The IV-GFF with free boundary-conditions is defined in the same manner except we replace $\bm_{\md \p \Lambda}\equiv 0$ by $\bm(x_0)=0$ for any choice of root vertex $x_0\in \Lambda$.
\end{definition}

%\margin{Notations $Z$, and definition of $\p \Lambda$. Probably the best is to follow Ron's notations.}
%\[
%\FK{\beta,\Lambda}{\mathrm{IV}}{\phi}:=\frac 1 {Z} \exp(-\frac \beta 2 \sum_{i\sim j} (\phi_i -\phi_j)^2) 1_{\phi_i \in Z} 1_{\phi_i=0, \forall i\in \p \Lambda}
%\]

This integer-valued undergoes a roughening-phase transition as $T$ increases (i.e. as $\beta$ decreases) as it was proved by Fröhlich-Spencer in \cite{FS} (see also the very useful survey \cite{RonFS}). Fröhlich-Spencer proved this striking phase transition for periodic and free boundary conditions on large square boxes $\Lambda$ and explained in \cite[Appendix D.]{FS} how to adapt their proof to the case of Dirichlet boundary conditions. Very recently, Wirth has written carefully in \cite[Appendix A]{Wirth} the details of this extension to Dirichlet boundary conditions. We will come back to it later in Section \ref{ss.BackKT}.  (See also Figure \ref{f.IVgff} for an illustration of the IV-GFF in $d=1$).
\begin{theorem}[Fröhlich-Spencer \cite{FS}]\label{th.FS}
There exists $0<\beta^+ \leq \beta^-<\infty$\footnote{This choice $\beta^+\leq \beta^-$ is made to highlight that these are inverse temperatures related to $T_{rec}^-\leq T_{rec}^+$.} such that for any square $\Lambda\subset \Z^2$, if we consider the IV-GFF with free boundary conditions rooted at $x_0\in \Lambda$ then we have the following dichotomy:
\bi 
\item {\bf Delocalised regime (rough regime).} If $\beta<\beta^+$, then  for any $f : \Lambda \to \R, \sum_{i\in \Lambda} f(i)=0$, 
\begin{align*}\label{}
\EFK{\beta,\Lambda}{\IV}{e^{\<{\phi, f}}} \geq e^{\frac 1 {4\beta} \<{ f, (-\Delta)^{-1} f} }
\end{align*}
(N.B. it is not hard to extract from this Laplace transform estimate, fluctuations bounds such as $\EFK{\beta,\{-n,\ldots,n\}^2,x_0=0}{\IV}{\phi(x)^2} \geq \frac{c}{\beta} \log \|x\|_2$ for any $x\in \{-n,\ldots,n\}^2$, see for example \cite{RonFS}). 
\item {\bf Localised regime.} If $\beta>\beta^-$, then for any $x\in \Lambda$, 
\begin{align*}\label{}
\EFK{\beta,\Lambda,x_0}{\IV}{\phi(x)^2} \leq \frac C \beta\,.
\end{align*}
\ei
\end{theorem}

The relationship between our statistical reconstruction problem and integer-valued fields is due to the following explicit structure of the conditional law of a GFF given its values modulo $\frac {2\pi} {T}$. We stick for simplicity to the case of Dirichlet boundary conditions. Let us fix $\ba\in [0,1)^{\Lambda}$ satisfying $\ba_{\md \p \Lambda}\equiv 0$. We will see in Lemma \ref{L.aIV} that the conditional law of the GFF $\phi$ on $\Lambda$ given 
$\phi \Mod{\frac{2\pi} T} = \frac {2\pi} T \ba $
%\margin{Do you prefere $\phi(i)\Mod{\frac{2\pi} T}  = \frac {2\pi} T a_i\,, \forall i \in \Lambda\}$  ?} 
is a multiple of the following \textbf{generalized integer-valued GFF} with $\beta=\beta_T:= (2\pi)^2 \,T^{-2}$. (See Lemma \ref{L.aIV} for a precise statement). 

\begin{figure}[h]
	\begin{center}
		\includegraphics[width=0.32\textwidth]{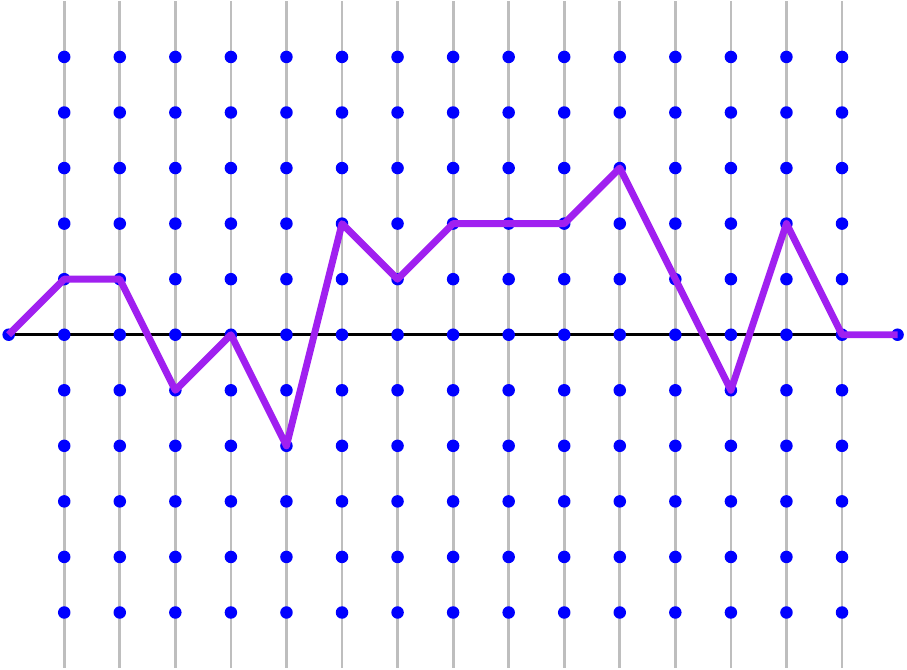}
		\includegraphics[width=0.32\textwidth]{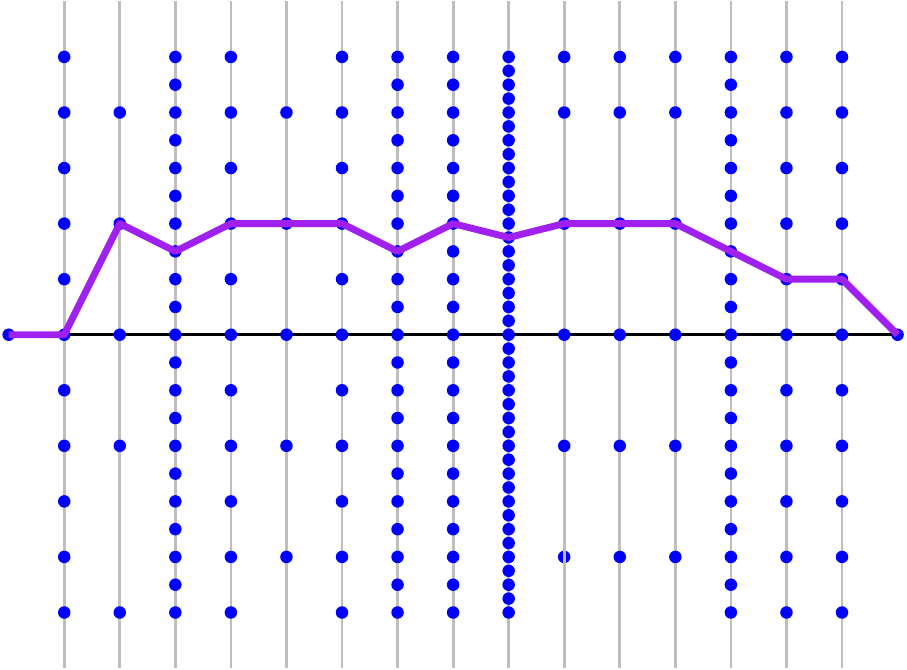}
		\includegraphics[width=0.32\textwidth]{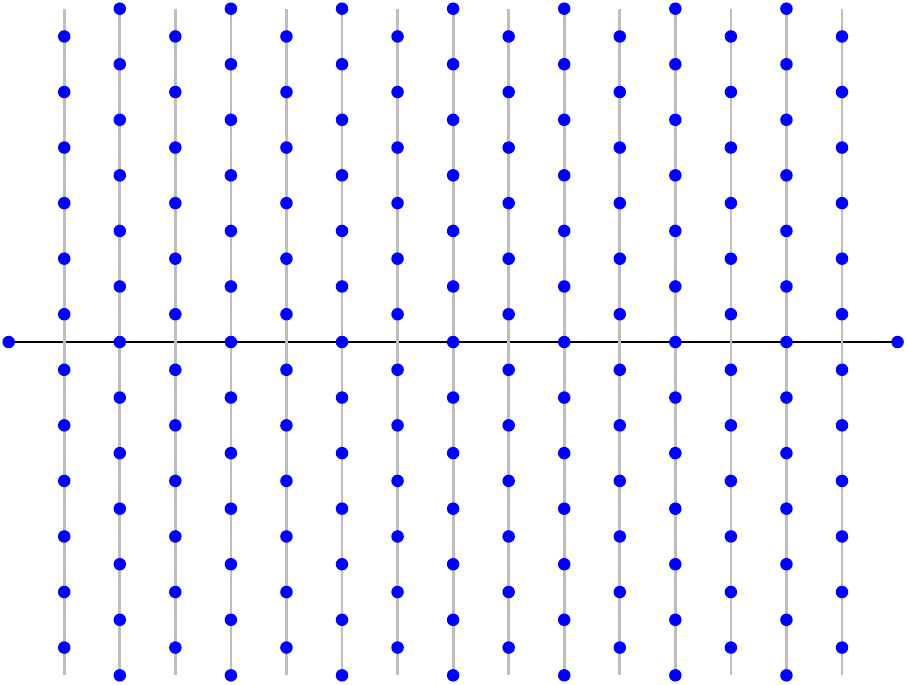}
	\end{center}
	\caption{%The integer-valued GFF (IV-GFF) is the Gaussian Free Field conditioned to take its values in $\Z$. 
		The picture on the left represents an instance of an IV-GFF in $d=1$ on a unit interval $\{0,\ldots,n\}$ with Dirichlet Boundary conditions. In $d=2$, the proof of the \textbf{roughening phase transition} for the IV-GFF in \cite{FS} (and \cite{Wirth} for the extension to Dirichlet b.c.) easily generalizes to certain shifts and scaled versions on the vertical fibers $\Z$ as illustrated in the other pictures. In the middle one, some {\em fibers} are $\Z$, some others are $2\Z$ and one may also add fibers with arbitrarily fine meshes $2^{-k}\Z$ along the interval. On the right, some fibers are $\Z$ while some others are $\frac 1 2 +\Z$. It is easy to check that in $d=2$, any of these can be handled with the techniques from \cite{FS}.}\label{f.IVgff}
\end{figure}
% easily expressed in terms of the following $\ba$-dependent probability measures:

\begin{definition}\label{d.IVba}
Let $\beta>0$ and $\ba=\{a_i\}_{i\in \Lambda}$ s.t. $\ba_{\md \p \Lambda}\equiv 0$ be any collection of real-valued numbers. We define the \textbf{$\ba$-IV-GFF} on $\Lambda$ to be the GFF $\{\phi(i)\}_{i\in \Lambda}$ (with Dirichlet b.c.s) conditioned to take its values in the shifted-fibers $\{ a_i+\Z\}_{i\in \Lambda}$ for any $i\in \Lambda$. It corresponds to the following discrete probability measure on fields:
\begin{align*}\label{}
\FK{\beta,\Lambda}{\ba,\mathrm{IV}}{d\phi}:=\frac 1 {Z}
\sum_{\bm \in \Z^\Lambda,  \bm_{\md \p \Lambda}\equiv 0} \delta_{\bm+\ba}(d\phi) 
\exp(- \frac \beta 2 %\frac {(2\pi)^2} {2 T^2} 
\<{\nabla(\phi), \nabla(\phi)})\,.
\end{align*}
\ni
Equivalently, for any $\bm \in \Z^\Lambda$ with  $\bm_{\md \p \Lambda}\equiv 0$, 
\begin{align}\label{E.P a,IV}
\FK{\beta,\Lambda}{\ba,\mathrm{IV}}{\phi=\bm+\ba} = \frac 1 Z
\exp(- \frac \beta 2 %\frac {(2\pi)^2} {2 T^2} 
\<{\nabla(\bm + \ba), \nabla(\bm+\ba)})\,.
\end{align}
\ni
Notice  that if $\ba \in \Z^{\Lambda}$, then the $\ba$-IV-GFF is nothing but the standard IV-GFF. 
See Figures \ref{f.IVgff},\ref{f.IVgff2}. Finally, this definition extends readily to the case of free boundary conditions in which case $\ba \in \R^\Lambda$ with $\ba_{x_0}=0$. 
\end{definition}

The proof of Fröhlich-Spencer (\cite{FS}) readily extends to some specific choices of the shift $\ba$ which are sufficiently symmetric. (i.e. any $\ba \in \{0,\frac 1 2\}^\Lambda$). See Figure \ref{f.IVgff}  for an illustration (in $d=1$ only) of the cases which can be analyzed using the techniques from \cite{FS} and Figure \ref{f.IVgff2} for the cases which need further analysis. See also Remark \ref{R.Wirth}.
Our main result on such integer-valued fields is the following extension of the above theorem of Fröhlich and Spencer \cite{FS}. 

%\begin{theorem}\label{th.IVgff}
%Recall $\Lambda_n:=[-1,1]^2\cap \frac 1 n \Z^2$. 
%There exists $\beta_c^{IV}>0$ and a constant $C>0$, s.t. for any $\beta<\beta_c^{IV}$ any $n\geq 1$ and \textbf{any} $\ba \in \R^{\Z^2}$, then  if $\phi^{\ba}_n$ denotes a 
%$\ba$-IV-GFF with Dirichlet boundary conditions on $\Lambda_n$, we have
%\margin{Maybe I should put an $\eps>0$ to show the near-sharpness.}
%\bi
%\item For any function $f$ \margin{Again regularity ? continuous say ?} 
%\[
%\Var{\<{\phi^{\ba}_n, f}} \geq \frac C \beta \<{f, (-\Delta)^{-1} f}
%\]
%\item The variance of the field $\phi^\ba_n$ at the origin satisfies 
%\[
%\Var{\phi^{\ba}_n(0,0)} \geq \frac {C} {\beta} \frac 1 {2\pi} \log n
%\]
%\ei
%
%The low-temperature regime ($\beta \gg 1$) happens to be less universal in the choice of the shift $\ba$. Indeed, the following different scenarios may happen (by tuning suitably $\ba$):
%\bnum
%\item $\Eb{\phi^\ba_n(0,0)} \geq 0.49 n$
%\item $\Var{\phi^\ba_n(0,0)} \geq (0.49)^2 n$
%\item $\Var{\phi^\ba_n(0,0)} \leq O(1)$ and $\Cov{\phi_n^\ba(x),\phi_n^\ba(y)}\leq e^{-c n \|x-y\|_2}$.  
%\enum
%\end{theorem}

\begin{theorem}\label{th.IVgff}
There exists $\beta_c^{IV}>0$ and a constant $C>0$, s.t. for any square domain $\Lambda \subset \Z^2$, any  $\beta<\beta_c^{IV}$, then uniformly\footnote{with  $\ba_{\md \p \Lambda }\equiv 0$ if Dirichlet b.c. and $\ba_{x_0}=0$ for free b.c.} in 
$\ba \in \R^{\Lambda}$,  if $\phi^{\ba} \sim \P_{\beta,\Lambda}^{\ba,\IV}$ (with either Dirichlet or free b.c.s) we have
%\margin{Maybe I should put an $\eps>0$ to show the near-sharpness.}
\bi
\item For any function $f \in \R^\Lambda$ 
\[
\Var{\<{\phi^{\ba}, f}} \geq \frac C \beta \<{f, (-\Delta)^{-1} f}\,,
\]
where the inverse of the Laplacian is taken here according to the b.c. 
\item If $\Lambda=\{-n,\ldots,n\}^2$ with Dirichlet boundary conditions, the variance of the field $\phi^\ba$ at the origin satisfies 
\[
\Var{\phi^{\ba}(0,0)} \geq \frac {C} {\beta}  \log n\,.
\]
\ni
The analogous statement also holds for free b.c.
\ei
\end{theorem}

\begin{remark}\label{r.LowTemp}
We wish to stress that the low-temperature regime ($\beta \gg 1$) happens to be much less universal in the choice of the shift $\ba$. Indeed, when $\Lambda=\{-n,\ldots,n\}^2$ equipped with Dirichlet boundary conditions, and if $\phi^\ba \sim \P_{\beta,\Lambda}^{\ba,\IV}$ then we expect that the following different scenarios may happen (by tuning suitably $\ba$ in each case) as $n\to\infty$. (See for example Figure \ref{f.GS} for the scenario $(1)$). 
\bnum
\item $\Eb{\phi^\ba(0,0)} \geq 0.49 n$
\item $\Var{\phi^\ba(0,0)} \geq (0.49)^2 n$
\item $\Var{\phi^\ba(0,0)} \leq O(1)$ and $\Cov{\phi^\ba(x),\phi^\ba(y)}\leq e^{-c n \|x-y\|_2}$.  
\enum
\end{remark}

\begin{remark}\label{}
We do not obtain a lower bound on the Laplace-transform of $\P_\beta^{\ba,\IV}$ only on its $L^2$ behaviour which is sufficient to detect localisation v.s. delocalisation. This is also the case in the recent works mentioned below on localisation/delocalisation of integer-valued random surfaces. 
\end{remark}

\begin{figure}
	\begin{center}
		\includegraphics[width=0.5\textwidth]{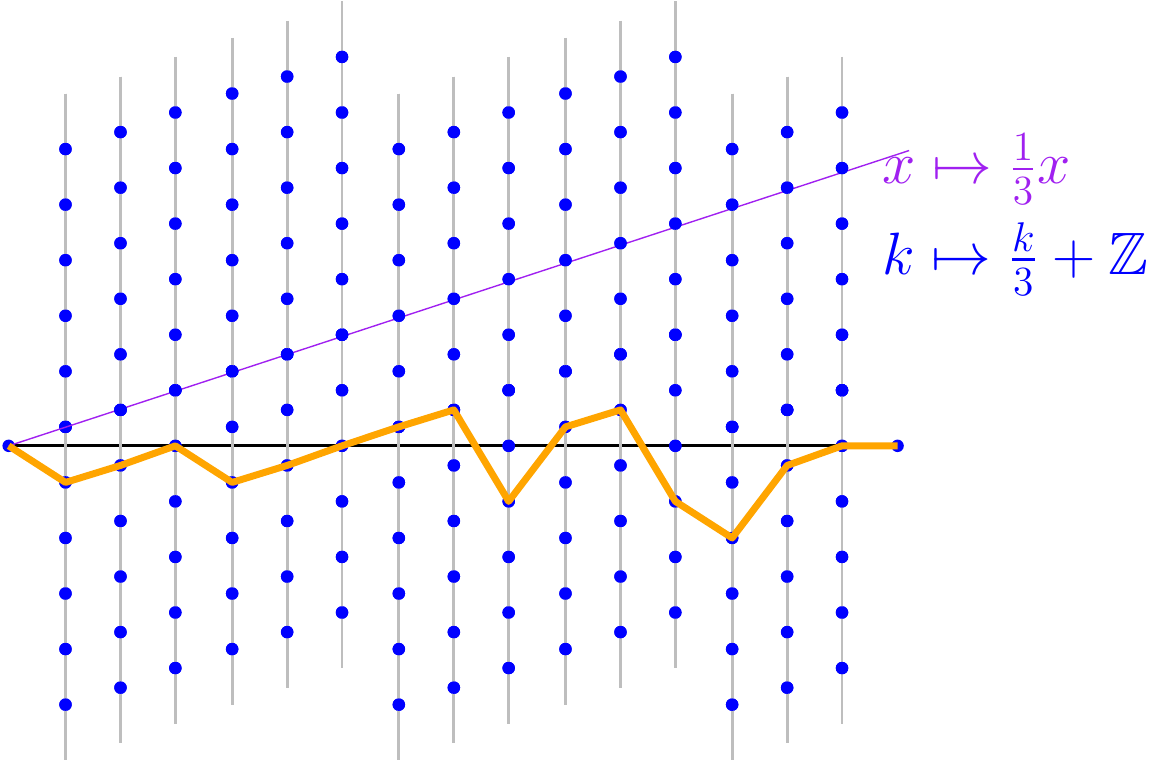}
		\includegraphics[width=0.45\textwidth]{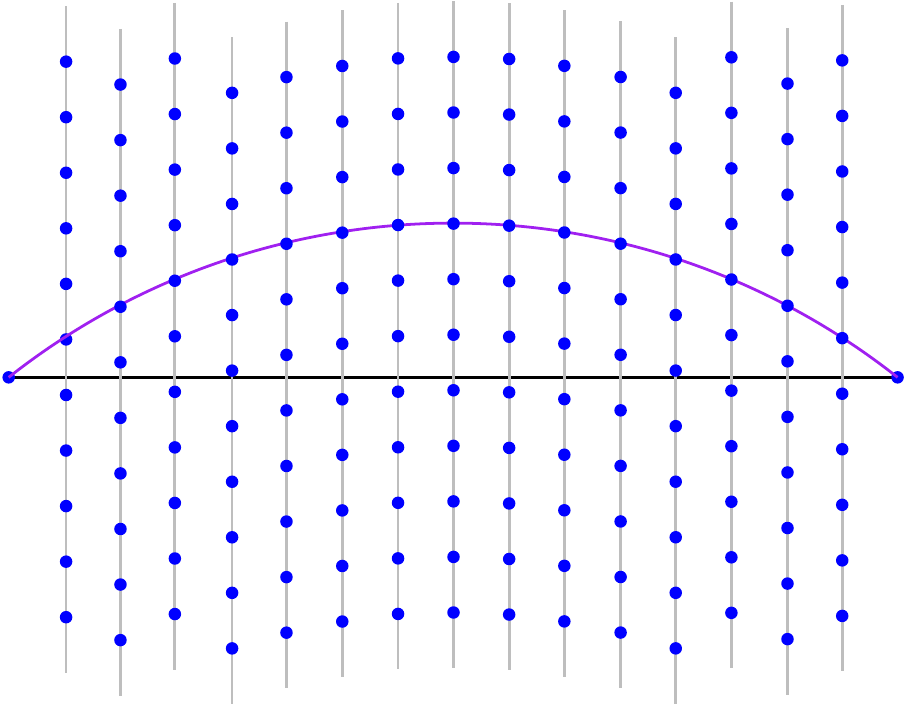}
	\end{center}
	\caption{As opposed to the examples in Figure \ref{f.IVgff}, these choices of fibers do not satisfy the sufficient symmetries to be readily analyzed by the techniques from \cite{FS}. The simplest such example is the picture on the left where $\Z$ fibers are shifted by $\{0,\frac 1 3, \frac 2 3\}$ (here $\ba:=(\frac k 3)_{k\in \Z}$). In such a case, $\sin(k\phi)$ functions appear in the Fourier transform of the periodic distribution $\sum_{i \in \frac 1 3 +\Z} \delta_i$ and this breaks the parity. For such linear shifts, Wirth obtained in \cite{Wirth} some related bounds using a nice symmetrization trick (however, even in this linear case, these bounds give a different control from the one we need here, see Remark \ref{R.Wirth}). A key property used in \cite{Wirth} is the fact that $x\mapsto a_x$ needs to be \textbf{harmonic} in $\Lambda\setminus \p \Lambda$. Otherwise the symmetrization technique breaks and one cannot rely anymore on Jensen's inequality, a key step in the proof \cite{FS}. For example the picture on the right where fibers are shifted by a quadratic curve requires an additional analysis w.r.t \cite{FS}.}\label{f.IVgff2}
\end{figure}

As we will see in Section \ref{s.deloc} and particularly in Appendix \ref{a}, 
our proof of Theorem \ref{th.IVgff} involves an exact identity (Proposition \eqref{c.magical}) which is closely related to the {\em modular invariance} identity for Riemann-theta functions. (N.B See also \cite{agostini2019} for another use of such identities in probability.) We briefly mention this connection here as it is interesting in its own and it allows us to rephrase Fröhlich-Spencer Theorem as well as our Theorem \ref{th.IVgff} easily in terms of those Riemann-theta functions. 

Indeed, the following function of $\ba\in \R^{\Lambda\setminus \p \Lambda}$:
\begin{align*}\label{}
\tilde \theta_\Lambda(\ba) := \sum_{\bm\in \Z^{\Lambda\setminus \p \Lambda}}  \exp\left (-\frac \beta 2 \<{\bm, (-\Delta) \bm}\right ) \exp(-\beta \, \bm \cdot a)
\end{align*}
can be easily written in terms of the classical Riemann-theta function $\theta(\bz \md \Omega)$ (see~\eqref{e.RT}).
Furthermore, one can check that for any $f: \Lambda \to \R$ and if $\phi^\ba \sim \P_{\beta,\Lambda}^{\ba,\IV}$, we have
\begin{align*}\label{}
\Var{\<{\phi^\ba, f}} = \left[\sigma \cdot \nabla_\ba \, \sigma \, \cdot \nabla_\ba \right] \log \tilde\theta_\Lambda\,,
\end{align*}
where $\sigma:= \frac 1 \beta (-\Delta)^{-1} f$. This expression clarifies the effect of the shift-vector $\ba \in \R^{\Lambda}$ and reveals that it plays the role of an exterior magnetic field.
% and $\log \tilde \theta$ plays the role of a free energy. 
We may now rephrase Fröhlich-Spencer as well as our main result from this Section as follows:
\bi
\item {\em (Theorem \ref{th.FS}).} If $\beta$ is small enough, then uniformly in $\Lambda=\{-n,\ldots, n\}^2$, 
\[
\left[\sigma \cdot \nabla_\ba \, \sigma \, \cdot \nabla_\ba \right]_{\ba\equiv 0} \log\, \tilde\theta_\Lambda \geq \frac C \beta \<{f, -\Delta^{-1} f} 
\]
\item {\em (Theorem \ref{th.IVgff}).} If $\beta$ is small enough, then uniformly in $\Lambda=\{-n,\ldots, n\}^2$,
\[
\inf_{\ba \in \R^{\Lambda\setminus \p \Lambda}} \sigma \cdot \nabla_\ba \, \sigma \, \cdot \nabla_\ba \; \left( \log\, \tilde\theta_\Lambda \right) \geq \frac C \beta \<{f, -\Delta^{-1} f} 
\]
\ei

Finally, let us point out that over the last few years, there have been several important works which analyzed the roughening phase transition (i.e. localisation/delocalisation) for other natural models of integer-valued random fields, such as the {\em square-ice} model, uniform Lipschitz functions $\Z^2 \to \Z$ etc: see in particular the recent works \cite{HugoRon,RonScott, Mano,HugoLog}. These works do not rely on the Coulomb-gas techniques from \cite{FS} but rather on geometric techniques such as RSW. 
% Choix different de constantes pour \tilde \theta, mais finalement redigé ainsi ca va. A CONSTANTE pres......

\subsection{Motivations behind this statistical reconstruction problem.}\label{ss.motiv}

$ $

As we will see below, one of the main reasons which lead us to consider this statistical reconstruction problem on the GFF has to do with the statistical analysis of the XY and Villain models in $d=2$. Each of these are celebrated models with continuous $O(2)$-symmetry. We briefly define what they are and we refer the reader to \cite{FS,  Roland, RonFS,friedli2017statistical} for useful background on these models.  

\begin{definition}[Villain and XY models]\label{}
Let us fix a finite graph $\Lambda \subset \Z^2$ and $\beta>0$ to be the inverse temperature. Both models are Gibbs measures on the state-space $ (\S^1)^\Lambda$. Let us parametrise this spin-space via its canonical identification with $[0,2\pi)^\Lambda$. 
\bi
\item {\em XY model (or plane rotator model)} 
\begin{align}\label{g.1}
d\FK{\beta}{\mathrm{XY}}{\{\theta_x\}_{x\in \Lambda}}
\propto  \prod_{i \sim j }\exp\left( \beta \cos(\theta_i - \theta_j)\right) \, \prod d\theta_i.
\end{align}
\item {\em Villain model}
\[
d\FK{\beta}{\mathrm{Villain}}{\{\theta_x\}_{x\in \Lambda}}
\propto \prod_{i\sim j} \sum_{m\in \Z}  \exp\left( - \frac \beta 2 (2\pi m + \theta_i - \theta_j)^2 \right)\, \prod d\theta_i .
\]
\ei
\end{definition}

\smallskip 

\ni
We may now list what are the main motivations which guided our work. 

\bnum
\item {\em Extracting macroscopic random structures from XY and Villain spins.} 
%In many ways, the success of the $\SLE_\kappa$ literature has been to associate macroscopic random interfaces to critical systems such as Ising, percolation etc. and to gain from the properties of these curves  

For spins systems such as the Ising model, Potts models or also percolation which all have discrete symmetries, it is clear how to associate natural macroscopic fluctuating objects such as interfaces which may then converge to suitable $\SLE_\kappa$ as the mesh goes to zero. On the other hand, for spin systems with continuous symmetry such as XY or Villain models, given a realization of the Gibbs measure, say $\{ \sigma_x\}_{x\in \Lambda}$ with $\sigma_x \in \S^1$, it is much less clear what macroscopic objects one may assign to $\{\sigma_x\}$. 
   
   One consequence of our present statistical reconstruction problem is that it gives strong evidence to the fact that it is possible to extract a macroscopic GFF $\phi_n$ from the observation of the spins $\{\sigma_x\}_{x\in \{-n,\ldots, n\}^2}$ (up to small microscopic errors).
   
   Indeed, at least in the case of the Villain model, it has been conjectured by Fröhlich-Spencer in \cite[Section 8.1]{FS1983} that at low temperature ($\beta\gg1 $), then up to ``microscopic errors'', one should have 
\begin{align*}\label{}
\{\sigma_x\}_{x\in \{-n,\ldots, n\}^2} \sim \P_\beta^\mathrm{Villain} \overset{\mathrm{law}}\approx 
\{ \exp\big(i\, \frac 1 {\sqrt{\beta'}} \phi_n(x) \big) \}_{x\in  \{-n,\ldots, n\}^2}\,,
\end{align*}
where $\beta'=\beta(\beta)$ satisfies $|\beta'-\beta|\leq e^{-C\beta}$ and where $\phi_n$ is a GFF on $\{-n,\ldots, n\}^2$ with either free or 0 b.c. 
%We refer to \cite{FS,RonFS} for a definition of the Villain model (whose Gibbs measure tends to the $\beta$-XY model as $\beta \to \infty$). 

Once one realises that one may extract a GFF out of the spins $\{\sigma_x\}_{x\in \Lambda_n}$, it is then natural to extract level lines and flow lines from this GFF studied in \cite{She05,Dub,MS1}. Corollary \ref{c.Levelines} is a proof of this concept. We discuss this further in Section \ref{ss.LL} where we highlight how our work lead us to conjecture that when $\beta$ is high enough, then the natural interface for the Villain model pictured in Figure \ref{f.XY} should converge to an $\SLE(4,\rho)$ process. 

\begin{figure}
	\begin{center}
		\includegraphics[width=0.7\textwidth]{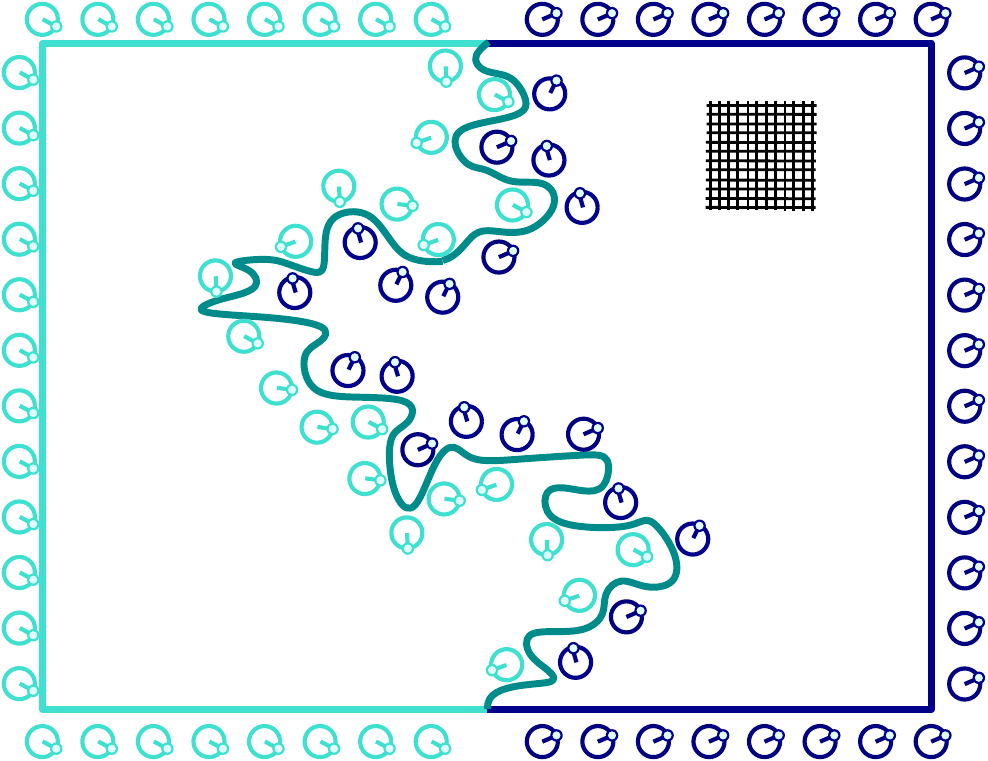}
	\end{center}
	\caption{Conjectures \ref{c.SLE4} and \ref{c.SLErho} in Section \ref{ss.LL} predict that at low temperature, the level lines of a Villain model with $e^{i/10}$ on the right boundary and $e^{-i/10}$ on the left boundary should converge to $\SLE(\kappa=4, \rho)$ processes. Furthermore, the set of all interfaces should converge to the so-called ALE (see Conjecture \ref{c.ALE}). These conjectures are supported by the present statistical reconstruction problem as well as by the techniques we have used. 
	}\label{f.XY}
\end{figure}

%Our main motivation comes from the following conjecture which seems to  The main motivation is the hope to extract macroscopic information from collections of spins in $\S^1$ ....La conjecture de Frohlich-Spencer
%\textbf{ICI notre conjectre sur le $\SLE_4$.}

\smallskip

\item {\em A different interpretation of the KT transition.} 

The classical way of understanding the KT transition for spins systems such as the XY model is to notice that vortices ($\equiv$ discrete 2-forms) come into the energy-balance when analyzing the Gibbs measure~\eqref{g.1}. 
This present work gives the following different interpretation of the role of the $\S^1$-geometry within the BKT transition which does not explicitly involve vortices. When the temperature $T$ is low, spins wiggle slowly around $\S^1$ and one should be able to recover a macroscopic GFF as we have seen in the above item (1). If instead, the temperature $T$ is large, the spins start wiggling too quickly around $\S^1$ so that one cannot extract the whole macroscopic fluctuating Gaussian field $\phi$ which leaves on the top of the spin field. 

%A different model for the KT transition, if not a different understanding / interpretation. Furthermore a Laboratory. 
%For example, see Section 5. 

\smallskip

\item {\em An integrable model for Integer-valued GFF}. 

The main tool we use for the regime $T>T_{rec}^+$ is the proof of delocalisation for the generalized IV-GFF $\phi^\ba \sim \P_\Lambda^{\ba,\IV}$ from Definition \ref{d.IVba}.  We think of $\ba$ as the random vector $\{a_i\}_{i\in \Lambda}\in [0,1)^\Lambda$ defined by  
\[
a_i := \frac T {2\pi}  \phi_i \Mod{1} \,, \,\,\, \forall i \in \Lambda,\,
\]
where $\phi$ is a GFF in $\Lambda$. As such, we may view the random measure $\P_\Lambda^{\ba,\IV}$ as a \textbf{quenched measure} on (shifted) integer-valued fields. Interestingly these highly non-trivial quenched measures have (by construction) a very simple \textbf{annealed measure}. Indeed Lemma \ref{L.aIV} readily implies that
%\purple{for any test function $F$}
\begin{align}\label{}
\int \P_{\beta_T,\Lambda}^{\IV,\ba}[d\phi] \;\P_T(d\ba) =\P_{\beta_T}^\GFF[d\phi]\,, 
\end{align}
%\purple{
%\begin{align}\label{}
%\int \E_{\beta_T,\Lambda}^{\IV,\ba}[F(\phi)] \;\P_T(d\ba) =\E^\GFF[F(\frac T {2\pi} \phi)]\,, 
%\end{align}
%}
where we denoted by $\P_T(d \ba)$ the law of the above random shift $\ba$ and $\beta_T=\frac{(2\pi)^2}{T^2}$. If one now assumes that some properties (such as fluctuations) are not very sensitive to $\ba$, this identity gives a ``useful laboratory'' to analyze the classical integer-valued GFF (i.e. $\ba\equiv 0$).  A first illustration of this is given in Section \ref{S.Information left} where we provide a new insight on the $\eps=\eps(\beta)$ correction in the bound of Fröhlich-Spencer. 
A second illustration is given in the item below.

\smallskip

\item {\em Random-phase Sine-Gordon model.}

As it was pointed out to us by Tom Spencer, our work is closely related to the \textbf{random-phase Sine-Gordon} model. This is a model of random interface with quenched disorder which has been studied extensively in physics and which is conjectured to exhibit a striking \textbf{super-roughening} behavior at low temperature. (See for example \cite{cardy1982,doussal2007}.) As we shall explain in Appendix \ref{a.SG}, our proof of Theorem \ref{th.IVgff} easily extends to the setting of the random-phase Sine-Gordon model and allows us to prove $\log n$ fluctuations in the high-temperature phase of this model. See Theorem \ref{th.SG}. 

\smallskip

\item {\em Imaginary multiplicative chaos.} 
In this work we focus on lattice fields $\phi : \Lambda \to \R$ or $\Lambda_n \to \R$, but the question in the continuum is also interesting. Namely, given a Gaussian free field $\Phi$ on $[-1,1]^2$ with 0-boundary conditions, can one recover $\Phi$ from $\Wick{e^{i \alpha \Phi}}$? This is the complex analog of the reconstruction procedure $\Wick{e^{\gamma \Phi}} \mapsto \Phi$ studied in \cite{berestycki2014equivalence}. We discuss this further in Section \ref{ss.Trec-}, where we show that the existence of a continuous reconstruction process in the imaginary case implies the existence of a discrete reconstruction process. However, let us highlight that even if this continuum process does exist, the discrete reconstruction process coming from it will converge much slower ($o(1)$) than the one we obtain in Theorem \ref{T.0-boundary theorem} using statistical mechanics (i.e. $O(1/n^2)$, see Proposition \ref{P.Var O(n^2)}). 
Also, as opposed to the discrete setting, there is no regime where the statistical reconstruction breaks in the continuum. {\em (N.B. After the first version of this work, the reconstruction problem in the continuum has been solved in the beautiful recent work \cite{Juhan} using completely different tools).}
\enum

\subsection{Idea of the proof.}
$ $

The first choice one needs to make in the proof is the reconstruction function $F_T$. We have essentially two natural choices here (see Figure \ref{f.GS} for an illustration of both). 
\bnum
\item First, if $\ba:= \frac {T} {2\pi}\left( \phi\Mod{\frac {2\pi} T}\right)$, then there is an a.s.  unique \textbf{ground-state\footnote{From the point of view of the Statistical reconstruction, ground-state should be read as the maximum-likelihood estimator.}} for $\P_{\beta}^{\ba,\IV}$ which we may call 
\begin{align*}\label{}
\hat\bm(\exp(i T \phi)) & = \hat \bm(\exp(2i \pi \ba)) \\
& := \mathrm{argmin}_{\bm \in \Z^\Lambda} \exp\left( - \frac \beta 2 \<{\bm+\ba, -\Delta(\bm+\ba)} \right)\,.
\end{align*}
It is reasonable to guess that when $T$ is small, the field $\phi$ should not fluctuate much around $\hat \bm(\exp(i T \phi))$. 
\item A second natural choice is to consider instead the conditional expectation of the field given $\exp(iT\phi)$.
\enum
\begin{figure}
	\begin{center}
		\includegraphics[width=0.9\textwidth]{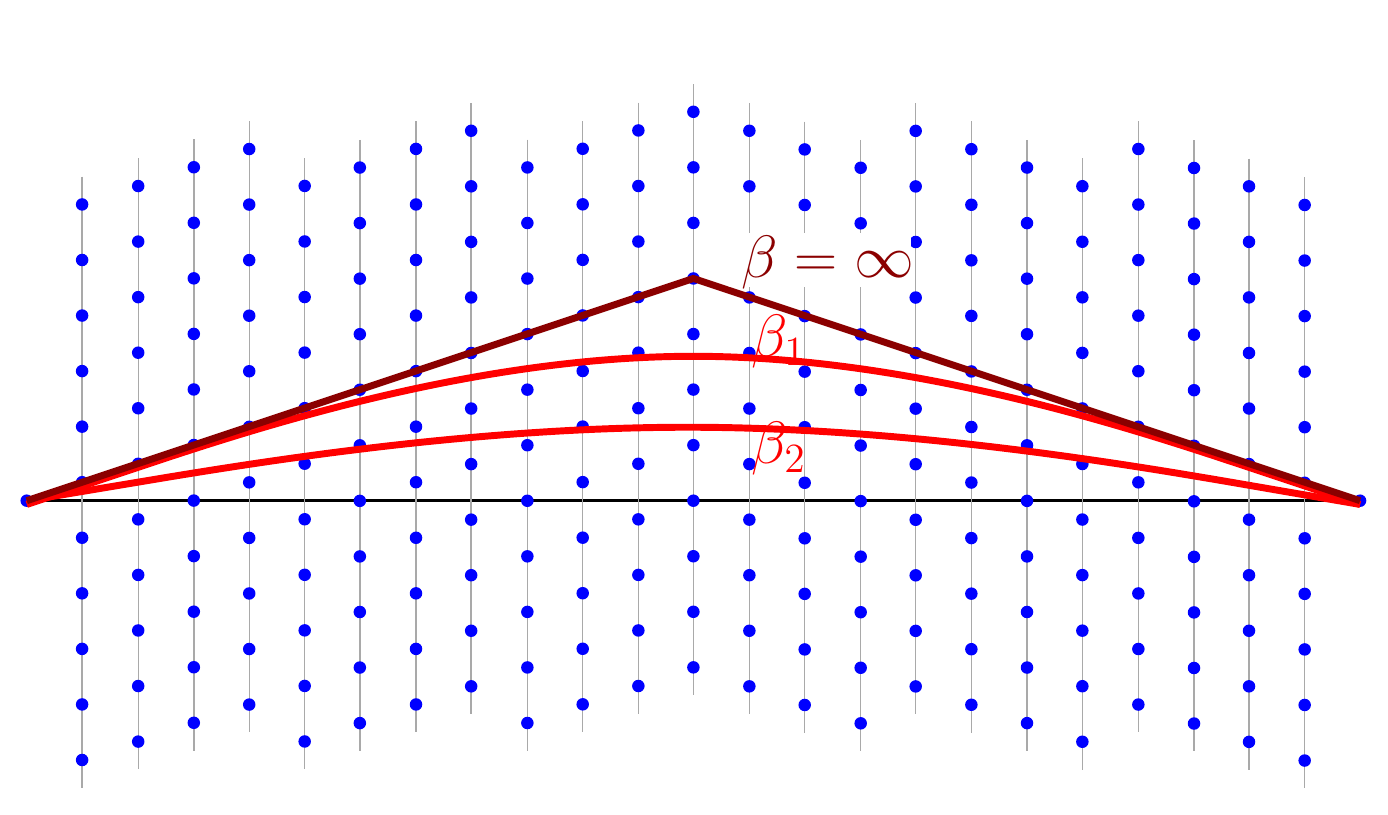}
	\end{center}
	\caption{Here $a_k:=\frac k 3$ on the L.H.S of the picture and then the slope is $-k/3$. It is easy to check that the ground-state $\hat \bm$ for the GFF conditioned to have its values in these shifted fibers is given by the purple curve $\beta=\infty$. Then, as $\beta$ decreases, the expectation $x\mapsto \EFK{\beta}{\ba,\IV}{\phi^\ba(x)}$ can be seen to decrease to $0$.}\label{f.GS}
\end{figure}

The {\em quenched} groundstate $\hat\bm(\exp(i T \phi))$ does not have enough symmetries to apply classical tools from Peierls theory and we are too far from the perturbative regime where \textbf{Pirogov-Sinai theory} can be used (see \cite[Chapter 7]{friedli2017statistical}). %\margin{check??}
\smallskip

%Therefore, for the \textbf{low-temperature regime} in the proofs of Theorem \ref{T.0-boundary theorem} and \ref{T.free-boundary theorem}, we will see that one way to recover the GFF given its phase is to take the function second choice, i.e,
Therefore, for the \textbf{low-temperature regime} in the proofs of Theorem \ref{T.0-boundary theorem} and \ref{T.free-boundary theorem}, we will recover the GFF given its phase via the  second choice, i.e,
\begin{align*}
F_T(\exp(iT\Phi))(x):=\E\left[\phi(x)\mid \exp(iT\phi) \right].
\end{align*}
It is not so easy to study this function $F$ directly. However, for any test function $f$ we can use Markov's inequality to see that
\begin{align}\label{E.Markov ineq}
\P\left( |\langle \phi-F(\exp(iT\Phi)), f \rangle|\geq \epsilon \right)\leq \frac{\E\left[ \Var{ \langle \phi, f \rangle\mid \exp(iT\phi)\ }\right] }{\epsilon ^2}.
\end{align}
This implies that to understand how well $F$ approximates $\phi$ it is enough to bound the conditional variance of $\langle \phi, f \rangle$ given $\exp(iT\phi)$. Working with the conditional variance is much easier than to work with $F$ directly. This is because one can study it by \textbf{coupling} two GFFs $(\phi_1,\phi_2)$ such that $\exp(iT \phi_1)=\exp(iT \phi_2)$ in such a way that they are conditionally independent given $\exp(iT\phi)$ (see Definition \ref{D.phis}). This is useful because
\begin{align}\label{E.Variance using phis}
\E\left[ \Var{ \langle \phi, f \rangle\mid \exp(iT\phi)}\right ]= \frac{1}{2}\E\left[\<{\phi_1 -\phi_2,f}^2 \right]. 
\end{align}
As this function does not involve any estimate of the function $F$, and both $\phi_1$ and $\phi_2$ have the law of a GFF, we set up an appropriate {\em annealed Peierl's argument} to show, in Section \ref{S.Localisation} that \eqref{E.Variance using phis} is small when $T$ is small. 
\medskip

The second part of Theorems \ref{T.0-boundary theorem} and \ref{T.free-boundary theorem}, also follows from similar ideas with a ``statistical flavour''. In fact, we are going to show that for any $f$ there exists an $\epsilon>0$ such that for all $n$ big enough 
%\margin{\avelio{I believe it works for all $n$, can you confirm?}}
\begin{align}\label{E.Variance positive}
\E\left[ \Var{ \langle \phi, f \rangle\mid \exp(iT\phi)}\right ]\geq \epsilon\Var{\langle \phi, f \rangle}.
\end{align}
This, together with some basic tension argument, implies that the probability that $\langle \phi_1 , f\rangle$ is macroscopically different from $\langle \phi_2, f \rangle$ is uniformly positive.

To obtain equation \eqref{E.Variance positive}, we need to modify the work of Fröhlich and Spencer \cite{FS}. In this seminal paper, the authors showed that the integer-valued GFF has variance similar to that of the GFF when the temperature is high enough. In our case, in Section \ref{s.deloc} we will prove a result with a similar taste (Theorem \ref{th.IVgff}) that will uniformly show that when $T$ is high enough, for any realisation of $\exp(iT\phi)$
\begin{equation}
 \Var{ \langle \phi, f \rangle\mid \exp(iT\phi)} \geq \epsilon\Var{\langle \phi, f \rangle}.
\end{equation}
This is read, in the context of \cite{FS}, as the study of \textbf{integer-valued GFF in an inhomogeneous medium}.

% Assume that you know $\exp(iT\phi)$ and you want to estimate $\phi$. Note that there is a conditional law of $\phi$ given $\exp(iT\Phi)$, so to study the quality of your approximation one could sample two independent copies of $\phi$ given that conditional law, $\phi_1$ and $\phi_2$. There is a way to recover $\phi$ if and only if $\phi_1$ and $\phi_2$ are both close enough, in other words if 
%\[\P\left((\phi_1-\phi_2,f)\geq \epsilon\right) \]
%is small.

%Let us first start with a classical identity that  will be key in this paper. Recall that if $X$ is in $L^2(d\P, \calF)$ and $\calA$ is any sub sigma-field. Then one has \footnote{The identity follows from the fact that $X= \Eb{X\md \calA}+ (X-\Eb{X \md \calA})$. Both are orthogonal in $L^2$ and $\Var{X \md \calA}$ is defined as $\Eb{(X-\Eb{X\md \calA})^2 \md \calA}$. }
%\begin{align}\label{E.conditional decomposition}
%\Eb{X^2} = \Eb{\Eb{X \md \calA}^2} + \Eb{\Var{X \md \calA}}\,.
%\end{align}

\bigskip

%\margin{$\Eb{X^2} = \Eb{\Eb{X \md \calA}^2} + \Eb{\Var{X \md \calA}}$}
\ni
{\bf Acknowledgments.}
We wish to thank J. Aru,
R. Bauerschmidt, 
V. Dang,
S. Druel,
K. Gawedzky,
P. Gille,
R. Peled,  
J-M. Stéphan and
F. Toninelli for very useful discussions. 
We also thank T. Spencer for very inspiring discussions after the first version of this work, in particular for the link with the random-phase Sine-Gordon model which he pointed out to us.
The research of the authors is supported by the ERC grant LiKo 676999. A.S. would also like to thank the hospitality of N\'ucleo Milenio ``Stochastic models of complex and disordered systems'' for repeated invitation to Santiago were a part of this paper was written.

\section{Preliminaries}\label{S.Preliminaries}

\subsection{Discrete differential calculus.} We start the preliminaries by discussing the basics of discrete differential calculus. As the only graph we work with in this paper is $\Lambda_n:= [-1,1]\cap \frac{1}{n}\Z^2$ with its canonical edge set, we only discuss the needed results in this framework. For simplicity we identify $\Lambda_n$ with its vertex set and we call and $E_{\Lambda_n}$ its edge set. For a  deeper discussion  on discrete differential calculus, we refer the reader to \cite{Sourav}.

In this section, we study two types of functions. Functions on vertices $S:\Lambda_n\mapsto \R$ and functions on directed edges $A:\overrightarrow{E}\mapsto \R$. Functions on vertices can take any values, however functions on directed edges have to always satisfy that
\begin{equation}
A(\overrightarrow{xy})=-A(\overrightarrow{yx}).
\end{equation}

Let us now present two canonical differential operators
\begin{align}
& \nabla S (\overrightarrow{xy}) = S(y)-S(x),\\
& \nabla \cdot  A (x) = \sum_{\overrightarrow{xy}} A(\overrightarrow{xy}).
\end{align}
Then, one can write the Laplacian of $S$ as follows
\begin{align}
\Delta S (x) = \nabla \cdot \nabla S (x)= \sum_{y\sim x} S(y)-S(x).
\end{align}

For a pair of functions on vertices $S_1,S_2: V\mapsto \R$, or on edges $A_1,A_2:\overrightarrow{E}\mapsto \R$,  we define
\begin{align*}
&\langle S_1, S_2\rangle := \sum_{x\in V} S_1(x) S_2(x),\\
&\langle A_1,A_2 \rangle := \frac{1}{2}\sum_{\overrightarrow{xy}} A_1(\overrightarrow{xy}) A_2 (\overrightarrow{xy}). 
\end{align*}
Furthermore, we define
\begin{align*}
\langle S_1,S_2 \rangle_{\nabla}= \langle \nabla S_1, \nabla S_2 \rangle. 
\end{align*}

Let us remark that the differentials $\nabla$ and $-\nabla\cdot$ are dual between them, i.e., 
\begin{equation}
\langle \nabla S, A \rangle = \langle S, -\nabla \cdot A \rangle.
\end{equation}
Thanks to this, we can see easily that $-\Delta$ is a positive definite operator. 

\begin{definition}[Inverse of the Laplacian]
	We fix $\partial \Lambda_n \subseteq \Lambda_n$ and we call it boundary. If $\partial \Lambda_n\neq \emptyset$, then for any function $S_1: \Lambda_n\backslash \partial \Lambda_n \mapsto \R$ there is a unique  function on the vertices $S_2: \Lambda_n\mapsto \R$ such that
	\begin{align*}
	&-\Delta S_2(x) = S_1(x), & & \forall x\in V\backslash \partial \Lambda_n\\
	& S_2(x)= 0 ,& & \forall x \in \partial \Lambda_n.
	\end{align*}
	In this case, we call $S_2:=(-\Delta^{-1})S_1$.
\end{definition}

The inverse of the Laplacian operator can be understood thanks to the Green's function
\begin{align}\label{E.Green's function definition}
G(x,\cdot)= -\Delta^{-1}(\1_x).
\end{align}
When it is needed, we will add a superscript to make explicit the boundary conditions of $G$. Let us recall a classical bound result for the Green's function in dimension 2.
	\begin{proposition} We have that for the graph $\Lambda_n$ and for both free and $0$ boundary condition and for any $x,y\in \Lambda_n$ 
		\begin{align}\label{E.Green's function}
		G(x,x)= C \log (d(x,\partial \Lambda_n))+O(1),
		\end{align}
		where $C$ does not depend on any other parameter.
	\end{proposition}

	\subsection{The Gaussian free field.} In this subsection, we introduce the GFF and some of the properties we use throughout the paper. For a more detailed discussion of the GFF, we refer the reader to \cite{SGFF,Sznitman2012LectureIso}.

	Let us fix a boundary set $\partial \Lambda_n$, the GFF with 0-boundary condition is the random function $\phi:V\mapsto \R$  such that
	\[\P((\phi(x_v)\in dx_v)_{v\in \Lambda_n})\propto \exp\left (-\frac{\langle \phi, \phi \rangle_\nabla}{2} \right )\prod_{v\in \Lambda_n \backslash \partial \Lambda_n}dx_v \prod_{\Lambda_n\in \partial \Lambda_n} \delta_0(dx_v).\]
	
	We say that $\phi$ is a GFF with free-boundary condition if $\partial \Lambda_n=\{x_0\}$, for some $x_0\in \Lambda_n$. We say that $\phi$ is a GFF with $0$ (or Dirichlet) boundary condition if
	\[\partial \Lambda_n := \{x\in \Lambda_n: |\Re(x)|=1 \text{ or } |\Im(x)|=1\}, \]
	in other words, the points in $\Lambda_n$ that are in the boundary of $[-1,1]^2$.

	An important equivalent characterisation of the GFF, is as the centred Gaussian process with covariance
	\[\E\left[\phi(x)\phi(y) \right]=G(x,y),  \]
	where the boundary values of the Green's function are associated with the boundary values of the GFF.
	
	A key property to understand the GFF is its Markov property.
	\begin{proposition}[Weak Markov property]\label{P.Weak Markov}
		Let $\phi$ be a GFF in $\Lambda_n$ with $0$-boundary condition in $\partial \Lambda_n$. Furthermore, let $B$ be a subset of the vertices of $\Lambda_n$. Then, there are two independent random function $\phi_B$ and $\phi^B$ such that $\phi=\phi_B+\phi^B$ and
		\begin{enumerate}
			\item $\phi_B$ is harmonic in $\Lambda_n\backslash B$.
			\item $\phi^B$ is a GFF in $\Lambda_n$ with 0-boundary condition in $\partial \Lambda_n\cup B$.
		\end{enumerate}
	\end{proposition}

		Let us, now, define a white noise on the edges of $\Lambda_n$.
		\begin{definition}[White noise]We denote $W$ a white noise, if $W$ is a function on the directed edges of $\Lambda_n$ such that $(W(\overrightarrow e))$ is a standard normal random variable independent of all other $W(\overrightarrow e')$ with $e\neq e'$.
		\end{definition}
		
		The discrete gradient of the GFF has an interesting relationship with the white noise. This result can be found in \cite{Aru} for this setting as well as in \cite{JC} for the same decomposition in the continuous.
%\margin{Ch: J'ai toujours trouvé cette proposition un peu disymetrique: on pourrait dire que $\zeta$ est un GFF sur les 2 formes non ? Si tu as une idee pour garder la proposition compacte, on peut changer, sinon on laisse comme ca. Ok, good, indeed il faudrait introduire le GFF sur les faces.} 
	\begin{proposition}\label{P.dphi+d*phi*}
		Let $\phi$ be a GFF  in $\Lambda_n$, then there exists a Gaussian process $\zeta(\overrightarrow{e})$ such that
		\[W:= \nabla \phi +\zeta \]
		is a white noise in $E$. Furthermore,
		\[\phi = \Delta^{-1} \nabla \cdot  W. \]
	\end{proposition}

%%%!!!!%%%%
%\subsection{The KT transition of the integer-valued GFF.}\label{ss.BackKT}
\subsection{Integer-valued Gaussian free field and the KT transition.}\label{ss.BackKT}

In this section, we briefly explain how Fröhlich and Spencer proved their delocalisation Theorem \ref{th.FS} as we will rely on the technology they developed (an expansion into Coulomb charges) later in Section \ref{s.deloc}. 
We refer the reader to the excellent review \cite{RonFS} from which we borrow the notations. See \cite{RonFS} for the relevant definitions. 

For simplicity, we fix a square domain $\Lambda \subset \Z^2$ and we consider the case of free boundary conditions rooted at some vertex $v\in \Lambda$. 

The proof by Fröhlich-Spencer can essentially be decomposed into the following successive steps:

{\em 1) The first step}  is to view the singular conditioning $\{\phi_i \in 2\pi \Z, \forall i \in \Lambda\}$ using Fourier series\footnote{It is slightly more convenient to consider the GFF conditioned to leave in $(2\pi \Z)^\Lambda$ rather than $\Z^\Lambda$. Following \cite{FS,RonFS}, we will stick to this convention here as well as in Section \ref{s.deloc} and Appendix \ref{a}.} thanks to the identity
\[
2\pi \sum_{m\in \Z} \delta_{2\pi m}(\phi)  \equiv 1 + 2 \sum_{q=1}^\infty \cos(q \phi)\,.
\]
To avoid dealing with infinite series, proceeding as in \cite{RonFS}, we consider the following approximate IV-GFF
\[
\P_{\beta,\Lambda,v}[d\phi] := \frac 1 {Z_{\beta,\Lambda,v}}  \prod_{i\in \Lambda} \left(1 + 2 \sum_{q=1}^N \cos( q \phi(i)) \right) \P_{\beta,\Lambda,v}^{\GFF}[d\phi]\,.
\]
In fact, more general measures are considered in \cite{FS,RonFS}:  they fix a family of trigonometric polynomials $\lambda_\Lambda := (\lambda_i)_{i\in \Lambda}$ attached to each vertex $i\in \Lambda$. These trigonometric polynomials are parametrized as follows: for each $i\in \Lambda$, 
\[
\lambda_i(\phi) = 1 + 2 \sum_{q=1}^N \hat \lambda_i(q) \cos(q \phi(i))\,.
\]
Now given a family of trigonometric polynomials $\lambda_\Lambda$, they define 
\[
\P_{\beta,\Lambda,\lambda_\Lambda, v}[d\phi] := \frac 1 {Z_{\beta,\Lambda,,\lambda_{\Lambda},v}}  \prod_{i\in \Lambda}  \lambda_i(\phi(i)) \P_{\beta,\Lambda,v}^{\GFF}[d\phi]\,.
\] 
We mention this degree of generality to keep the same notations as in \cite{FS, RonFS} and also so that the reader will not get confused when opening these references. Also, this degree of generality will be useful later in Appendix \ref{a.SG}.  Yet, in the present case, we will stick to the case where $\hat \lambda_i(q) = 1$ for all $i\in \Lambda$ and $1\leq q \leq N$. % and will thus only work with the above measures $\P_{\beta, \Lambda, v}$. 

\medskip
{\em 2) The second step} in the proof is to fix a test function $f:\Lambda \to \R$ such that $\sum_{i\in \Lambda} f(i) =0$ and to consider the Laplace transform of $\<{\phi, f}$, $\EFK{\beta,\Lambda,\lambda_\Lambda, v}{}{e^{\<{\phi,f}}}$. As $N\to \infty$ with our choice of trigonometric polynomials\footnote{Note that, in this section, instead of conditioning the GFF to be in $\Z^\Lambda$ as in Definition \ref{d.IVgff}, our convention in this section is to condition it to be in $(2\pi \Z)^\Lambda$. Besides changing constants, this does not make much difference}, this will converge to the Laplace transform $\EFK{\beta,\Lambda}{\IV}{e^{\<{\phi,f}}}$.  

By a simple change of variables, this Laplace transform can be rewritten 
\[
\EFK{\beta,\Lambda,\lambda_\Lambda, v}{}{e^{\<{\phi,f}}} 
= \frac 1 {Z_{\beta,\Lambda,\lambda_\Lambda,v}} \exp(\frac 1 {2 \beta} \<{f, - \Delta^{-1} f}) 
\EFK{\beta,\Lambda,v}{\GFF}
{\prod_{i\in \Lambda} \lambda_i(\phi(i) \purple{+\sigma(i)})}\,,
\]
where the function $\sigma=\sigma_f$ will be used throughout and is defined by
\begin{align}\label{e.sigma}
\sigma:= \frac 1 \beta [- \Delta]^{-1} f 
\end{align}

The main difficulty in the proof in \cite{FS} is in some sense to show that the effect induced by the shift $\purple{\sigma}$ does not have a dramatic effect compared to the exponential term $\exp(\frac 1 {2 \beta} \<{f, - \Delta^{-1} f})$ so that ultimately,
\begin{align*}\label{}
\EFK{\beta,\Lambda}{\IV}{e^{\<{\phi,f}}} \geq  \exp(\frac 1 {2 \beta(1+\eps)} \<{f, - \Delta^{-1} f})\,.
\end{align*}
From such a lower bound on the Laplace transform,  one can easily extract delocalisation properties of the IV-GFF. 

\medskip

{\em 3) The third (and by far most difficult) step} is to control the effect of the shift $\purple{\sigma}$ via a highly non-trivial expansion into Coulomb charges which enables to rewrite the partition function as follows:
\begin{align*}\label{}
Z_{\beta,\Lambda,\lambda_\Lambda,v} = \sum_{\calN \in \calF} c_\calN \int \prod_{\rho\in \calN} [1+z(\beta,\rho, \calN) \cos(\<{\phi,\bar\rho})] d\mu_{\beta,\Lambda,v}^\GFF(\phi)\,.
\end{align*}

We refer to \cite{FS,RonFS} for the notations used in this expression and in particular for the concept of {\em charges} (i.e. $\rho : \Lambda \to \R$), {\em ensembles} (i.e. sets $\calN$ of mutually disjoint charges $\rho$) etc. 

One important feature of this expansion into charges is the fact that under some (very general) assumptions on the growth of the Fourier coefficients $|\hat \lambda_i(q)|$ (see (5.35) in \cite{FS}), it can be shown that the effective activities $z(\beta, \rho, \calN)$ decay fast. Namely (see (1.14) in \cite{RonFS}),
\[
|z(\beta, \rho, \calN)|
\leq \exp \left( -\frac c \beta (\|\rho\|_2^2 + \log_2(\mathrm{diam}(\rho) + 1))\right).
\]
As such we see that at high temperature, the partition function corresponds to a sum of positive measures. (Also the weights $c_\calN$ are positive and s.t. $\sum c_\calN =1$). 

\begin{remark}\label{}
In \cite{RonFS}, the authors have introduced a slightly different definition of the free b.c. GFF which makes the analysis behind this decomposition into charges more pleasant (their definition cures the presence of non-neutral charges $\rho$ very easily). One can switch to their more convenient definition in our setting since in the limit $N\to \infty$, both give the same integer-valued GFF. 
\end{remark}

This crucial third step thus allows us to rewrite the Laplace transform $\EFK{\beta,\Lambda,\lambda_\Lambda, v}{}{e^{\<{\phi,f}}}$ as follows: 
\begin{align*}\label{}
e^{\frac 1 {2 \beta} \<{f, - \Delta^{-1} f}}
\frac 
{
\sum_{\calN \in \calF} c_\calN \int \prod_{\rho\in \calN} [1+z(\beta,\rho, \calN) \cos(\<{\phi,\bar\rho} + \purple{\<{\sigma, \rho}})] d\mu_{\beta,\Lambda,v}^\GFF(\phi)
}
{
\sum_{\calN \in \calF} c_\calN \int \prod_{\rho\in \calN} [1+z(\beta,\rho, \calN) \cos(\<{\phi,\bar\rho})] d\mu_{\beta,\Lambda,v}^\GFF(\phi)
}\,.
\end{align*}
We now rewrite this ratio as  (thus defining $Z_\calN(\purple{\sigma})$ and $Z_\calN(0)$)
\begin{align*}\label{}
\EFK{\beta,\Lambda,\lambda_\Lambda, v}{}{e^{\<{\phi,f}}}
= e^{\frac 1 {2 \beta} \<{f, - \Delta^{-1} f}}
\frac
{\sum_{\calN\in \calF} c_\calN Z_{\calN}(\purple{\sigma})}
{\sum_{\calN\in\calF} c_\calN Z_\calN(0)}
\end{align*}

\medskip
{\em 4) The fourth step is an analysis} for each fixed ensemble $\calN\in \calF$ of the above ratio $\frac {Z_{\calN}(\purple{\sigma})} {Z_{\calN}(0)}$. Trigonometric inequalities are used here to in order to obtain for each $\calN$:
\begin{align*}\label{}
 \frac {Z_{\calN}(\purple{\sigma})} {Z_{\calN}(0)}  \geq &  
\exp\big[-D_4 \sum_{\rho\in \calN} |z(\beta,\rho,\calN)|\<{\sigma,\rho}^2\big] \\
& \times 
\int \frac
{e^{S(\calN,\phi)}} 
{Z_{\calN}(0)}
\prod_{\rho\in\calN}[1+z(\beta,\rho,\calN) \cos(\<{\phi,\bar \rho})]
d\mu_{\beta,\Lambda,v}^\GFF(\phi)\,, \nn
\end{align*}
where
\begin{align}\label{e.Sba}
S(\calN,\phi):=
-\sum_{\rho\in\calN} \frac
{z(\beta,\rho,\calN)\sin(\<{\phi,\bar \rho})\sin(\<{\sigma,\rho})}
{1+z(\beta,\rho,\calN)\cos(\<{\phi,\bar \rho})}
\end{align}
Two crucial observations are made at this stage:
\bnum
\item The functional $\phi\mapsto S(\calN, \phi)$ is odd in $\phi$
\item The  measure $\prod_{\rho\in\calN}[1+z(\beta,\rho,\calN) \cos(\<{\phi,\bar \rho})] d\mu_{\beta,\Lambda,v}^\GFF(\phi)$ is invariant under $\phi \to -\phi$.
\enum
All together this simplifies tremendously the above lower bound, as by using Jensen, one obtains readily 
\begin{align*}\label{}
 \frac {Z_{\calN}(\purple{\sigma})} {Z_{\calN}(0)}  \geq &  
\exp\big[-D_4 \sum_{\rho\in \calN} |z(\beta,\rho,\calN)|\<{\purple{\sigma},\rho}^2\big] \,.\nn
\end{align*}
From this lower bound together with the specific construction of the ensembles of charges $\calN$, it is then not very difficult to conclude the proof with the desired lower bound
\begin{align*}\label{}
\EFK{\beta,\Lambda,\lambda_\Lambda,c}{}{e^{\<{\phi,f}}} \geq  \exp(\frac 1 {2 \beta(1+\eps)} \<{f, - \Delta^{-1} f})\,.
\end{align*}

\medskip
As we will see in Section \ref{s.deloc}, the effect of shifting the $\Z$ fibers by $\ba\in\R^\Lambda$ will translate as follows: 
\begin{align*}\label{}
&\EFK{\beta,\Lambda,\lambda_\Lambda,v}{\ba}{e^{\<{\phi,f}}}  \nn\\
&=
e^{\frac 1 {2\beta} \<{f, - \Delta^{-1} f}} 
\frac{\sum_{\calN \in \calF} c_\calN \int \prod_{\rho\in \calN} [1+z(\beta,\rho, \calN) \cos(\<{\phi,\bar\rho}+\<{\sigma-\ba,\rho})] d\mu_{\beta,\Lambda,v}^\GFF(\phi)}
{\sum_{\calN \in \calF} c_\calN \int \prod_{\rho\in \calN} [1+z(\beta,\rho, \calN) \cos(\<{\phi,\bar\rho} - \<{\ba,\rho})]d\mu_{\beta,\Lambda,v}^\GFF(\phi)}\,.\nn 
\end{align*}

The difficulty for us will be that, generically, $\ba$ is much less regular than $\sigma$ (defined in ~\eqref{e.sigma}) which thus makes the Dirichlet energy $\<{\nabla(\sigma - \ba),\nabla( \sigma - \ba)}$ typically huge. Because of this, we will not be able anymore to rely on the two symmetries above (in particular the use of Jensen is not longer possible except for very specific choices of $\ba$, see the discussion after Definition \ref{d.IVba}). We will come back to this in Section \ref{s.deloc}.

\begin{remark}\label{R.Wirth}
The case of Dirichlet  boundary conditions has been outlined in the appendix D. of \cite{FS} and the details of the proof appeared very recently in the appendix of \cite{Wirth}. The proof structure highlighted above for free b.c. still  holds except the decomposition into charges needs to be adapted to the presence of a boundary. See the Appendix in \cite{Wirth}. 
 
We also point out that the nice symmetrization argument used in \cite{Wirth} does not apply to our case (as $\ba$ is far from being harmonic) and also because the symmetrized measure in most cases does not provide informations on the fluctuations we need.
\end{remark}

%%%%%!!!!%%%%

\subsection{Link with the $\ba$-shifted integer-valued GFF.}

In this section, we precise the link between our statistical reconstruction problem and the $\ba$-shifted IV-GFF introduced earlier (in Definition \ref{d.IVba}).
\begin{lemma}\label{L.aIV}
Let $\Lambda \subset \Z^2$, $T>0$ and $\ba \in  \R^\Lambda$ with $\ba_{\md \p \Lambda} \equiv 0$. If $\phi$ is a 0-boundary GFF (with inverse temperature $\beta=1$) on $\Lambda$, then its conditional law given $\phi\Mod{\frac {2\pi} T} = \frac {2\pi} T \ba$ is given by $\frac{2\pi} T \psi$, where $\psi \sim \P_{\beta_T, \Lambda}^{\ba,\IV}$ and where the $T$-dependent inverse temperature $\beta_T$ is given by 
\[
\beta_T:= \frac {(2\pi)^2} {T^{2}}\,.
\]
Equivalently, for any functional $F : \R^\Lambda \to \R$, 
\begin{align*}\label{}
\Eb{ F(\phi)  \md  e^{iT \Phi}=e^{2\pi i \ba }} &=
\EFK{\beta=\beta_T,\Lambda}{\ba,\mathrm{IV}}{F(\frac{2\pi}{T} \psi)}\,.
\end{align*}
\end{lemma}

%The proof is straighforward.

\ni
{\bf Proof.}

Recall from Definition \ref{d.IVba}, 
\begin{align*}\label{}
\FK{\beta,\Lambda}{\ba,\mathrm{IV}}{d\phi}:=\frac 1 {Z}
\sum_{\bm \in \Z^\Lambda,   \bm_{\md \p \Lambda}\equiv 0} \delta_{\bm+\ba}(\phi) 
\exp(- \frac \beta 2 %\frac {(2\pi)^2} {2 T^2} 
\<{\nabla(\bm + \ba), \nabla(\bm+\ba)})
\end{align*}
Now, by desintegration, for any $\calC^0$ functional $F : \R^\Lambda \to \R$, one has
\begin{align*}\label{}
\Eb{ F(\phi)  \md  e^{iT \Phi}  = e^ {2\pi i \ba } } 
& = 
\frac
{\sum_{\bm\in \Z^\Lambda} \exp\left( -\frac 1 2 \<{\frac {2\pi} T (\ba + \bm),  (-\Delta) \frac {2\pi} T (\bm+ \ba)}\right)   F(\frac {2\pi} T (m +  \ba)) }
{\sum_{\bm\in \Z^\Lambda} \exp\left( -\frac 1 2 \<{\frac {2\pi} T (\ba + \bm),  (-\Delta) \frac {2\pi} T (\bm+ \ba)}\right)  }
\\
& = 
\frac
{\sum_{\bm\in \Z^\Lambda} \exp\left( -\frac {(2\pi)^2} {2T^2} \langle (\ba + \bm),  (-\Delta) (\bm+ \ba)\rangle \right)  F(
\frac {2\pi} {T} (m +  \ba)) }
{
\sum_{\bm\in \Z^\Lambda} \exp\left( -\frac {(2\pi)^2} {2T^2} \langle (\ba + \bm),  (-\Delta) (\bm+ \ba)\rangle \right) 
 } \\
 & = \EFK{\beta=\beta_T,\Lambda}{\ba,\mathrm{IV}}{F(\frac{2\pi}{T} \phi)}
\end{align*}
where we have made the slightly unusual choice $\beta_T:= (2\pi)^2 \,T^{-2}$ (in order to avoid dealing with $\sqrt{T}$ in most of the introduction).  \qed

\section{Localisation regime}\label{S.Localisation}
In this section, we prove the first part of Theorems \ref{T.0-boundary theorem} and \ref{T.free-boundary theorem}. That is to say, we show that one can recover a GFF knowing $\exp(i T\phi)$, in fact the recovery function is fairly straightforward:
\[F(\exp(i T \phi))(x):=\E\left[\phi(x) \mid \exp(i T \phi) \right].\]

To show that this is the right function, we need to recall \eqref{E.Markov ineq}. It says that to prove the first part of Theorems \ref{T.0-boundary theorem} and \ref{T.free-boundary theorem}, it is enough to show that  if $\phi$ is a GFF in $\Lambda_n$ and $f$ a fixed smooth function in $f$ 
\begin{equation}\label{E.localisation by variance}
\E\left[\Var{\langle{\phi,f}\rangle \md e^{iT \phi}} \right]=o\left( n^4\right).
\end{equation}

Let us note that this approach may not look useful at first glance, as to bound this conditional variance we need to compute the  conditional expectation, which is a non-trivial function of $\exp(iT\phi)$. To circumvent this issue, we write the conditional variance as follows. 
\begin{align*}\label{} 
\Var{\langle{\phi,f}\rangle \md e^{iTh}} & = 
\Eb{(\<{\phi,f}-\Eb{\<{\phi,f}\md e^{i Th}})^2 \md  e^{i T\phi}} \\
& = \frac 1 2 \Eb{(\<{\phi_1,f}-\<{\phi_2,f})^2 \md  e^{i T\phi}}\,,
\end{align*}
where $\phi_1,\phi_2$ are conditionally independent given $e^{iT\phi}$. Let us be more explicit about this law.
\begin{definition}\label{D.phis}
	Let us take $\phi$ a GFF in $\Lambda_n$ with any given boundary. We denote $(\phi_1,\phi_2)$ a pair of GFF in $\Lambda_n$ with the same boundary condition such that a.s. $e^{iT \phi_1}=e^{iT \phi_2}=e^{iT \phi}$ and $\phi_1$ is conditionally independent of $\phi_2$ given $e^{iT \phi}$. 
	In other words  %\margin{Ch: $d\phi_i(x)$ ?}
	\begin{align*}
	\P\left[(d\phi_1,d\phi_2)\mid e^{iT \phi}\right]\propto \prod_{i=1,2}\left( e^{ -\frac{1}{2}\langle \phi_i, \phi_i \rangle_{\nabla }}  \prod_{x\in \Lambda_n\backslash \partial \Lambda_n}\left (\sum_{k\in \Z}   \delta_{\frac{2\pi k}{T}+\phi(x)}\left (d\phi_i(x) \right )\right )\right).
	\end{align*}
\end{definition}

To prove \eqref{E.localisation by variance}, we use an averaged Peierls argument.

\subsection{Large gradients are costly for a GFF.}
The first stage to implement Peierls argument, is to show that it is costly for a GFF to have many edges with large gradients. To do this we are going to use the Markov property, i.e. Proposition \ref{P.Weak Markov}. In fact for a given deterministic set $B\subseteq \Lambda_n$, we need to understand what is the law of the norm of $\phi_B$ .
\begin{lemma}\label{L.Law of the norm} Let us work in the context of Proposition \ref{P.Weak Markov} with $B\cap \partial \Lambda_n=\emptyset$. 
	\begin{enumerate}
		\item The law of $\|\phi_B\|_{\nabla}^2$ is that of a $\chi^2$ with $|B|$ degrees of freedom.
		\item The law of $\|\phi^B\|_{\nabla}^2$ is that of a $\chi^2$ with $|\Lambda_n\backslash (\partial\Lambda_n\cup B)|$ degrees of freedom.
	\end{enumerate}
\end{lemma}
\begin{proof}
	We start defining Harm$(B)$ as the set of functions $\Lambda_n\mapsto \R$ that are harmonic in $\Lambda_n\backslash (B\cup \partial \Lambda_n)$ and take value $0$ in $\partial \Lambda_n$. In fact, we have that $\Phi_B$ is the orthogonal projection of $\phi$ to Harm$(B)$ under the inner product $\langle \cdot, \cdot \rangle_{\nabla}$ (see for example Section 2.6 of \cite{SGFF}). One can, now, check that the subspace Harm$(B)$ has dimension $|B|$ from which (1) follows. As $\phi^B$ is the orthogonal projection under Harm$(B)^\perp$, (2) follows by a similar reason, as the space of functions with $0$ boundary condition on $B\cup \partial \Lambda_n$ has dimension  $|\Lambda_n\backslash (\partial\Lambda_n\cup B)|$.
\end{proof}

We can now use this proposition to obtain the basic input we need for a Peierl's argument.
\begin{lemma}\label{L.Energy lemma}
	Let $\phi$ be  a GFF in $\Lambda_n$ with either 0 or free boundary condition. Then, there exist constants $\alpha, C,u_0>0$ independent of $\Lambda_n$ such that for all a finite set  $F$ of edges and all $u>u_0$
	\begin{align*}\label{}
	\Pb{|\phi(x)-\phi(y)| \geq u\,, \forall xy\in F} \leq Ce^{- \alpha u^2 |F|}\,.
	\end{align*}
\end{lemma}
\begin{proof}
	We use the Markov property of the GFF (Proposition \ref{P.Weak Markov}) with the subset of vertices $B$ such that $x\in B$ if there exists $xy\in F$. Let us note that $| B|\leq 2 |F|$. We have that
	\begin{equation}
	\| \phi_B\|_{\nabla}^2= \sum_{xy\in E} (\phi_B(y)-\phi_B(x))^2\geq \sum_{xy\in F} (\phi(y)-\phi(x))^2.
	\end{equation}
	has the law of  a $\chi^2$ with $|B|$ degrees of freedom. Let us note that thanks to Proposition \ref{P.Weak Markov} (1), we have that $\phi_B(y)-\phi_B(x)$ is equal to $\phi(y)-\phi(x)$. Thus, using that $| B |\leq 2 | F|$
	\begin{align*}
	\P\left( |\phi(x)-\phi(y)| \geq u\,, \forall xy\in F\right)&\leq \P\left (\| \phi_B\|_{\nabla}^2\geq u^2|F| \right )\\
	&\leq \P\left (\| \phi_B\|_{\nabla}^2\geq \frac{u^2|B|}{2} \right ).
	\end{align*}
	We can now use Lemma \ref{L.Law of the norm} (1), to continue and see that when $u$ is big enough 
	\begin{equation}
	\P\left( |\phi(x)-\phi(y)| \geq u\,, \forall xy\in F\right) \leq C\exp(-4\alpha u^2|B|) \leq C\exp(-\alpha u^2 |F|),
	\end{equation}
	where we used that $|F|\leq 4|B|$.

\end{proof}

\subsection{ The GFFs $\phi_1$ and $\phi_2$ agree on a dense percolating set.}

\subsubsection{The $0$-boundary case.}
Take $\phi$ a $0$-boundary GFF in $\Lambda_n$ and assume we are given an instance of $e^{iT\phi}$. Let us sample two conditionally independent copies $\phi_1,\phi_2$ given $e^{i T\phi}$ as in Definition \ref{D.phis}. Let us now introduce the following definition
	\begin{definition}[$I$]We denote $I:=I(\phi_1,\phi_2)$ the connected component connected to the boundary, $\partial \Lambda_n$, of the random set 
		$\{x\in \Lambda_n, \phi_1(x) =\phi_2(x)\}.$
	\end{definition}
	
	Recall that by definition, $\phi_1,\phi_2$ are GFF with zero boundary conditions and as such one needs to have $\phi_1\equiv \phi_2$ on $\p \Lambda_n$. 
	
	%\begin{figure}[h!]
	%	\includegraphics[width=0.3\textwidth]{Pictures/I2.pdf}
	%	\includegraphics[width=0.3\textwidth]{Pictures/I.pdf}
	%	\caption{The grey colour on the left side picture represent al points where $\phi_1=\phi_2$. The grey colour in the right picture represents $I$.}
	%\end{figure}

	Our goal in this subsection is to show via an {\em annealed} Peierls argument, that with high probability when $T$ is small, the random set $I$ is {\em percolating} inside $\Lambda_n$. To study this, for any $x\in \Lambda_n$ \textbf{we define $O(x)$} as the empty set if $x\in I$ and as the connected component containing $x$ of $\Lambda\backslash I$ if $x\notin I$.

	Our main observation is that having an edge connecting $O(x)$ with $\Lambda_n\backslash O(x)$ is costly in the sense that it forces either $|\nabla\phi_1(e)|$ or $|\nabla \phi_2(e)|$ to be larger than $\pi/T$. Indeed the values of $\phi_1$ and $\phi_2$ are fixed modulo $2\pi T$, in other words for any $x\in \Lambda_n$ and $i\in \{1,2\}$,
	\[
	\phi_i(x) \in \phi(x) + \frac{2\pi}{T} \Z\,.
	\]
	This way, if $\phi_1,\phi_2$ agree on $x$ but disagree on $y\sim x$, this means that either $|\phi_1(x)-\phi_1(y)|> \pi/T$ or $|\phi_2(x)-\phi_x(y)|>\pi/T$. We then have the following proposition.
	
	\begin{proposition}\label{P.Peierls 0}
		Using the definitions introduced above, for all $T$ small enough there exists $\varpi(T)>0$ and $C>0$ such that
		\[\P(\diam(O(x))\geq L)\leq C\exp(-\varpi(T) L) \]
	\end{proposition}
	\begin{proof}
		Let us note that if $\diam(O(x))\geq L$ there is a subset of edges $\eta$ of length at least $L$ such that its dual is a connected path surrounding $x$ and for every $e \in \eta$ either $|\nabla \phi_1 (e)|\geq \pi/T$ or $|\nabla \phi_2 (e)|\geq \pi/T$. This implies that
		\begin{align}
		\label{E.cardinal of B}\P(\diam(O(x))\geq L) &\leq \sum_{\substack{|\eta|\geq L\\\eta \text{ surrounds }x}} \P(|\nabla \phi_1 (e)|\geq \pi/T \text{ or } |\nabla \phi_2 (e)|\geq \pi/T, \forall e \in \eta )
		\end{align}
		Let us fix $\eta$ and suppose that for all $e\in \eta$, either $|\nabla \phi_1 (e)|\geq \pi/T$ or $|\nabla \phi_2 (e)|\geq \pi/T$. This implies that there exists a $F\subseteq \eta$ and $i\in \{1,2\}$ such that for all $e\in F$ we have that $|\nabla\phi_i(e)|\geq \pi/T$. This implies that
		\begin{align*}
		\P\left (|\nabla \phi_1 (e)|\geq \frac{\pi}{T} \text{ or } |\nabla \phi_2 (e)|\geq \frac{\pi}{T}, \forall e \in \eta \right )
		&\leq 2\sum_{j=\lfloor \frac{|\eta|}{2}\rfloor}^{|\eta|}\binom{|\eta|}{j} \P\left( |\nabla \phi (e)| \geq \frac{\pi}{T}, \forall e\in F \right)\\
		&\leq 2^{|\eta|+1}\sum_{j=\lfloor \frac{|\eta|}{2}\rfloor}^{|\eta|}\exp\left (-2\tilde \alpha \frac{j}{ T^2}\right ),
		\end{align*}
		where we used Lemma \ref{L.Energy lemma} and  that both $\phi_1$ and $\phi_2$ have the law of a GFF in $\Lambda$. Additionally, $\tilde \alpha:=\alpha \pi^2/2$ . Thus, \eqref{E.cardinal of B} is less than or equal to
		\begin{align*}
		C\sum_{k \geq  L} \sum_{\substack{|\eta|=k\\\eta \text{ surrounds }x}}2^k \exp\left (-\frac{\tilde \alpha}{T^2} k\right )&\leq C\sum_{k\geq  L} \exp(-k(\tilde \alpha T^{-2}-\log 2 -\log 3) )\\
		&\leq \tilde C \exp(- L(\tilde \alpha T^{-2}-\log 2 -\log 3))
		\end{align*}
		where we used that the amount of $\eta$ such that $|\eta|=k$ and $\eta$ surrounds $x$ is less than $C\,3^k$, and that $T$ is such that
		\begin{equation}
		\tilde \alpha T^{-2} - \log 6 >0.
		\end{equation}
	\end{proof}

	\subsubsection{The free boundary case.}\label{sss.free}
	
	We need to modify significantly the above definitions in order to analyze the free boundary case. We assume the free boundary GFF is rooted at some vertex $x_0 \in \Lambda_n$. 
	As in the Dirichlet case, $(\phi_1,\phi_2)$ will still denote two conditionaly independent copies of the GFF given $e^{i T \phi}$. 
	
	The main difference w.r.t. Dirichlet is that when $T$ is small, it is no longer true that with high probability $\phi_1$ and $\phi_2$ will agree on a large percolating set. Instead, we will find a large set, which we will call $I$ again together with a random integer $m_I\in \Z$ %\margin{WILL WE NEED to CONDITION on $m_I$, \avelio{Yes, but in a hidden way}} 
	such that 
	\[
	\phi_1(x) = \phi_2(x) + m_I \frac {2\pi} T \,,\,\, \forall i \in  I
	\]
	Let us then introduce the following sets: for each $m\in \Z$, let 
	\[
	\hat I_m : = \text{Largest connected component of }\left \{ x\in \Lambda_n, \phi_1(x) = \phi_2(x) + m \frac {2\pi} T \right \}.
	\]
         If there are two of the same size, we choose one in a deterministic way.
	From these subsets $\hat I_m$, we define the set $I$ as well as the connected components $\{ O(x)\}_{x\in \Lambda_n}$  as follows:
	\bi
	\item If there is a unique $m_0\in \Z$ s.t. $\hat I_{m_0}$ has (graph) diameter larger than $\frac n 2 $, then we define 
	\[
	I:= \hat I_{m_0}
	\] 
	and for any $x\in \Lambda_n$, we define $O(x)$ to be empty if $x\in I$ and to be the connected component of $x$ in  $\Lambda_n \setminus I$ otherwise. 
	\item If on the other hand, one can find two integers $m_1,m_2$ s.t. both $\hat I_{m_1}$ and $\hat I_{m_2}$ have diameter greater than $\frac n 2$, then we define 
	\[
	I:= \emptyset\text{    and    } O(x):=\Lambda_n \,,\,\,\,\,\forall x \in \Lambda_n
	\]
	\ei
	
	We can now state the analogue of Proposition \ref{P.Peierls 0} for free b.c.

	%\begin{proposition}\label{P.Peirls free}
	%Let $\phi_1, \phi_2$ two free-boundary GFF such that $\exp(iT\phi_1)=\exp(iT\phi_2)$ and conditionally independent given $\exp(iT\phi_1)$. Then, if we define $I:=I(\phi_1,\phi_2)$ as the biggest connected component of a set such that $\phi_1-\phi_2$ is constant on that set and $O(x)$ as the connected component of $V\backslash I$ that contains $x$, we have that
	% \[\P(\mathop{diam}(O(x))\geq L)\leq C\exp(-\varpi(T)  L).\]
	%\end{proposition}
	
	\begin{proposition}\label{P.Peirls free}
		Let $\phi_1, \phi_2$ two free-boundary GFF such that $\exp(iT\phi_1)=\exp(iT\phi_2)$ and conditionally independent given $\exp(iT\phi_1)$. Then using the above definitions (for free b.c.), for all $T$ small enough there exists $\varpi(T)>0$ and $C>0$ such that for all $x\in \Lambda_n$, 
		\begin{equation}\label{e.exponential_decay_free}
		\P(\diam(O(x))\geq L)\leq C\exp(-\varpi(T) L)
		\end{equation}
	\end{proposition}

	\begin{proof}
		The proof follows the same lines as in the Dirichlet case, as Lemma \ref{L.Energy lemma} does not care about the boundary conditions. The only difference is that we need to deal with the dichotomy entering into the definition of the set $I$ (which does not exist for the Dirichlet case). For this, note that  in order to have two sets $\hat I_{m_1}, \hat I_{m_2}$ with $m_1\neq m_2$ and both have diameter $\geq n/2$, there must exist at least one path $\eta$ in the dual graph $(\Z^2)^*$ which has diameter greater than $n/2$ and which satisfies the constraint that any $e\in \eta$ is such that either $|\nabla \phi_1(e)| \geq \frac \pi T$ or 
		$|\nabla \phi_2(e)| \geq \frac \pi T$. By Lemma \ref{L.Energy lemma} and the same argument that in the Dirichlet case, such a case only happens with probability less than $O(n) \exp(- \frac {\tilde \alpha} {10 T^2} n)$. Note that the same argument implies that there is at most one connected component of $\{ x\in \Lambda_n, \phi_1(x) = \phi_2(x) + m \frac {2\pi} T \}$ with diameter at least $n/2$. 
%		
%		
%		 Let us now see that with high probability there is one $I_m$ with size at least $n/2$. If this were not the case, there would be a dual path from the top to the bottom of $\Lambda_n$ such that that any $e\in \eta$ is such that either $|\nabla \phi_1(e)| \geq \frac \pi T$ or 
%			$|\nabla \phi_2(e)| \geq \frac \pi T$. As before, this case only happens with probability less than $O(n) \exp(- \frac {\tilde \alpha} {10 T^2} n)$.
%		
		
		Note that after defining $I$, the argument of Proposition \ref{P.Peierls 0} implies that the possibility for any point $x$ that $O(x)$ is huge only arises with exponentially decaying probability in the diameter $n$. Now, note that \eqref{e.exponential_decay_free} is obviously true as long as $L> 2n$, and thus for $L\leq 2n$ we have that
		\begin{align*}
		\P(\diam(O(x))\geq L) &\leq \P(\diam(O(x))\geq L, I\neq \emptyset) + \P(I=\emptyset )\\
		& \leq \frac{C}{2}e^{-\varpi(T)L} + \frac{C}{2}e^{-\frac{\varpi(T)}{2}n }\\
		&\leq Ce^{-\varpi(T)L}.		
		\end{align*}
	\end{proof}

	%
	%Let us now finish with a short discussion about the free boundary case.
	%\begin{proposition}\label{P.Peirls free}
	%Let $\phi_1, \phi_2$ two free-boundary GFF such that $\exp(iT\phi_1)=\exp(iT\phi_2)$ and conditionally independent given $\exp(iT\phi_1)$. Then, if we define $I:=I(\phi_1,\phi_2)$ as the biggest connected component of a set such that $\phi_1-\phi_2$ is constant on that set and $O(x)$ as the connected component of $V\backslash I$ that contains $x$, we have that
	% \[\P(\mathop{diam}(O(x))\geq L)\leq C\exp(-\varpi(T)  L).\]
	%\end{proposition}
	%\begin{proof}
	%	\avelio{Proof missing but simple?.}
	%\end{proof}
	
	\subsection{The conditional variance is small for $0$-boundary GFF.}\label{SS.Conditional variance 0-boundary}
	We will now prove \eqref{E.localisation by variance} for a $0$-boundary GFF. Let us now study the law of $(\phi_1,\phi_2)$ conditionally on $I$ and the values of $\phi_1$ on $I$. We fix $e^{iT \phi}$, $I$, and the values of $\phi_1$ on $I$ and take $(\varphi_1,\varphi_2)$ a possible value of $(\phi_1,\phi_2)$ that satisfy the conditioning. Note that to check whether $(\varphi_1,\varphi_2)$ is a possible realisation, one just needs to check that $(\varphi_1)_{\mid I}=(\varphi_2)_{\mid I}=(\phi_1)_{\mid I}$, and that for any $O$ connected component of $\Lambda_n\backslash I$, the pair $(\varphi_1,\varphi_2)$ restricted to $O$ locally satisfies the conditions,  i.e. \begin{equation*}
	\varphi_1(x) = \varphi_2(x) = \phi(x) \Mod{\frac{2\pi}{T}}, \ \ \text{ for all } x\in O.
	\end{equation*} Furthermore, if we define $\bar O$ the graph induced by all the edges in $\Lambda_n$ that have at least one vertex in $O$ we have that
	\begin{align}\label{E.Conditioning}
	\P\left( (\phi_1,\phi_2)=(\varphi_1,\varphi_2)\mid e^{iT\phi}, I,(\phi_1)_{\mid I}\right)
	\propto \prod_{ O} e^{-\frac{1}{2}\left( \langle(\varphi_1)_{\mid \bar O}, (\varphi_1)_{\mid \bar O}\rangle_{\nabla}+\langle(\varphi_2)_{\mid \bar O}, (\varphi_2)_{\mid \bar O}\rangle_{\nabla}\right)}.
	\end{align}
	
	As a consequence of \eqref{E.Conditioning}, we have that under this conditioning the law of $(\phi_1,\phi_2)$ restricted to $O$ is independent of the law of $O'$ if $O\neq O'$. Thus, $\E\left[\langle \phi_1-\phi_2,f\rangle^2 \right]$ is equal to
	\begin{align}
	\nonumber&\sum_{x,y\in \Lambda_n} f(x)f(y)\E\left[ (\phi_1-\phi_2)(x)(\phi_1-\phi_2)(y) \1_{O(x)=O(y)}\right]\\
	& \hspace{0.005\textwidth}\leq \sum_{x,y\in \Lambda_n} |f(x)||f(y)|	\E\left[ (\phi_1-\phi_2)^2(x)(\phi_1-\phi_2)^2(y)\right ]^{1/2}\P\left(  \1_{O(x)=O(y)}\right)^{1/2}.\nonumber\\
	& \hspace{0.005\textwidth}\leq \sum_{x,y\in \Lambda_n} |f(x)||f(y)|\left( 	\E\left[ (\phi_1-\phi_2)^4(x)\right]+ \E\left [ (\phi_1-\phi_2)^2(y)\right ]\right)^{1/2}\P\left(  \1_{O(x)=O(y)}\right)^{1/2}.\label{E.conditional variance}\end{align}
	
	We can now just bound
	\begin{align*}
	\E\left[ (\phi_1-\phi_2)^4(x)\right ]\leq 16 \E\left[\phi_1(x)^4 \right]=48\, G_{n}^2(x,x).
	\end{align*}
	%To bound the second factor, we need a basic Gaussian inequality whose proof you can find in Lemma 2.2 of \cite{Se}.
	%\begin{lemma}\label{L.Basic Gaussian bound}Let $X$ be a centered Gaussian random variable, and $A$ a set. Then, there exists a deterministic constant $K>0$ such that
	%	\[\E\left[X^2 \1_{A} \right]\leq K\Var{X} \P(A)|\log \P(A)|  \]
	%\end{lemma}
	
	%Using Lemma \ref{L.Basic Gaussian bound}, we have that
	%\begin{align}
	%\label{E.Variance at point y}\E\left [ (\phi_1-\phi_2)^2(y) \1_{O(x)=O(y)}\right]& \leq 2\E\left[\phi_1(y)^2 \1_{O(x)=O(y)} \right]\\
	%\nonumber&\leq 2 K G(y,y) \P(O(x)=O(y))|\log \P(O(x)=O(y))|.
	%\end{align}
	Note that on the event $O(x)=O(y)$ the diameter of $O(x)\geq d_{\Lambda_n}(x,y)$. Thus, we have that there exists an absolute constant $C,\varpi(T)>0$ such that
	\begin{align*}
	\P\left [ O(x)=O(y)\right]\leq C\exp(-\varpi(T) \|x-y\|).
	\end{align*}
	
	From the fact that $\exp(-\varpi(T) \|x-y\|)$ decreases exponentially as $\|x-y\|$ goes to infinity, we have that  
	\begin{align}\label{E.Final variance}
	\E\left[\langle \phi_1-\phi_2,f\rangle^2 \right]\leq C\|f\|_\infty^2 \sup_{x}G_n(x,x)n^2\leq C\|f\|_\infty^2 \log(n) n^2,
	\end{align}
	which proves \eqref{E.localisation by variance} and gives in fact a more quantitative rate of convergence.
	
	\subsection{The conditional variance is small enough for free boundary Gaussian free field.}\label{SS.Conditional variance free-boundary}
	We will now prove \eqref{E.localisation by variance} for a free-boundary GFF.   The proof is very similar to that of the $0$-boundary condition so we are going to do a sketch of the proof only highlighting the difference with the Dirichlet boundary case.
	
	Let us take $(\phi_1,\phi_2)$ a pair of GFF with $0$-boundary condition in $\{x_0\}$ coupled as in Definition \ref{D.phis}. Thanks to Proposition \ref{P.Peirls free}, we have that there exist a (random) set $I$ and a (random) integer $m_I$ such that for all $y\in I$, $\Phi_1(y)= \Phi_2(y) + 2\pi m_I/T$, and furthermore for any $x$ if we define $O(x)$ as the connected component of $\Lambda_n\backslash I$ containing $x$, we have that Same conditional independence property of islands in this setting?
	\begin{equation}
	\P(\diam (O(x))\geq L)\leq \exp(-\varpi(T) L).
	\end{equation}
	
	Let us note that the same argument as in Subsection \ref{SS.Conditional variance 0-boundary} together with the estimate of Proposition \ref{P.1-point function} implies that for any smooth function $f:[-1,1]^2\mapsto \R$ we have that
	\begin{equation*}\label{E.Difference bound 0-boundary}
	\E\left[\left \langle \phi_1-\phi_2-\frac{2\pi m_I}{T},f \right \rangle^2 \right]\leq C \|f\|_\infty^2 n^2\log(n).
	\end{equation*}
	Let us, now, note that for any continuous function  $f$ with $\int f=0$, we have that $\hat f:=|\Lambda_n|^{-1}\langle f,1 \rangle=\|f\|_\infty o(1)$. 
%\margin{\avelio{Not sure about the $\|f\|_{\infty}$, $\|f\|_{C^1}$ is clearly enough}}
	Thus, defining $\tilde f$ as $f-\hat f$ and noting that $\langle m_I, \tilde f\rangle=0$ we have that
	\begin{align}
	\nonumber\E\left[ \left \langle \phi_1-\phi_2,f \right \rangle^2\right]&\leq  \E\left[\left \langle \phi_1-\phi_2+\frac{2\pi m_I}{T},\tilde f \right \rangle^2 \right] + \left( \frac{\langle f,1 \rangle}{|\Lambda_n|}\right)^2 \E\left[\left \langle \phi_1-\phi_2,1 \right \rangle^2 \right]\\
	&\leq C\|f\|_\infty^2 \log(n)n^2 + C\|f\|_\infty o(1)n^4.\label{E.Variance free}
	\end{align}
	Which finishes the proof

	%Take $\hat \phi$ a free boundary GFF and recall that, thanks to Remark \ref{R.free with 0-boundary}, $\phi(y)=\hat \phi(y) - \hat \phi(0)$ is a GFF with $0$-boundary in $0\in V$. Furthermore, let us note that the sigma-algebra generated by $\exp(iT\phi)$ is contained in the sigma-algebra generated by $\exp(iT\hat \phi)$\footnote{These sigma-alebras are not equal.}. This implies that to prove \eqref{E.localisation by variance} it suffices to prove that
	%\begin{align}\label{E.Variance free-boundary}
	%\E\left[\Var{(\hat \phi,f)\mid \exp(iT\phi)} \right]=o(n^4)
	%\end{align}
	%as $n\to \infty$. This is due to the monotonicity of the left-side term with respect to the filtrations \footnote{This can be proven in many ways, either be seeing the variance as a minimisation problem or by using \eqref{E.conditional decomposition}} 
	
	%To obtain \eqref{E.Variance free-boundary}, let us take $(\phi_1,\phi_2)$ a pair of GFF with $0$-boundary condition in $\{x\}$ coupled as in Definition \ref{D.phis}. Of course, we call
	%\[\hat \phi_i(x):= \hat \phi_i(x) - \frac{1}{|\Lambda_n|}\sum_{y\in V}\phi_i(y) \]
	%the free-boundary GFF associated to it.

	%Equation \eqref{E.Difference bound 0-boundary} with $f=1/|\Lambda_n|$ implies that for any $x\in I$ 
	%\begin{equation}
	%\label{E.Close in I}\E\left[ (\hat\phi_1(x)-\hat \phi_2(x))^2\right]\leq C\frac{\log(n)}{n^2}.
	%\end{equation}
	%This, again by the same argument as Subsection \ref{SS.Conditional variance 0-boundary}, implies \ref{E.Variance free-boundary}.

	\subsection{The conditional variance at a given point is bounded.} 
	In this subsection, we are going to improve the result of \eqref{E.Final variance} for the case $f=\1_{x}$. 
	\begin{proposition}\label{P.1-point function}
		Let $\phi_1$ and $\phi_2$ be two zero boundary (or free-boundary) GFF coupled as in Definition \ref{D.phis}. Then, for all $T$ small enough there exists $K>0$ such that for all $n\in \N$ and for all $x,y \in \Lambda_n$
		\begin{align}\label{E.variance x}
		&\E\left[(\phi_1-\phi_2)^2(x) \right]\leq K, \ \ \ \text{and}\\
		&\label{E.2-point_exp}\E\left[(\phi_1-\phi_2)(x) (\phi_1-\phi_2)(y) \1_{O(x)=O(y)}  \right]\leq e^{-Kd_{\Lambda_n}(x,y)} 
		\end{align}
	\end{proposition}

	\begin{proof}
		We start by proving \eqref{E.variance x} for $\phi$ a $0$-boundary GFF as in Subsection \ref{SS.Conditional variance 0-boundary}. Let $\gamma$ be a horizontal edge path connecting $\partial [-1,1]^2$ to $x$ in $\Lambda_n$. We say that the edge $e$ belongs to $\gamma \cap O(x)$ if $e\in \gamma$ and $e\cap O(x)\neq \emptyset$. We then have that
		\begin{equation}\label{E.Telescopic sum}
		(\phi_1-\phi_2)(x)= \sum_{e \in \gamma \cap O(x)} \nabla(\phi_1-\phi_2)(e).
		\end{equation}
		Thus,
		\[(\phi_1-\phi_2)^2(x)\leq \sum_{e,e'\in E}\nabla(\phi_1-\phi_2)(e)\nabla(\phi_1-\phi_2)(e')\1_{e,e'\in O(x)\cap \gamma}. \]
		We can now upper bound 	$\E\left[ (\phi_1-\phi_2)^2(x)\right]$ by
		\begin{align*}
		&\sum_{e,e'\in E}\E\left[\nabla(\phi_1-\phi_2)(e)\nabla(\phi_1-\phi_2)(e')\1_{e,e'\in O(x)\cap \gamma}  \right] \\
		&\hspace{0.05\textwidth}\leq K\sup_{e}\E\left[ (\nabla(\phi_1-\phi_2)(e))^4\right]^{1/2}\sum_{e,e' \in \gamma} \P\left[e,e' \in O(x) \right]^{1/2}.
		\end{align*}
		We conclude \eqref{E.variance x} by first noting that $\var(\nabla(\phi_1-\phi_2)(e))\leq 4$ thanks to Proposition \ref{P.dphi+d*phi*}, and by the fact that
		\begin{equation*}
		\P\left[e,e' \in O(x) \right]\leq \exp(-\omega(T) \max\{d_{\Lambda_n}(e,x),d_{\Lambda_n}(e',x) \}).
		\end{equation*}
		
		We, now, prove \eqref{E.variance x} in the free-boundary case with $0$ value in $z$. In this case, one needs to take an edge path $\gamma$ going from $x$ to $z$  that only makes one turn (so that $\sum_{e,e' \in \gamma} \P[e,e' \in O(x)]^{1/2}$ is bounded). The same argument as before shows that,
		\begin{align*}
		\E\left[(\phi_1-\phi_2)^2(x) \right] \leq C + 2\E\left[m_I^2 \right],
		\end{align*}
		where $m_I:=0$ if $I=\emptyset$, and if $I\neq \emptyset$, $m_I:=(\phi_1-\phi_2)(z)/(2\pi T)$ at a point $z\in I$ (recall that this value is a constant in $I$). 
		
		To bound the variance of $n$, we note that we can take an edge path $\gamma$ starting from $z$ such that it always hit $I$, when $I\neq\emptyset$, and that it only makes $4$ turns (again so that $\sum_{e,e' \in \gamma} \P[e,e' \in O(z)]^{1/2}$ is bounded). By the same argument as before, one sees that
		\begin{align*}
		\E\left[m_I^2 \right]\leq C. 
		\end{align*}
			
	We, now, prove \eqref{E.2-point_exp}. Note that this directly follows from showing that
	\begin{equation}\label{E.variance y}
	\E\left[(\phi_1-\phi_2)^2(y)\1_{O(x)=O(y)} \right]\leq K \exp\left (-\frac{\varpi(T)}{2}d_{\Lambda_n}(x,y)\right ).
	\end{equation}
	This can be done, exactly as before by choosing an appropriate path $\gamma$. 
	\end{proof}
	
	\begin{remark} \label{r.differece_bc}
		Proposition \ref{P.1-point function} hides in plain sight an important fact. There is a difference regarding the behaviour of the (conditional) correlation function between the two different types of boundary condition we study. 
		
		Let us be more precise, in the case of the zero boundary GFF, one has that
		\begin{equation*}
		\E\left[(\phi_1-\phi_2)(x) (\phi_1-\phi_2)(y) \1_{O(x)=O(y)}  \right] = \E\left[(\phi_1-\phi_2)(x) (\phi_1-\phi_2)(y)  \right],
		\end{equation*}
		which proves \eqref{e.2-points}. However, in the case of free boundary conditions, one has that
		\begin{align*}
&\E\left[(\phi_1-\phi_2)(x) (\phi_1-\phi_2)(y)  \right]\\
&\hspace{0.2\textwidth}= \E\left[(\phi_1-\phi_2)(x) (\phi_1-\phi_2)(y) \1_{O(x)=O(y)}  \right] + \E\left[m_I^2 \1_{O(x)\neq O(y)} \right].
		\end{align*}
		As we do not expect that $\E\left[m_I^2 \right]$ goes to $0$ as $n\to \infty$, one can see that the (conditional) correlations do not decrease to $0$ as $d_{\Lambda_n}(x,y) \to \infty$. However, it is also interesting to note that these correlations do decay exponentially to $0$ if, we condition, not only on $e^{iT\phi}$, but also on the value of $m_I$. In fact, this seems to be closely related to the large-scale correlations which arise for Coulomb gases in $2d$ with free boundary conditions, see for example \cite{federbush1985}.
	\end{remark}

	Note that Proposition \ref{P.1-point function}, improves the result of \eqref{E.localisation by variance} and \eqref{E.Final variance}.
	\begin{proposition}\label{P.Var O(n^2)}
		 For $T$ small enough one has that
			\begin{align}\label{}
			\E\left[\Var{\frac {1} {n^2} \langle{\phi,f}\rangle \md e^{iT h}} \right]\leq K \frac {\|f\|_\infty^2} {n^2}. 
			\end{align}
		\end{proposition}
		\begin{proof}
			The proof of the Proposition follows the same lines as the proof of \eqref{E.Final variance}. The main difference is that we now use \eqref{E.variance y}.
%			The proof of \eqref{E.variance y}, is close to that of \eqref{E.variance x}. The only difference  is that one needs to choose $\gamma$ an edge path going from $\partial[-1,1]^2$ to $y$ in such a way that the distance from $x$ to $\gamma(\cdot)$ is a decreasing function (in particular the distance from $\gamma$ to $x$ is equal to that of $y$ to $x$).
%The proof of \eqref{E.variance y}, is close to that of \eqref{E.variance x}. The only different point is that one need to choose $\gamma$ a straight line going from $\partial[-1,1]^2$ to $y$ such that the distance from $x$ to $\gamma(\cdot)$ is a decreasing function, in particular the distance from $\gamma$ to $x$ is equal to that of $y$ to $x$.			
		\end{proof}

\section{Delocalisation regime}\label{s.deloc}

We start by proving the roughening transition for generalized integer-valued fields (Theorem \ref{th.IVgff}) and then, as a corollary, extract the delocalisation regime for our statistical reconstruction problem. 

\subsection{Proof of Theorem \ref{th.IVgff}.}\label{ss.proof12}

In this proof, we focus on the case of Free boundary conditions (as in \cite{FS,RonFS}), however following the Appendix D. from \cite{FS} or the recent \cite{Wirth} (see Remark \ref{R.Wirth}), our proof works in the exact same way in the Dirichlet case. 

Recall from Subsection \ref{ss.BackKT} and from (1.13) in \cite{RonFS} the following series expansion for the Laplace transform of the discrete GFFs with periodic weights $\lambda_\Lambda= (\lambda_j)_{j\in \Lambda}$ (we assume the same hypothesis as in Theorem 1.6 from \cite{RonFS}) %\margin{??I should state this one in the intro for self-consistency. \avelio{I think it would make sense}}, 

\begin{align}\label{}
&\EFK{\beta,\Lambda,\lambda_\Lambda,v}{}{e^{\<{\phi,f}}}  \nn\\
&=
e^{\frac 1 {2\beta} \<{f, - \Delta^{-1} f}} 
\frac{\sum_{\calN \in \calF} c_\calN \EFK{\beta,\Lambda}{\GFF}{ \prod_{\rho\in \calN} [1+z(\beta,\rho, \calN) \cos(\<{\phi,\bar\rho}+\<{\sigma,\rho})] }}
{\sum_{\calN \in \calF} c_\calN \int \prod_{\rho\in \calN} [1+z(\beta,\rho, \calN) \cos(\<{\phi,\bar\rho})] d\mu_{\beta,\Lambda,v}^\GFF(\phi)} \nn \\
&=
e^{\frac 1 {2\beta} \<{f, - \Delta^{-1} f}} 
\frac
{\sum_{\calN\in\calF} c_\calN  Z_\calN(\sigma)}
{\sum_{\calN\in\calF} c_\calN  Z_\calN(0)}.
\end{align}

We will denote by $\E_{\beta,\Lambda,\lambda_\Lambda,v}^\ba$ or $\mu_{\beta,\Lambda,\lambda_\Lambda,v}^\ba$  the discrete GFF whose periodic weights are shifted by environment $\ba$, namely:
\begin{align}\label{}
d \mu_{\beta,\Lambda,\lambda_\Lambda,v}^\ba(\phi):=
\frac
{1}
{Z_{\beta,\Lambda,\lambda_\Lambda,v}^\ba}
\prod_{j\in \Lambda} \lambda_j(\phi_j-a_j) d\mu_{\beta, \Lambda,v}^\GFF(\phi)
\end{align}

The shift by $\ba$ easily translates into the following expression for the Laplace transform under $\mu_{\beta,\Lambda,\lambda_\Lambda,v}^\ba$:
\begin{align}\label{e.Eqba}
&\EFK{\beta,\Lambda,\lambda_\Lambda,v}{\ba}{e^{\<{\phi,f}}}  \nn\\
&=
e^{\frac 1 {2\beta} \<{f, - \Delta^{-1} f}} 
\frac{\sum_{\calN \in \calF} c_\calN \int \prod_{\rho\in \calN} [1+z(\beta,\rho, \calN) \cos(\<{\phi,\bar\rho}+\<{\sigma-\ba,\rho})] d\mu_{\beta,\Lambda,v}^\GFF(\phi)}
{\sum_{\calN \in \calF} c_\calN \int \prod_{\rho\in \calN} [1+z(\beta,\rho, \calN) \cos(\<{\phi,\bar\rho} - \<{\ba,\rho})]d\mu_{\beta,\Lambda,v}^\GFF(\phi)} \nn \\
&=
e^{\frac 1 {2\beta} \<{f, - \Delta^{-1} f}} 
\frac
{\sum_{\calN\in\calF} c_\calN  Z_\calN(\sigma-\ba)}
{\sum_{\calN\in\calF} c_\calN  Z_\calN(-\ba)}
\end{align}

As the shift $\ba$ is fixed once and for all in this proof, let us introduce the \textbf{shifted partition functions} $\{Z_\calN^\ba(\sigma)\}_{\calN,\sigma}$. For any $\sigma : \Lambda\to \R$ and any {\em collection of charges} $\calN\in\calF$, 
\begin{align}\label{}
Z_\calN^\ba(\sigma):= Z_\calN(\sigma-\ba) = 
\int \prod_{\rho\in \calN} [1+z(\beta,\rho, \calN) \cos(\<{\phi,\bar\rho} + \<{\sigma - \ba,\rho})] d\mu_{\beta,\Lambda,v}^\GFF(\phi)
\end{align}

Following the same analysis as in Section 3 from \cite{RonFS} (or also Section 5 in \cite{FS}), we obtain the following lower bound on the ratio of partition functions,
\begin{align}\label{e.LBba}
\frac{Z_\calN^\ba(\sigma)}
{Z_\calN^\ba(0)}
\geq &
\exp\big[-D_4 \sum_{\rho\in \calN} |z(\beta,\rho,\calN)|\<{\sigma,\rho}^2\big] \\
& \times 
\int \frac
{e^{S(\calN,\ba,\phi)}} 
{Z_{\calN}^\ba(0)}
\prod_{\rho\in\calN}[1+z(\beta,\rho,\calN) \cos(\<{\phi,\bar \rho}-\<{\ba,\rho})]
d\mu_{\beta,\Lambda,v}^\GFF(\phi)\,, \nn
\end{align}
where 
\begin{align}\label{e.Sba}
S(\calN,\ba,\phi):=
-\sum_{\rho\in\calN} \frac
{z(\beta,\rho,\calN)\sin(\<{\phi,\bar \rho}-\<{\ba,\rho})\sin(\<{\sigma,\rho})}
{1+z(\beta,\rho,\calN)\cos(\<{\phi,\bar \rho}-\<{\ba,\rho})}
\end{align}
As mentioned in Subsection \ref{ss.BackKT},  one major observation in \cite{FS} is that $S(\calN,\phi):=S(\calN, \ba\equiv 0,\phi)=-S(\calN,-\phi)$. Indeed this property together with the fact that 
the probability measure
\[
d\P_{\calN}(\phi):=\frac 1 {Z_{\calN}(0)}
\prod_{\rho\in\calN}[1+z(\beta,\rho,\calN) \cos(\<{\phi,\bar \rho}) 
d\mu_{\beta,\Lambda,v}^\GFF(\phi)
\]
is invariant under $\phi \mapsto -\phi$  avoids controlling terms such as $\exp(S(\calN, \phi))$ thanks to Jensen: %This is because Jensen's inequality readily implies (when $\ba\equiv 0$)
\begin{align*}\label{}
\frac{Z_\calN(\sigma)}
{Z_\calN(0)}
&\geq 
\exp\big[-D_4 \sum_{\rho\in \calN} |z(\beta,\rho,\calN)|\<{\sigma,\rho}^2\big] 
 \times \int
e^{S(\calN,\phi)} 
d\P_{\calN}(\phi)  \\
& \geq  
\exp\big[-D_4 \sum_{\rho\in \calN} |z(\beta,\rho,\calN)|\<{\sigma,\rho}^2\big] 
 \times 
\exp\big(\int S(\calN,\phi)  d\P_{\calN}(\phi)\big)  \\
& =  \exp\big[-D_4 \sum_{\rho\in \calN} |z(\beta,\rho,\calN)|\<{\sigma,\rho}^2\big] \,.
\end{align*}
Claim 3.2. in \cite{RonFS} then shows that when $\beta$ is sufficiently small, 
\begin{align}\label{e.3.2}
\frac{Z_\calN(\sigma)}
{Z_\calN(0)}
& \geq \exp\left( -\frac {\eps \beta} {2(1+\eps) } \sum_{j\sim l} (\sigma_j -\sigma_l)^2\right) 
= \exp( -\frac {\eps } {2(1+\eps)\beta}  \<{f, -\Delta^{-1} f})\,,
\end{align}
which thus ended the proof in \cite{FS,RonFS}.

\medskip

In our present setting, the functional $\phi \mapsto S(\calN,\ba,\phi)$ introduced in~\eqref{e.Sba} is no longer an odd functional of $\phi$. Furthermore, the Lower-bound~\eqref{e.LBba} suggests introducing the following $\ba$-reweighted probability measure 
\[
d\P^\ba_{\calN}(\phi):=\frac 1 {Z^\ba_{\calN}(0)}
\prod_{\rho\in\calN}[1+z(\beta,\rho,\calN) \cos(\<{\phi,\bar \rho}-\<{\ba,\rho}) 
d\mu_{\beta,\Lambda,v}^\GFF(\phi)
\]
which is no longer invariant under $\phi\mapsto -\phi$. This lack of symmetry does not allow us to rely on Jensen and we are left with analyzing the quantity 
\begin{align*}\label{}
\int e^{S(\calN,\ba,\phi)}d\P^\ba_{\calN}(\phi) 
\end{align*}

We will not succeed in controlling the full Laplace transform but will instead extract bounds on the first and second moments from the series expansion near $\alpha\sim 0$ of the Laplace transform
$\alpha \mapsto \EFK{\beta,\Lambda,\lambda_\Lambda,v}{\ba}{e^{\alpha\<{\phi,f}}}$. 

For any $\alpha\in\R$, we have (recall~\eqref{e.Eqba},~\eqref{e.LBba} and~\eqref{e.3.2}) the lower bound
\begin{align}\label{e.mainLB}
& \EFK{\beta,\Lambda,\lambda_\Lambda,v}{\ba}{e^{\purple{\alpha}\<{\phi,f}}} \\
& = e^{\frac {\purple{\alpha^2}} {2\beta} \<{f, - \Delta^{-1} f}} 
\frac
{\sum_{\calN\in\calF} c_\calN  Z_\calN^\ba(\purple{\alpha}\sigma)}
{\sum_{\calN\in\calF} c_\calN  Z_\calN^\ba(0)} \\ 
& \geq 
%\exp\big[-D_4 \sum_{\rho\in \calN} |z(\beta,\rho,\calN)|\<{\sigma,\rho}^2\big] \, 
e^{\frac {\purple{\alpha^2}} {2\beta} \<{f, - \Delta^{-1} f}} 
%\exp\left( -\frac {\eps \beta} {2(1+\eps)} \alpha^2 \sum_{j\sim l} (\sigma_j-\sigma_l)^2 \right)  
\frac
{
\sum_{\calN\in\calF} c_\calN  Z_\calN^\ba(0) e^{-\frac {\eps \purple{\alpha^2} } {2(1+\eps)\beta}  \<{f, -\Delta^{-1} f}} \, \int e^{S_\alpha(\calN,\ba,\phi)}d\P^\ba_{\calN}(\phi)  
}
{
\sum_{\calN\in\calF} c_\calN  Z_\calN^\ba(0) 
}\,,\nn
\end{align}
where now 
\begin{align}\label{e.Salpha}
S_{\purple \alpha}(\calN,\ba,\phi)&=
-\sum_{\rho\in\calN} \frac
{z(\beta,\rho,\calN)\sin(\<{\phi,\bar \rho}-\<{\ba,\rho})\sin(\purple{\alpha}\<{\sigma,\rho})}
{1+z(\beta,\rho,\calN)\cos(\<{\phi,\bar \rho}-\<{\ba,\rho})} \\
& = 
- \purple{\alpha} \sum_{\rho\in\calN} \frac
{z(\beta,\rho,\calN)\sin(\<{\phi,\bar \rho}-\<{\ba,\rho}) \<{\sigma,\rho}}
{1+z(\beta,\rho,\calN)\cos(\<{\phi,\bar \rho}-\<{\ba,\rho})}
+\purple{O(\alpha^3)}\,.
\end{align}
This Taylor expansion holds first because we are in the regime where $\beta$ can be chosen small enough so  that the denominators are uniformly $\geq 1/2$ (see \cite{FS,RonFS}), and second because our parameters $\Lambda, \beta$ etc. are fixed as $\alpha$ is going to zero.  %(as such sums etc. are finite). 
\medskip

\ni
\textbf{First order analysis.}
At first order in $\alpha$, we obtain combining~\eqref{e.mainLB} and~\eqref{e.Salpha} that for any $f : \Lambda \to \R$ and as $\alpha\to 0$, 
\begin{align*}\label{}
& 1 + \alpha \EFK{\beta,\Lambda,\lambda_\Lambda,v}{\ba}{\<{\phi,f}} +O(\alpha^2)  \\
&\geq (1+O(\alpha^2)) 
\frac
{
\sum_{\calN\in\calF} c_\calN  Z_\calN^\ba(0) \, \int \big[1 + S_\alpha(\calN,\ba,\phi) + O(\alpha^2)\big] d\P^\ba_{\calN}(\phi)  
}
{
\sum_{\calN\in\calF} c_\calN  Z_\calN^\ba(0) 
} \\
& = 1  - \alpha \frac
{
\sum_{\calN\in\calF} c_\calN  Z_\calN^\ba(0) \, \EFK{\calN}{\ba}{\sum_{\rho\in\calN} \frac
{z(\beta,\rho,\calN)\sin(\<{\phi,\bar \rho}-\<{\ba,\rho}) \<{\sigma,\rho}}
{1+z(\beta,\rho,\calN)\cos(\<{\phi,\bar \rho}-\<{\ba,\rho})}} 
}
{
\sum_{\calN\in\calF} c_\calN  Z_\calN^\ba(0) 
} +O(\alpha^2)
\end{align*}

In particular, identifying order 1 terms (and recalling that $\sigma:=\frac 1 \beta (-\Delta)^{-1} f$, see~\eqref{e.sigma}), we thus have for any  $f: \Lambda \to \R$, 
\begin{align*}\label{}
\EFK{\beta,\Lambda,\lambda_\Lambda,v}{\ba}{\<{\phi,f}} \geq 
-\frac
{
\sum_{\calN\in\calF} c_\calN  Z_\calN^\ba(0) \, \EFK{\calN}{\ba}{\sum_{\rho\in\calN} \frac
{z(\beta,\rho,\calN)\sin(\<{\phi,\bar \rho}-\<{\ba,\rho}) \<{\frac 1 \beta (-\Delta)^{-1}f,\rho}}
{1+z(\beta,\rho,\calN)\cos(\<{\phi,\bar \rho}-\<{\ba,\rho})}} 
}
{
\sum_{\calN\in\calF} c_\calN  Z_\calN^\ba(0) 
} 
\end{align*}

The key observation at this stage is that for each collection of charges $\calN$, the functional 
\[
f\mapsto \hat S(\calN,\ba,\phi,f):= -\, \sum_{\rho\in\calN} \frac
{z(\beta,\rho,\calN)\sin(\<{\phi,\bar \rho}-\<{\ba,\rho}) \<{\frac 1 \beta (-\Delta)^{-1}f,\rho}}
{1+z(\beta,\rho,\calN)\cos(\<{\phi,\bar \rho}-\<{\ba,\rho})}
\]
is \textbf{linear} in $f$. Obviously the functional $f\mapsto  \EFK{\beta,\Lambda,\lambda_\Lambda,v}{\ba}{\<{\phi,f}}$ is linear as well. Now by using this linearity and plugging $-f$ into the above inequality, we obtain a rather surprising exact expression for the mean value of $\<{\phi,f}$ under the measure $\mu_{\beta,\Lambda,\lambda_\Lambda,v}^\ba$. We state this exact identity as a proposition below and we call it  {\em modular invariance identity} for reasons which will be explained in Appendix \ref{a}.

\begin{proposition}[Modular invariance identity]%[{\em Magical formula}]
\label{c.magical}
For any function $f$ and any weights $\lambda_\Lambda=(\lambda_i)_{i\in \Lambda}$ satisfying the same hypothesis as in (5.35) in \cite{FS} (or equivalently (1.9) in \cite{RonFS}), we have 
\begin{align}\label{e.magical}
\EFK{\beta,\Lambda,\lambda_\Lambda,v}{\ba}{\<{\phi,f}} 
& = 
-\frac
{
\sum_{\calN\in\calF} c_\calN  Z_\calN^\ba(0) \, \EFK{\calN}{\ba}{\sum_{\rho\in\calN} \frac
{z(\beta,\rho,\calN)\sin(\<{\phi,\bar \rho}-\<{\ba,\rho}) \<{\frac 1 \beta (-\Delta)^{-1}f,\rho}}
{1+z(\beta,\rho,\calN)\cos(\<{\phi,\bar \rho}-\<{\ba,\rho})}} 
}
{
\sum_{\calN\in\calF} c_\calN  Z_\calN^\ba(0) 
} \nn
\\
& = 
\frac{
\sum_{\calN\in\calF} c_\calN  Z_\calN^\ba(0) \, \EFK{\calN}{\ba}{\hat S(\calN,\ba,\phi,f)} 
}
{
\sum_{\calN\in\calF} c_\calN  Z_\calN^\ba(0) 
}.
%\\
%& = \blue{
%- 
%\frac{
%\sum_{\calN\in\calF} c_\calN  \int 
% \left( 
% \sum_{\rho\in\calN} \frac
% { z(\beta,\rho,\calN) \sin(\<{\phi,\bar \rho}-\<{\ba,\rho})\<{g,\rho}
% }
% {
% 1+z(\beta,\rho,\calN) \cos(\<{\phi,\bar \rho}-\<{\ba,\rho})
% }
% \right) \times
%\prod_{\rho\in\calN}
%[1+z(\beta,\rho,\calN) \cos(\<{\phi,\bar \rho}-\<{\ba,\rho})]
%d\mu_{\beta,\Lambda,v}^\GFF(\phi)
%}
%{
%\sum_{\calN\in\calF} c_\calN  Z_\calN^\ba(0) 
%} 
%}\nn
\end{align}
%\margin{A effacer plus tard, juste pour la voir une fois??}
\end{proposition}

\begin{remark}\label{r.magical}
%\margin{C: you changed Remarrk style to Italic ?? } 
This exact identity, as we shall see below, is a key step in our proof. Because it is so central and since it does not look like anything familiar, we added Appendix \ref{a} to give a longer but more natural second derivation of this identity. It should not come as a surprise that our second derivation is longer as the above one relies in fact on several key parts of the proof of Fröhlich-Spencer \cite{FS}. Appendix \ref{a} gives a complementary  interpretation/explanation of the origin of such an identity. In particular in Appendix \ref{a}, we shall view the shift vector $\ba=\{a_x\}_{x\in \Lambda}$ as an exterior magnetic field and we will also explain why we call this identity ``modular invariance'' due to a relationship with the functional equation for Riemann-theta functions. 
\end{remark}

\ni
\textbf{Second order analysis.}
The above identity for the first moment will be instrumental in bounding from below the desired second moment as we shall now see.

Again by combining~\eqref{e.mainLB} and~\eqref{e.Salpha}, we find that %\margin{I change the notation so it fits better in the line} 
\begin{align*}\label{}
& 1 + \alpha \EFK{\beta,\Lambda,\lambda_\Lambda,v}{\ba}{\<{\phi,f}} + \frac 1 2 \alpha^2  
\EFK{\beta,\Lambda,\lambda_\Lambda,v}{\ba}{\<{\phi,f}^2} + O(\alpha^3) \\
& \geq 
%\exp\big[-D_4 \sum_{\rho\in \calN} |z(\beta,\rho,\calN)|\<{\sigma,\rho}^2\big] \, 
e^{\frac {\purple{\alpha^2}\<{f, - \Delta^{-1} f}} {2(1+\eps)\beta} } 
%\exp\left( -\frac {\eps \beta} {2(1+\eps)} \alpha^2 \sum_{j\sim l} (\sigma_j-\sigma_l)^2 \right)  
\frac
{
\sum_{\calN\in\calF} c_\calN  Z_\calN^\ba(0)   \E^\ba_{\calN}\left[  e^{S_\alpha(\calN,\ba,\phi)}\right] 
}
{
\sum_{\calN\in\calF} c_\calN  Z_\calN^\ba(0) 
} \nn \\ 
&\geq 
e^{\frac {\purple{\alpha^2}\<{f, - \Delta^{-1} f}} {2(1+\eps)\beta} } 
\frac
{
\sum_{\calN\in\calF} c_\calN  Z_\calN^\ba(0)   \E^\ba_{\calN}\left[  1 + S_\alpha(\calN,\ba,\phi) + \frac 1 2 [S_\alpha(\calN,\ba,\phi)]^2 + O(\alpha^3)\right]
}
{
\sum_{\calN\in\calF} c_\calN  Z_\calN^\ba(0) 
}  \\
&\geq 
e^{\frac {\purple{\alpha^2}\<{f, - \Delta^{-1} f}} {2(1+\eps)\beta} } 
\frac
{
\sum_{\calN\in\calF} c_\calN  Z_\calN^\ba(0)   \E^\ba_{\calN}\left [1 -  \purple{\alpha} \hat S(\calN,\ba,\phi,f)  + \frac {\purple{\alpha^2}} 2   [\hat S(\calN,\ba,\phi,f)]^2 + O(\alpha^3) \right ]   
}
{
\sum_{\calN\in\calF} c_\calN  Z_\calN^\ba(0) 
} \\
& = 1 - \purple{\alpha} \frac{
\sum_{\calN\in\calF} c_\calN  Z_\calN^\ba(0) \, \EFK{\calN}{\ba}{\hat S(\calN,\ba,\phi,f)} 
}
{
\sum_{\calN\in\calF} c_\calN  Z_\calN^\ba(0) 
} \\
&\,\, + \frac{\purple{\alpha^2}} 2  \left[ \frac 1 {(1+\eps)\beta} \<{f, -\Delta^{-1} f}  +
\frac{
\sum_{\calN\in\calF} c_\calN  Z_\calN^\ba(0) \, \EFK{\calN}{\ba}{[\hat S(\calN,\ba,\phi,f)]^2} 
}
{
\sum_{\calN\in\calF} c_\calN  Z_\calN^\ba(0) 
} 
\right] +O(\alpha^3)\,.
\end{align*}
The first order term are equal by Proposition \ref{c.magical} and from the second order terms, we extract the following lower bound
\begin{align}\label{}
\EFK{\beta,\Lambda,\lambda_\Lambda,v}{\ba}{\<{\phi,f}^2}
& \geq
\frac 1 {(1+\eps)\beta} \<{f, -\Delta^{-1} f}  +
\frac{
\sum_{\calN\in\calF} c_\calN  Z_\calN^\ba(0) \, \EFK{\calN}{\ba}{[\hat S(\calN,\ba,\phi,f)]^2} 
}
{
\sum_{\calN\in\calF} c_\calN  Z_\calN^\ba(0) 
}  \\
%& \geq 
%\frac 1 {(1+\eps)\beta} \<{f, -\Delta^{-1} f}  +
%\frac{
%\sum_{\calN\in\calF} c_\calN  Z_\calN^\ba(0) \, \EFK{\calN}{\ba}{\hat S(\calN,\ba,\phi,f)}^2 
%}
%{
%\sum_{\calN\in\calF} c_\calN  Z_\calN^\ba(0) 
%}  \\
& \geq
\frac 1 {(1+\eps)\beta} \<{f, -\Delta^{-1} f}  +
\left(\frac{
\sum_{\calN\in\calF} c_\calN  Z_\calN^\ba(0) \, \EFK{\calN}{\ba}{\hat S(\calN,\ba,\phi,f)} 
}
{
\sum_{\calN\in\calF} c_\calN  Z_\calN^\ba(0) 
} 
\right)^2 \\
& = \frac 1 {(1+\eps)\beta} \<{f, -\Delta^{-1} f}  + \EFK{\beta,\Lambda,\lambda_\Lambda,v}{\ba}{\<{\phi,f}}^2 
\end{align}
first by applying Cauchy-Schwarz inequality to a  suitable probability measure on the coupling $(\calN,\phi)$.  And then we used Proposition  \ref{c.magical}  for the last equality, i.e. the modular invariance identity ~\eqref{e.magical}.
This ends our proof. \qed

\subsection{Non-recovery phase ($T>T^+_{rec}$).}

In this subsection, we complete the non-recovery phases of Theorem \ref{T.0-boundary theorem} and \ref{T.free-boundary theorem}.

As in Definition \ref{D.phis}, let $(\phi_1,\phi_2)$ be two conditionally independent instances of $\phi$ given  $\phi \Mod{ \frac{2\pi}{T}}= \frac {2\pi} {T}\ba$.
By Lemma \ref{L.aIV}, the law of $(\phi_1,\phi_2)$ is given by $\frac{2\pi}{T}(\psi_1,\psi_2)$ where $\psi_1,\psi_2$ are independently sampled according to $\P_{\beta_T= 2\pi^2/T^2}^{\ba,\IV}$.

%where $\psi_1$ and $\psi_2$ have the law of two independent IV-GFF with inverse temperature equal to $2\pi^2/T^2$. 
%Let $(\phi_1,\phi_2)$ be two-GFF coupled as in Definition \ref{D.phis}. Recall from Lemma \ref{L.aIV} that conditionally on $\phi \Mod{ \frac{2\pi}{T}}=\ba$, the law of $\frac{2\pi}{T}(\psi_1,\psi_2)$ 
Thanks to this, we have that for any continuous function $f:[-1,1]^2\mapsto \R$
\begin{align}
\label{E.Tower}\E\left[\langle \phi_1-\phi_2,f \rangle ^2 \right]&=\left( \frac{2\pi}{T}\right)^2\E\left[\E^{IV,\ba}_{\frac{(2\pi)^2}{T^2},\Lambda}\left[\langle\psi_1-\psi_2,f \rangle^2 \right]  \right]\\
\nonumber&\geq \frac{1}{1+\epsilon}\langle f, (-\Delta)^{-1} f\rangle.
\end{align}
Furthermore, let us note that because both $\phi_1$ and $\phi_2$ are two GFF we have that
\begin{align*}
\E\left[\langle \phi_1-\phi_2,f \rangle ^4 \right]\leq 4\E\left[\langle \phi_1,f\rangle^2 \right]^{2}= 4 \langle f, (-\Delta)^{-1} f\rangle^{2}.
\end{align*}

Therefore,  using Paley-Zygmund inequality we have that
\begin{align*}
\P\left(\langle \phi_1-\phi_2,f \rangle^2\geq\frac 1 2\langle f, (-\Delta)^{-1}f\rangle \right)\geq \frac{1}{16(1+\epsilon)}>2^{-5}. 
\end{align*}

Now, we use that for any deterministic function $F$ depending only on $\exp(iT\phi)$, we have that $F(\exp(iT\phi_1))=F(\exp(iT\phi_2))$. Using this we can compute
\begin{align}
2^{-5}&\leq \P\left(\langle \phi_1-\phi_2,f \rangle^2\geq\frac 1 2\langle f, (-\Delta)^{-1}f\rangle \right)\\
&\leq\P\left[\langle F(\exp iT\phi_1)-\phi_i,f \rangle^2\geq \frac{1}{8}\langle f, (\Delta)^{-1}f \rangle, \text{ for some } i\in \{1,2\} \right]\\
&\leq 2\P\left[\langle F(\exp iT\phi_1)-\phi_1,f \rangle^2\geq \frac{1}{8}\langle f, (\Delta)^{-1}f \rangle \right]. 
\end{align}

We conclude by noting that for any continuous non-zero function $f$, we  have that  $\langle f, (-\Delta)^{-1}f \rangle \geq Cn^4$.
%This can be seen, for example, by noting that the discrete GFF converges to the continuum one \eqref{E.Convergence GFF}. 

To finish, let us show \eqref{E.1-point zero boundary}. We start by noting that
\begin{align*}
\E\left[(\phi_1(x)-\phi_2(x))^2 \right]&=\left( \frac{2\pi}{T}\right)^2\E\left[\E^{IV,\ba}_{\frac{(2\pi)^2}{T^2},\Lambda}\left[ (\psi_1(x)-\psi_2(x))^2\right]\right]\\
&\geq \frac{1}{(1+\epsilon)^2}G(x,x)\\
&\geq 4c(T,x)\log(n).
\end{align*}

We now see that
\begin{align*}
2\E\left[(\phi_1(x)-F(\phi_1)(x))^2 \right]&=\E\left[(\phi_1(x)-F(\phi_1)(x))^2+(\phi_2(x)-F(\phi_1)(x))^2 \right] \\
&\geq \frac{1}{2}\E\left[(\phi_1(x)-\phi_2(x))^2 \right]\\
&\geq 2c(T,x)\log(n). 
\end{align*}\qed

\ni
We now complete this Section by proving 
Corollary \ref{C.Convergence}. 
%\red{Even tough only half of it deals with delocalisation,}
\subsection{Proof of Corollary \ref{C.Convergence}.} %\margin{\avelio{I don't know where to put this.}}
Let us take $\phi_n\to \Phi$ in probability for the topology of the space of generalised functions. Let us now analyze the two regimes $T\ll 1$ and $T\gg 1$.
\subsubsection{Small $T$.} Let us note that thanks to part (1) of Proposition \ref{T.0-boundary theorem}, we have that for any smooth function $f$ (with $0$-mean if we are in the free boundary case). We have that %\margin{Issue with $1/n^2$ below}
\begin{align*}
\frac{1}{n^2}\langle F_T(e^{iT\phi_n}), f \rangle&= \langle \phi_n, f \rangle + \langle  F_T(e^{iT\phi_n})-\phi_n, f \rangle \to (\Phi,f).
\end{align*}
From this we see that $F_T(e^{iT\phi_n})$ also converges in probability to $\Phi$.

\subsubsection{Big $T$.} Let us reason by contradiction and assume such a function $F$ exists. Let $(\phi^{(n)}_1,\phi^{(n)}_2)$ be two GFF coupled as in Definition \ref{D.phis}, we see that in this case 
\begin{equation}\label{E.Equality of F}
F(\exp(iT\phi^{(n)}_1))=F(\exp(iT\phi^{(n)}_2)).
\end{equation}
Because both $\phi_1^{(n)}$ and $\phi_2^{(n)}$ have the law of a GFF in $\Lambda_n$ we see that the pair $(\phi_1^{(n)},\phi_2^{(n)})$ is tight. We can now take  a subsequence of  $(\phi_1^{(n)},\phi_2^{(n)})$, that we denote the same way such that 
\[(\phi_1^{(n)},\phi_2^{(n)})\to (\Phi_1,\Phi_2). \]
The second part of Theorems \ref{T.0-boundary theorem} and \ref{T.free-boundary theorem} imply that $\Phi_1\neq \Phi_2$. However, we have that (by the contradiction hypothesis)
\[F(\exp(iT\phi^{(n)}_1)) \to \Phi_1\neq \Phi_2 \leftarrow F(\exp(iT\phi^{(n)}_2)), \]
which is a contradiction with \eqref{E.Equality of F}.

\section{There is always information left}\label{S.Information left}

The objective of this section is to prove that for any $T>0$, $\exp(i T \phi)$ gives non-trivial (macroscopic)  information of $\phi$. More precisely, in this section we quantify how much information is preserved under the operation $\phi \mapsto \phi \Mod{\frac {2\pi} T}$.

Let us note that Theorem \ref{th.IVgff}, implies that for all possible values of $\exp(iT\phi)$ and for all $T$ big enough there exists $\epsilon(T)$ such that
\[\var\left[ \langle \phi, f\rangle \mid \exp(iT\phi) \right] \geq (1-\epsilon(T)) \E\left[\langle \phi, f \rangle^2 \right]. \]
At the same time, it is clear that
\[\E\left[\var\left[ \langle \phi, f\rangle \mid \exp(iT\phi) \right] \right]\leq \E\left[\langle \phi, f \rangle^2 \right].\]

Let us remark that it is not clear whether this $\epsilon(T)$ is a technical constant coming from the proof or whether it is telling us something meaningful about the model. In the following proposition we show that in the average case the existence of this $\epsilon(T)$ is not technical. In fact, in Remark \ref{R.Information} bellow we give an interpretation of its meaning. See also Remark \ref{r.Lab} for the link with the $\eps=\eps(T)$ correction in Fröhlich-Spencer. 

\begin{proposition}\label{P.Information}
	Let $T>0$ and $\phi$ be a GFF with either free or $0$ boundary condition in $\Lambda_n$. Then, there exists $\epsilon'(T)>0$ such that  
	\begin{align}\label{E.UB variance}
	\E\left[ \var\left[ \langle \phi, f\rangle^2 \mid \exp(iT\phi) \right] \right]\leq (1-\epsilon'(T)) \E\left[\langle \phi,f\rangle ^2 \right].
	\end{align}
	Furthermore, we have the following lower bound for $\epsilon'(T)$ when $T\gg 1$: 
%% OUR OLD (and WRONG) estimate :) 	
%	\begin{equation}\label{E.bound on epsilonp}
%	\epsilon'(T) \purple{\geq (1+o(1))} \frac{(2\pi)^2}{T^2}e^{-\frac{T^2}{(2\pi)^2}} = \beta_T e^{-\frac 1 {\beta_T}}\,,
%	\end{equation}
\begin{equation}\label{E.bound on epsilonp}
	\epsilon'(T) \geq (1+o(1)) 2 T^2 e^{- T^2}\,.
	%= (1+o(1)) 2 \frac{(2\pi)^2} {\beta_T}  e^{-\frac {(2\pi)^2} {\beta_T}}\,,
	\end{equation}
%\purple{where recall from Lemma \ref{L.aIV} that $\beta_T:=\frac{(2\pi)^2}{T^2}$.}	
\rm
\end{proposition}

\begin{remark}\label{R.Information}
	Proposition \ref{P.Information} should be interpreted in the following way:
	\begin{itemize}
		\item[] The field $\exp(iT\phi_n)$ gives non-trivial information on the GFF $\phi_n$.
	\end{itemize}
	This is because, if this were not the case we would have that for any continuous function $f:[-1,1]^2\mapsto \R$
	\begin{align*}
	n^{-4}\var\left[ \langle \phi_n, f\rangle \mid \exp(iT\phi_n) \right] &\to \iint_{[-1,1]^2\times [-1,1]^2} f(x)G(x,y)f(y) dx dy\\
	& = \lim_{n\to \infty} n^{-4}\E\left[\langle \phi_n,f\rangle  \right] ,
	\end{align*}
	where $G$ is the continuous Green's function in $[-1,1]^2$. In the Statistics world, we would say that Proposition \ref{P.Information} means that $\exp(iT\phi)$ explains at least $\epsilon'(T)$ of the variance of $\phi$.
\end{remark}

%\red{There exists a constant $\eps(T)>0$ such that  if}
%\begin{align*}
%F_n(\exp(iT\phi_n))(x):=\E\left[\phi_n(x)\mid \exp(iT\phi_n) \right],
%\end{align*}
%\red{then,} we have for any smooth function $f:[-1,1]^2\mapsto \R$
%\begin{align*}
%\E\left[\langle F_n(\exp(iT\phi_n) )-\phi_n,f \rangle^2 \right] \leq (1-\epsilon(T)) \E\left[\langle \phi_n,f\rangle ^2 \right] 
%\end{align*}

\begin{proof}
	Let us write $F(x):= \E\left[\phi(x)\mid \exp(iT\phi_n)\right]$ and $\phi=\phi_n$. We are going to prove that
	\begin{equation}\label{E.Still some information}
	\E\left[\langle F,f \rangle^2 \right]\geq \epsilon \E\left[\langle \phi,f\rangle ^2 \right].
	\end{equation}
	This suffices as
	\begin{align*}
	\E\left[ \var\left[ \langle \phi, f\rangle\mid \exp(iT \phi)\right]\right] &=\E\left[\langle F-\phi,f \rangle^2 \right]\\
	&= \E\left[\langle \phi,f\rangle ^2 \right]-\E\left[\langle F ,f \rangle^2 \right].
	\end{align*}
	
	To prove equation \eqref{E.Still some information}. Let us take $W=\nabla \phi + \zeta$ as in Proposition \ref{P.dphi+d*phi*}, let us bound the following 
	\begin{align*}
	\E\left[\langle F,f\rangle^2 \right]&= \E\left[\E\left[ \langle \phi, f\rangle\mid \exp(iT\phi)\right]^2   \right]\\
	&= \E\left[\E\left[ \E\left[\langle \phi, f\rangle\mid \exp(i T\phi),\exp(iT\zeta)\right]^2\mid \exp(iTW)\right]  \right] \\
	&\geq \E\left[\E\left[\langle \phi, f\rangle \mid \exp(iTW) \right]^2  \right],
	\end{align*}
	where we have used Cauchy-Schwartz and the fact that $\exp(iT\zeta)$ is independent of the pair $(\langle \phi,f \rangle, \exp(iT\phi))$. As such, to end it only remains to show that
	\begin{align}\label{E.Still some information W}
	\E\left[\E\left[\langle \phi, f\rangle \mid \exp(iTW) \right]^2 \right]\geq \epsilon \E\left[\langle \phi,f\rangle^2 \right].
	\end{align}
	
	Now, recall from Proposition \ref{P.dphi+d*phi*} that $\phi=\Delta^{-1} \nabla\cdot W$ and compute
	\begin{align*}
	\E\left[\langle \phi, f\rangle \mid \exp(i T W) \right]&= \E\left[\langle W, -\nabla \Delta^{-1} f \rangle\mid \exp(iT W) \right] \\
	&=-\frac{1}{2}\sum_{\overrightarrow e \in \overrightarrow E}\E\left[ W(\overrightarrow e)\mid \exp(i T W)\right]\nabla \Delta^{-1} f( \overrightarrow e)\\
	&= -\frac{1}{2} \sum_{\overrightarrow e \in \overrightarrow E} \E\left[W(\overrightarrow{e})\mid \exp(i TW(\overrightarrow{e}))  \right]\nabla \Delta^{-1} f( \overrightarrow e),
	\end{align*}
	where the last line comes from the independence between the value of $W$ in different edges, and the fact that $W(\overrightarrow e)=-W(\overleftarrow{ e})$. 
	The equality of this last line may seem innocent but it is the may reason why the problem simplifies when we work with the white noise. 
	
	Let us note that the random variable $\E\left[W(\overrightarrow{e})\mid \exp(i T W(\overrightarrow{e}))  \right]$ is centred and has the same law for all $\overrightarrow e$. Furthermore, it is independent for all $e\neq e'$. Let us define 
	\begin{align}\label{E.Variance 1-point}
	\sigma(T)=\Var{\E\left[W(\overrightarrow{e})\mid \exp(i T W(\overrightarrow{e}))  \right] }>0.
	\end{align} We can now compute
	\begin{align*}
	\E\left[\E\left[\langle \phi, f\rangle \mid \exp(iT W) \right]^2 \right]&= 2\pi\sigma(T) \langle \nabla \Delta^{-1} f( e),\nabla \Delta^{-1} f( e)\rangle \\
	&= 2\pi\sigma(T) \langle  f, -\Delta^{-1} f\rangle = \sigma(T)\E\left[\langle \phi, f\rangle \right] ,
	\end{align*}
	from where we obtain \eqref{E.UB variance}.
	
	To obtain \eqref{E.bound on epsilonp}, we remark that we set $\epsilon'(T)=\sigma(T)$. When $T\gg1$, one can get \eqref{E.bound on epsilonp} by estimating  \eqref{E.Variance 1-point} using  \eqref{e.theta}.  This is the subject of our next lemma.	
\end{proof}

\begin{lemma}\label{l.ST}
As $T\to \infty$, 
\begin{align}\label{}
\sigma(T)&=\Var{\E\left[W(\overrightarrow{e})\mid \exp(i T W(\overrightarrow{e}))  \right] }  \nn \\
& = 2  T^2 \,  e^{- T^2 } +o(e^{-  T^2})\,.
\end{align}
\end{lemma}

\begin{remark}\label{}
Equivalently, if $Z\sim \calN(0,\beta)$, then as $\beta \to \infty$, 
\begin{align}\label{}
\Var{\E\left[Z \mid  Z \Mod{1}  \right] } =  2 (2\pi)^2\, \beta^2 e^{- (2\pi)^2 \beta} + o(e^{-(2\pi)^2\beta})\,.
\end{align}
This straightforward rewriting of the lemma will happen to be useful in our coming work \cite{GS2}. 
\end{remark}

\begin{proof}
Let $Z\sim \calN(0,\beta_T^{-1})$ with $\beta_T:= \frac {(2\pi)^2}{T^2}$ as in Lemma \ref{L.aIV} so that $Z\overset{(d)}=\frac T {2\pi} W$.  
\begin{align}\label{e.ST2}
\sigma(T)&=\Var{\E\left[W(\overrightarrow{e})\mid \exp(i T W(\overrightarrow{e}))  \right] }  \nn \\
& = 
\Var{\E\left[W(\overrightarrow{e})\mid W(\overrightarrow{e}) \Mod{\frac {2\pi} T}  \right]} \nn \\
&= \beta_T \,  \Var{\E\left[Z \mid  Z \Mod{1}\right]} \nn \\
&= \beta_T \,  \Eb{\E\left[Z \mid  Z \Mod{1}\right]^2}\,.
\end{align}
Notice that 
\begin{align*}\label{}
\E\left[Z \mid  Z \Mod{1}=a\right]
&= \frac{\sum_{n\in \Z} \exp(-\frac {\beta_T} 2 (n+a)^2) \cdot (n+a)}
{\sum_{n\in \Z} \exp(-\frac {\beta_T} 2 (n+a)^2)}\,.
\end{align*}
As $\beta_T\to 0$, it will be convenient to rely on the Jacobi's identity~\eqref{e.theta} which plays the role of a {\em temperature inversion}.  Below, we start by slightly rewriting this identity via a straightforward change of variable, so that it matches with integer-valued field (as opposed to fields in $2\pi \Z$). The following three identities are equivalent
\begin{align}
&\frac
{
\sum_{n\in \Z} \exp(-\frac \beta 2 (2\pi n + 2\pi a)^2)\cdot(2\pi n +2\pi a)
}
{
\sum_{n\in \Z} \exp(-\frac \beta 2 (2\pi n + 2\pi a)^2)
}
=
\frac
{
\frac 1 \beta 
\sum_{q\in \Z}
e^{-\frac {q^2} {2 \beta}}
\sin(q \cdot 2\pi a) \cdot q
}
{
\sum_{q\in \Z}
e^{-\frac {q^2} {2\beta}}
\cos(q\cdot 2\pi a)
},\nonumber \\
&
\frac
{
\sum_{n\in \Z} \exp(-\frac {(2\pi)^2 \beta}  2 ( n +  a)^2)\cdot(n + a)
}
{
\sum_{n\in \Z} \exp(-\frac  {(2\pi)^2 \beta} 2 (n +  a)^2)
}
=
\frac
{
\frac 1 {2\pi \beta}  
\sum_{q\in \Z}
e^{-\frac {q^2} {2 \beta}}
\sin(q \cdot 2\pi a) \cdot q
}
{
\sum_{q\in \Z}
e^{-\frac {q^2} {2\beta}}
\cos(q\cdot 2\pi a)
},\nonumber \\
&\nonumber
\frac
{
\sum_{n\in \Z} \exp(-\frac {\beta_T}  2 ( n +  a)^2)\cdot(n + a)
}
{
\sum_{n\in \Z} \exp(-\frac  {\beta_T} 2 (n +  a)^2)
}
=
\frac
{
\frac {2\pi} { \beta_T}  
\sum_{q\in \Z}
e^{-\frac {q^2 (2\pi)^2} {2 \beta_T}}
\sin(q \cdot 2\pi a) \cdot q
}
{
\sum_{q\in \Z}
e^{-\frac {q^2 (2\pi)^2} {2\beta_T}}
\cos(q\cdot 2\pi a)
}.
\end{align}
This rewriting of ~\eqref{e.theta} implies the following useful expression for the conditional expectation
\begin{align*}\label{}
\E\left[Z \mid  Z \Mod{1}=a\right]
& = 
\frac
{
\frac {2\pi} { \beta_T}  
\sum_{q\in \Z}
e^{-\frac {q^2 (2\pi)^2} {2 \beta_T}}
\sin(q \cdot 2\pi a) \cdot q
}
{
\sum_{q\in \Z}
e^{-\frac {q^2 (2\pi)^2} {2\beta_T}}
\cos(q\cdot 2\pi a)
}\,.
\end{align*}

This implies readily
\begin{align*}\label{}
\sigma(T)&= \beta_T \,  \Eb{\E\left[Z \mid  Z \Mod{1}\right]^2} \\
& = \beta_T \frac{(2\pi)^2} {\beta_T^2} 
\Eb{\frac
{
\big(
2 \sin(2\pi a) e^{-\frac{(2\pi)^2}{2\beta_T}} 
+ 4 \sin(4\pi a) e^{-\frac{(2\pi)^2}{\beta_T}}
\big)^2
}
{
(1+2 \cos(2\pi a) e^{-\frac{(2\pi)^2}{2\beta_T}})^2
}
+o(e^{-\frac{(2\pi)^2}{\beta_T}})}\\
&= 
\frac{(2\pi)^2} {\beta_T} 
\Eb{4 \sin^2(2\pi a) e^{-\frac{(2\pi)^2}{\beta_T}} + o(e^{-\frac{(2\pi)^2}{\beta_T}})}\\
&= 2 \frac{(2\pi)^2} {\beta_T}  + o(e^{-\frac{(2\pi)^2}{\beta_T}}) = 2 T^2 e^{-T^2}+o(e^{-T^2})\,,
\end{align*}
where we relied on the convenient abuse of notation $a$ for the random variable $Z\Mod{1}$ throughout. 
\end{proof}

\begin{remark}\label{r.Lab}
	Proposition \ref{P.Information} is one of the reasons why this model is a laboratory for  IV-GFF especially with quenched disorder. In this case, it allows us to obtain explicit lower bounds on the $\epsilon(T)$- correction between the GFF and the integer-valued GFF with quenched disorder $\ba$ given by a GFF (at inverse temperature $\beta_T^{-1}$) modulo 1. We will discuss in more details such explicit bounds in our work in preparation \cite{GS2}.
\end{remark}

%\begin{remark}\label{r.Lab} %OLD version of this remark
%	Proposition \ref{P.Information} is one of the reasons that this model is a laboratory for the IV-GFF. In this case, it allows us to conjecture that the best $\epsilon(T)$ that one can obtain in the result of Fröhlich and Spencer at inverse temperature $\beta\ll1$ is $\Omega(\beta e^{-1/\beta})$.
%\end{remark}

\section{Conjectures on $T_{rec}$  and the interfaces of the models.}\label{s.conjectures}
%\section{Conjectures on $T_{rec}$  and the interfaces of the models.}\label{s.conjectures}
%In this section, we conjecture certain results that should hold for our model and other models close to it. Yet, when possible, we will also prove some intermediates results  that will give support to these conjectures.
The main focus of this section is to state several conjectures. However, we also prove some intermediate results  which are interesting on their own and which will give support to each of these predictions. As such, this section has more mathematical content than a list of open questions.

\subsection{Lower bound on the value of $T_{\rec}^-$.}\label{ss.Trec-}
The objective of this part is to justify the following conjecture:
\begin{conjecture}\label{c.Trec}
	We have that $T_{\rec}^-\geq 2\sqrt{\pi}$.
\end{conjecture}

We have two reasons to believe this conjecture, both of them being related to the continuum Gaussian free field. The first reason concerns the so-called imaginary chaos and the second one is related to the flow lines of the continuum GFF.

\subsubsection{Reason 1: Imaginary chaos.} 
We will not introduce all the definitions here. We refer to \cite{LacoinRhodesVargas2014Hyperb,JSW} for context and the definition. Take $\Phi$ a $0$-boundary continuum Gaussian free field in a domain $D\subseteq \C$ and let $\nu_{\epsilon}^x$ be the uniform measure on $\partial B(x,\epsilon)$. We normalise $\Phi$ so that if $d(x,y)\geq \epsilon$
\[\E\left[(\Phi,\nu_{\epsilon}^x ) (\Phi,\nu_{\epsilon}^y) \right] =G_D(x,y). \]
Note that in our normalisation $G_D(x,y)\sim \frac{1}{2\pi}|\log(\|x-y\|)|$.

Let us take $\alpha$ and define $\V^\alpha$ as the imaginary chaos associated with $\alpha$.
\[\V^\alpha=\V^{\alpha}(\Phi):= \lim_{\epsilon \to 0} \exp\left (i\alpha \phi_\epsilon(\cdot)+\frac{\alpha^2}{2}\E\left[\phi^2_\epsilon(\cdot) \right] \right ). \]
Here the limit is taken in the space of distribution, and it is only non-trivial in the case $\alpha<2\sqrt{\pi}$. Note that our normalisation is different from the one in these references, in which our $\alpha$ correspond to $\tilde \alpha=\sqrt{2}$.

We can now prove the following result.
\begin{proposition}\label{P.C Mes implies D Mes}
	\hypertarget{Assumption sun}{Assume }
	\begin{enumerate}
		\item[(H$_1$)]There exists $\hat \alpha$ such that for all $\alpha< \hat \alpha$ the GFF $\Phi$ can be measurably recovered from $\V^\alpha(\Phi)$, i.e., that there exists a deterministic measurable function $F$ such that a.s $F(\V^\alpha)=\Phi$.
	\end{enumerate}
	Then, we have that $T_{\rec}^{-}\geq \hat \alpha$.
\end{proposition} 
%In fact, we expect that \hyperlink{Assumption sun}{(H$_1$)} holds for  $\hat \alpha=2\sqrt{\pi}$. Indeed, it is conceivable that a technique similar to the one developed by \cite{berestycki2014equivalence} may  apply in this case. However, let us emphasize that this case is more subtle as one needs to connect the local fluctuations all the way to the values of the boundary. 

\begin{remark}\label{}
After the first version of this work, hypothesis \hyperlink{Assumption sun}{(H$_1$)} was proved up to $\hat \alpha=2\sqrt{\pi}$ in the recent work \cite{Juhan}. Let us emphasize that the imaginary case is more subtle than the same question for the real chaos  analyzed in \cite{berestycki2014equivalence} as one needs to control the local fluctuations all the way to the values of the boundary. 
\end{remark}

To prove Proposition \ref{P.C Mes implies D Mes}, we need to show that the discrete imaginary chaos is converging to the continuous one. 
\begin{proposition}\label{P.Convergence of discrete chaos}
	Let $\phi^{(n)}$ be a discrete $0$-boundary GFF in $\Lambda_n$ and let
	\[\V_n^\alpha(\cdot):= \exp\left (i\alpha \phi^{(n)}(\cdot)+\frac{\alpha^2}{2}\E\left[\phi^{(n)}(\cdot) \right]  \right ), 
	\]
	then for all $\alpha < 2\sqrt{\pi}$, as $n\to \infty$
	\[(\phi^{(n)}, \V_n^\alpha)\to (\Phi,\V^\alpha(\Phi)), \ \ \text{ in law}, \]
	for the topology of generalised functions. Here $\Phi$ is a $0$-boundary GFF in $[-1,1]²$.
\end{proposition}
As this section is concerned mostly with conjectures we will only do a sketch of the proof of this result. The main input is the fact that Theorem 1.3 of \cite{JSW} which states that $(\Phi,\V^\alpha(\Phi))$ is characterised by its moments. 
\begin{proof}
	We start by recalling that, thanks to  Theorem 1.3 of \cite{JSW}, the field $(\Phi,\V^\alpha(\Phi))$ is characterised by its moments. By this, we mean that it is characterised by
	\begin{align}
	\label{E.correlation function}&\E\left[\left( \prod_{i} (\Phi,f^1_i)\right)\left (\prod_{j} (\V^\alpha,f^2_j)\right )\left (\prod_{k} \overline{(\V^\alpha,f^3_k)}\right ) \right]\\
	&\hspace{0.05\textwidth}=\int \left( \prod_{i} f^1_i(x_i)dx_i\right)\left (\prod_{j}f^2_j(y_j)dy_j \right )\left (\prod_{k} \overline{f^3_k(z_k)}dz_k\right )C((x_i)_{i},(y_j)_j,(z_k)_k),\nonumber
	\end{align}
	where all $f^{\ell}_\cdot$ are smooth functions in $[-1,1]^2$ (with 0-mean if $\Phi$ is a free-boundary GFF). The function $C(\cdot, \cdot, \cdot)$ is called the correlation function of this model. By a simple (but lenghty) computation one can see that \eqref{E.correlation function} also appears from the discrete setting
	\begin{align}
	\label{E.correlation function2}&\E\left[\left( \prod_{i} n^{-2}\langle\phi_n,f^1_i\rangle\right)\left (\prod_{j} n^{-2}\langle\V_n^\alpha,f^2_j\rangle \right )\left (\prod_{k} \overline{n^{-2}\langle\V_n^\alpha,f^3_k\rangle}\right ) \right]\to\\
	&\hspace{0.05\textwidth}\int \left( \prod_{i} f^1_i(x_i)dx_i\right)\left (\prod_{j}f^2_j(y_j)dy_j \right )\left (\prod_{k} \overline{f^3_k(z_k)}dz_k\right )C((x_i)_{i},(y_j)_j,(z_k)_k),
	\end{align}
	at least when all function $f$s have different support. This can be proven by noting  that $C$ is obtained only from the Green's function and that the discrete Green's function is converging to the continuum one (Corollary 3.11 of \cite{ChSm}). To finish, one needs to show that \eqref{E.correlation function2} is true for all possible $f$s. This can be done using the dominated convergence theorem. To see that the sum coming from the LHS of \eqref{E.correlation function2} is uniformly dominated one uses Theorem 2.5 of \cite{ChSm}, i.e., that \[G(x,y)= -(2\pi)^{-1}\log\left (\frac{\|x-y\|}{n}\right )+O(1),\]
	and uses the same techniques as Section 3.2 of \cite{JSW}.
\end{proof}

We can now prove Proposition \ref{P.C Mes implies D Mes}.
\begin{proof}[Proof of Proposition \ref{P.C Mes implies D Mes}]
	Take $\phi_1^{(n)}$ and $\phi_2^{(n)}$ two $0$-boundary GFF coupled as in Definition \ref{D.phis}. Thanks to Proposition \ref{P.Convergence of discrete chaos}, we have that the $4$-tuple
	\[(\phi_1^{(n)},\V_{1,n}^\alpha,\phi_2^{(n)},\V_{2,n}^{\alpha}) \]
	is tight. Take $(\Phi_1,\V^{\alpha}_1,\Phi_2,\Phi_2,\V^{\alpha}_2)$, any accumulation point of the sequence and note that because for all $n\in \N$, a.s. $\V_{1,n}^{\alpha}=\V_{2,n}^{\alpha}$ we have that $\V^{\alpha}_1=\V^{\alpha}_2$. This equality implies, thanks to Assumption \hyperlink{Assumption sun}{(H$_1$)} 
	that a.s. $\Phi_1=\Phi_2$. Then, as all accumulation points are the same we have that, in fact, as $n\to \infty$
	\[(\phi_1^{(n)},\V_{1,n}^\alpha,\phi_2^{(n)},\V_{2,n}^{\alpha})\to (\Phi_1,\V_1^\alpha,\Phi_1,\V_1^\alpha)\ \  \ \text{in distribution}. \]
	Let us, now, take any smooth function $f$, we have that for all $j\in \{1,2\}$
	\[\sup_{n}\E\left[\left (\frac{1}{n^2}\langle\Phi_j,f\rangle\right )^4\right] < K, \]
	which implies that
	\[\E\left[\left (\frac{1}{n^2}\langle \Phi_1-\Phi_2,f\rangle\right )^2 \right]\to \E\left[(\Phi_1-\Phi_1,f)^2 \right]=0. \]
	As this implies that $\E\left[\var{\langle \phi_1^{(n)},f\rangle}\mid \exp(i\alpha \phi_1^{(n)}) \right]=o(n^4)$, we conclude as in the beginning of Section \ref{S.Localisation}. 
\end{proof}

\begin{remark}
	Let us note that even if we Assumption \hyperlink{Assumption sun}{(H$_1$)} is proven, this only shows that $\E\left[\var{\langle \phi_1^{(n)},f\rangle} \right]=o(n^4)$, which is a weaker result than the one in Proposition \ref{P.Var O(n^2)} in which we showed that $\E\left[\var{\langle \phi_1^{(n)},f\rangle} \right]=O(n^2)$.
\end{remark}

\subsubsection{Reason 2: flow lines.}
Flow lines of the Gaussian free field were introduced in \cite{She05,Dub} and were studied in depth in \cite{MS1,MS2,MS3,MS4}. Informally, they can be described as the curve which is the solution of 
\[\eta'(t)= e^{i(\sqrt{2\pi}\frac{\Phi}{\chi}+u)}, \ \ \ \eta(0)=z\in \partial D, \] 
where $\Phi$ is a GFF in a simply connected domain $D$ and $u$ is a harmonic function. For us it is important to note that the curve $\eta$ should only be determined by $e^{i(\sqrt{2\pi}\frac{\Phi}{\chi})}$. This will motivate Assumption \hyperlink{Assumption AC}{(H$_2$)}.

Flow lines can be defined using the concept of local sets \cite{SchSh2,WWln2}. In other words, $\eta$ is a flow line of a GFF $\Phi$ if for any stopping time $\tau$ of the natural filtration of $\eta$ we have
\[\Phi= \Phi^{\eta_\tau}+ h_{\eta_\tau}, \]
where $\eta_\tau=\eta([0,\tau])$, $\Phi^{\eta_\tau}$ has the law of a GFF of $D\backslash \eta_\tau$ and $h_{\eta_\tau}$ is a harmonic function in $D\backslash \eta_\tau$. Let us remark that in this case the function $h_{\eta_\tau}$ is, in fact, a measurable function of $\eta_{\tau}$. In fact, it can be found in Theorem 1.1 of \cite{MS1}.

A generalisation of flow-lines is given by the angle-varying flow lines defined in Section 5.2 of \cite{MS1}, which can be roughly described as running a flow line with initial angle $\theta_1$ until a stopping time\footnote{w.r.t. the natural filtration of $\eta$}) $\tau_1$, and then continue with an angle $\theta_2$ until a stopping time $\tau_w$, and continue until finitely many iterations. This lines are called $\eta_{\theta_1...\theta_\ell}^{\tau_1...\tau_\ell}$ and they are a measurable function of $\Phi$, the GFF they are coupled with (Lemma 5.6 of \cite{MS1}).

In fact, Proposition 5.9 of \cite{MS1}, shows that if $\chi\geq 1/\sqrt{2}$, there exists a countable set of angle-varying flow lines $(\eta_{\theta^{k}_1...\theta^k_{\ell_k}}^{\tau^{k}_1...\tau^k_{\ell_k}})_{k\in \N}$ such that a.s.
\[\bigcup_n \eta_{\theta^k_1...\theta^k_{\ell_k}}^{\tau^k_1...\tau^k_{\ell_k}} \]
is dense (because SLE$_8$ is a space-filling curve).  Now, define $\F_n$ as the $\sigma$-algebra generated by $\eta_{\theta^{\ell_n}_1...\theta^n_{\ell_n}}^{\tau^{\ell_n}_1...\tau^n_{\ell_n}}$. The discussion in the paragraph before and the fact that $h_{\eta_{\theta^{\ell_n}_1...\theta^n_{\ell_n}}^{\tau^{\ell_n}_1...\tau^n_{\ell_n}}}$ is a measurable function of the set $\eta_{\theta^{\ell_n}_1...\theta^n_{\ell_n}}^{\tau^{\ell_n}_1...\tau^n_{\ell_n}}$ implies that the $\F=\bigvee_{n}\F_n$ is equal to the sigma algebra generated by $\Phi$ (see for example Lemma 2.3 of \cite{ALS1}). In other words, $\Phi$ is a deterministic function of $(\eta_{\theta^{\ell_n}_1...\theta^n_{\ell_n}}^{\tau^{\ell_n}_1...\tau^n_{\ell_n}})_{n\in \N}$.

This allows us to show the following proposition.
\begin{proposition}\label{P.Con flow lines}
	Take $\phi_n$ a $0$-boundary GFF in $\Lambda_n$,\hypertarget{Assumption AC}{assume }
	\begin{enumerate}
		\item[(H$_2$)]There exists $\hat \chi\geq 1/\sqrt 2$ such that for all $\chi>\hat \chi$ and for any $\eta_{\theta_1...\theta_\ell}^{\tau_1...\tau_\ell}$ an angle-varying flow line, there exists an approximated angle-varying flow line $\eta^{(n)}$ depending on $\exp(i\sqrt{2\pi}\phi_n/\chi)$ such that $(\phi_n,\eta^{(n)})$ converges in law to $(\Phi,\eta_{\theta_1...\theta_\ell}^{\tau_1...\tau_\ell}(\Phi))$.
	\end{enumerate}
	Then $T_{\rec}^{-}\geq \sqrt{2\pi}/\chi$.
\end{proposition}
Before proving the proposition, let us recall that it is expected that the flow lines related to the discrete GFF are converging to the flow lines of the continuum GFF, as this is already the case for $\chi=\infty$, the SLE$_4$ case \cite{SchSh}. If this were the case, Proposition \ref{P.Con flow lines} implies that $T_{\rec}^-\geq 2\sqrt{\pi}$.
\begin{proof}
	Let us take $\Phi$ a continuous GFF with $0$-boundary condition. Thanks to Assumption \hyperlink{Assumption AC}{(H$_2$)}, we can define $\eta^{(n)}_k$ such that as $n\to \infty$
	\begin{equation}\label{E.Convergence k}(\phi_n,\eta^{(n)}_k)\to \left (\Phi,\eta_{\theta^k_1...\theta^k_{\ell_k}}^{\tau^k_1...\tau^k_{\ell_k}}(\Phi) \right ),
	\end{equation}
	in law. We then, have that
	\begin{equation}\label{E.convergence of infinite flow lines}
	(\phi_n,(\eta_k^{(n)})_{k\in \N})\to \left (\Phi,(\eta_{\theta^k_1...\theta^k_{\ell_k}}^{\tau^k_1...\tau^k_{\ell_k}}(\Phi))_{k\in \N}\right )
	\end{equation}
	in law for the product topology. This follows because \eqref{E.Convergence k} implies that $(\phi_n,(\eta_k^{(n)})_{k\in \N})$ is tight for the product topology. We can then check, again thanks to \eqref{E.Convergence k}, that any accumulation point $(\Phi,(\eta^\infty_k)_{k\in \N})$ has to be such that
	\[(\Phi, \eta^\infty_k)= \left (\Phi,\eta_{\theta^k_1...\theta^k_{\ell_k}}^{\tau^k_1...\tau^k_{\ell_k}}(\Phi)\right ).  \]
	As a consequence, we have that
	\[(\Phi,(\eta^\infty_k)_{k\in \N})= \left (\Phi,(\eta_{\theta^k_1...\theta^k_{\ell_k}}^{\tau^k_1...\tau^k_{\ell_k}}(\Phi))_{k\in \N}\right ),\]
	which implies \eqref{E.convergence of infinite flow lines}.
	
	We can now conclude in a similar way as in Proposition \ref{P.C Mes implies D Mes}. We take $(\phi_1,\phi_2)$ coupled as in Definition \ref{D.phis} and we study the $4$-tuple 
	\begin{equation*}\label{E.4-tuple discrete}
	\left (\phi^{(n)}_1,(\eta^n_k(\phi^{(n)}_1))_{k\in \N},\phi^{(n)}_2, (\eta^n_k(\phi^{(n)}_2)_{k\in \N} \right ).
	\end{equation*}
	Again, we have that this $4$-tuple is tight and that any accumulation point is such that
	\[\left (\Phi_1,(\eta_{\theta^k_1...\theta^k_{\ell_k}}^{\tau^k_1...\tau^k_{\ell_k}}(\Phi_1))_{k\in \N},\Phi_2, (\eta_{\theta^k_1...\theta^k_{\ell_k}}^{\tau^k_1...\tau^k_{\ell_k}}(\Phi_2))_{k\in \N} \right ), \]
	as for all $n,k\in \N$, we have that $\eta^n_k(\phi^{(n)}_1)=\eta^n_k(\phi^{(n)}_2)$, we have that in this accumulation point a.s.
	\[ \eta_{\theta^k_1...\theta^k_{\ell_k}}^{\tau^k_1...\tau^k_{\ell_k}}(\Phi_1)=\eta_{\theta^k_1...\theta^k_{\ell_k}}^{\tau^k_1...\tau^k_{\ell_k}}(\Phi_2). \]
	As $\Phi_i$ is a function of this $(\eta_{\theta^k_1...\theta^k_{\ell_k}}^{\tau^k_1...\tau^k_{\ell_k}}(\Phi_i))_{k\in \N}$, we see that $\Phi_1=\Phi_2$, which implies that  $(\phi_1^{(n)},\phi_2^{(n)})$ converges in law to $(\Phi_1,\Phi_1)$. By the same reasoning as the end of \ref{P.C Mes implies D Mes} we have that for any continuous function $f$, $\E\left[\var\left [\langle\phi^{(n)},f\rangle\mid e^{\sqrt{2\pi} \phi^{(n)}/\chi}\right ]\right]=o(n^4)$.

\end{proof}

\subsection{Interfaces of $\exp(iT\phi)$.}\label{ss.LL}
In this section, we discuss the possible scaling limit of certain interfaces naturally appearing in $\exp(iT\phi)$ and how they may relate with the interfaces of the GFF $\phi$.
\subsubsection{Level lines of $\exp(iT\phi)$}
In \cite{SchSh}, the authors showed that the level line of a zero boundary GFF with a special boundary condition converges in law to an SLE$_4$. We believe a similar story holds for both $\exp(iT\phi)$, and more importantly for the Villain model. Let us be more explicit.

We define $u_n$ as the bounded harmonic function in $\Lambda_n\backslash \partial \Lambda_n$ with boundary condition $\lambda=\sqrt{\pi/8}$ in $\partial \Lambda_n \cap \{x:\Re(x)\geq 0\}$ and $-\lambda=-\sqrt{\pi/8}$ in $\partial \Lambda_n\cap \{x:\Re(x)<0\}$. It is shown in \cite{SchSh}, that if $\phi_n$ is a GFF in $\Lambda_n$ with $0$-boundary condition and $\eta$ the level line of $\phi+ u_n$. That is to say $\eta^{(n)}(\cdot)$ is a path in the dual of $\Lambda_n$ that has the following properties (see Figure \ref{F.dGFF}):
\begin{itemize}
	\item It goes from the dual of the edge $(-i+1/n,-i)$ to the dual of the edge $(i-1/n,i)$.
	\item The primal edge associated to a dual edge in the path is such that $\phi_n$ is negative to its left and positive to its right.
\end{itemize}
\begin{figure}[h!]
	\includegraphics[width=0.44\textwidth]{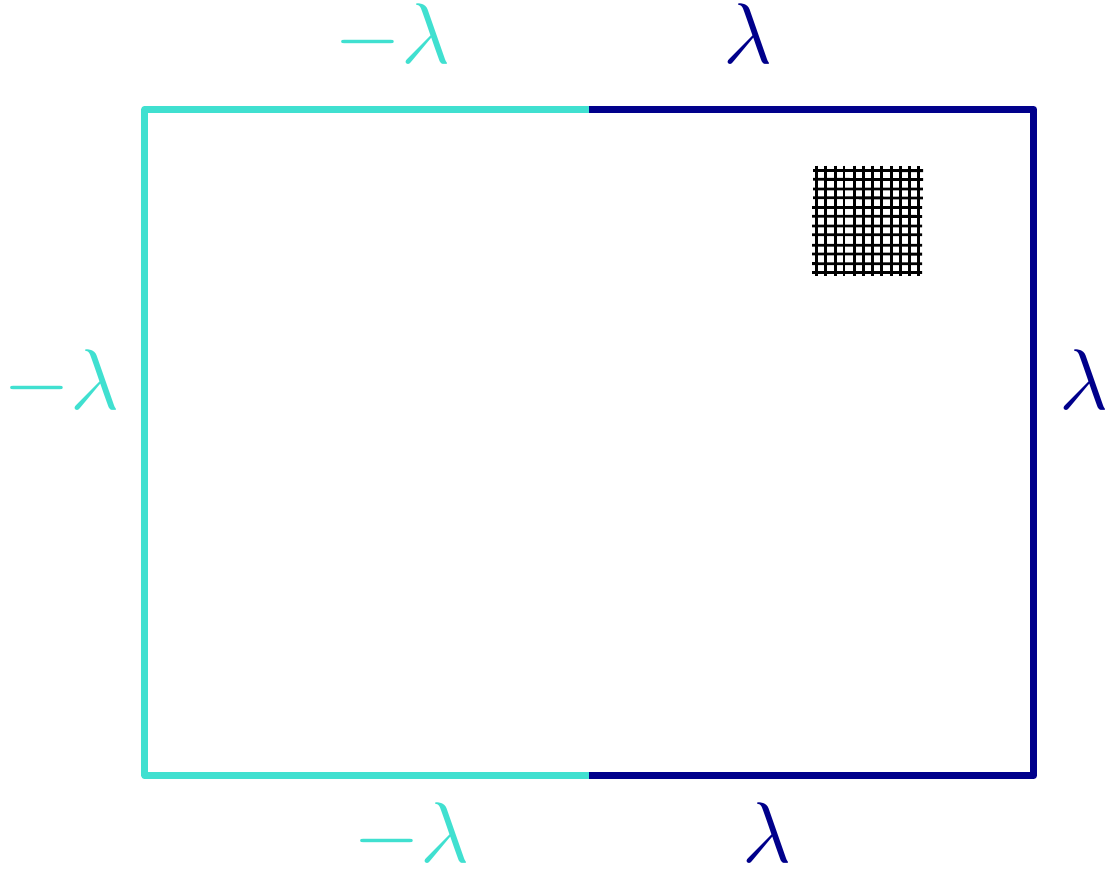}
	\includegraphics[width=0.44\textwidth]{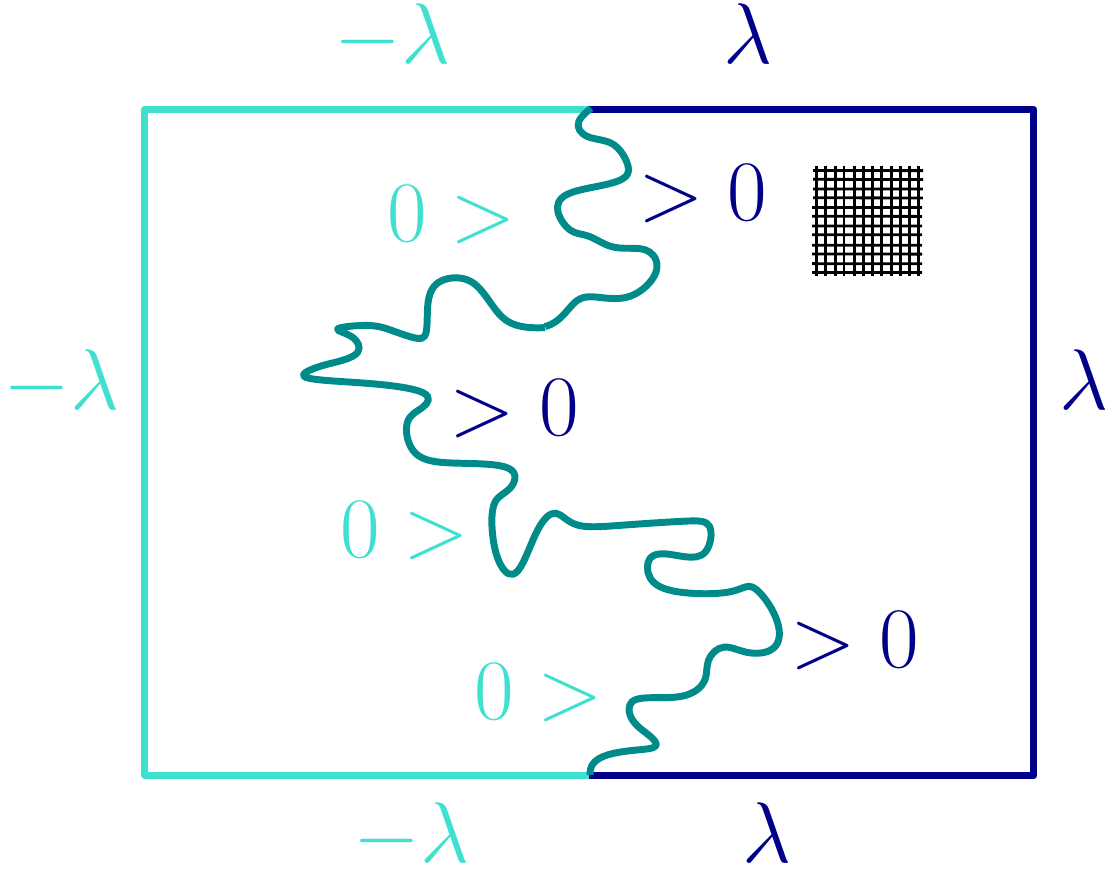}
	\caption{The image of the left depicts the boundary values of the harmonic function $u_n$. The image to the right represents the level line of $\phi+u_n$, note that $\phi+u_n$ takes positive values to the left and negative to the right.}
	\label{F.dGFF}
\end{figure}

Theorem 1.4 of \cite{SchSh} is that $\eta^{(n)}(\cdot)$ parametrised by capacity converges in the uniform topology to an SLE$_4$. This result is improved in \cite{SchSh2} by showing that as $n\to\infty$
\[(\phi_n,\eta^{(n)})\to (\Phi,\eta) \ \ \ \text{in law.} \]
Here $(\Phi,\eta)$ are such that $\Phi$ is a GFF in $[-1,1]$ and $\eta$ is the so-called level line of the continuous GFF. More precisely, $\eta$ is a measurable function of $\Phi$ and the law of $\Phi$ conditioned on $\eta$ is such that
\[\Phi+u_\infty= \Phi^L + \Phi^R, \]
where $\Phi^L$, resp. $\Phi^R$, is a GFF in the domain to the left, resp. right, of $\eta$ with $-\lambda$, resp. $\lambda$, boundary condition (see Figure \ref{F.cGFF}).

\begin{figure}[h!]
	\includegraphics[width=0.44\textwidth]{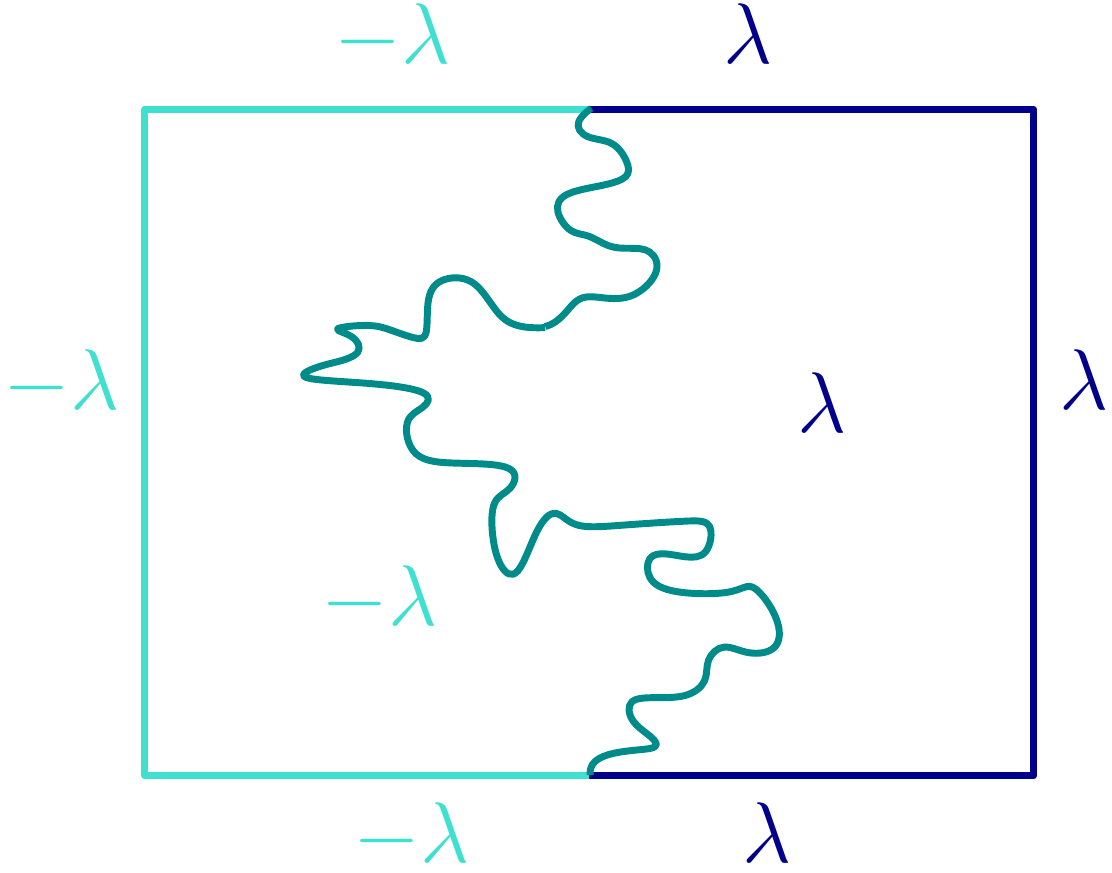}
	\caption{The image shows how the limiting curve $\eta$ separates the domain in to two different domains, the left where the GFF has $-\lambda$ boundary condition and the right where its boundary condition is $\lambda$.}
	\label{F.cGFF}
\end{figure}

We now have the tools to prove Corollary \ref{c.Levelines}.
	\begin{proof}[Proof of Corollary \ref{c.Levelines}]
		We assume that $\phi_n\to \Phi$ a continuum GFF and define $L^{(n)}=L^{(n)}_T(\exp(iT\phi_n))$ as a set parametrised by $q$, where
		\[L^{(n)}(q)=\E\left[\eta^{(n)}(q)\mid \exp(iT\phi) \right]. \]
		Let us now prove that the set $L^{(n)}$ converges in probability to $\eta$ the level line of $\phi$. To do this, it is enough to show that for all $q$, $L^{(n)}(q)$ converges in probability to $\eta(q)$. Thanks to Theorem 1.4 of \cite{SchSh} we have that $\eta^{(n)}(q)$ converges in law to $\eta(q)$, now it suffices to show that as $n\to \infty$
		\begin{equation}\label{e.variance_eta}
		\var[\eta^{(n)}(q)\mid \exp(iT\phi) ]\to 0, \ \ \ \text{ in probability.}
		\end{equation}
		To do this, we use the same trick as always. Let $(\phi_1^{(n)},\phi_2^{(n)})$ be two GFFs coupled as in Definition \ref{D.phis}, we know that thanks to (1) of Theorem \ref{T.0-boundary theorem} $(\phi_1^{(n)},\eta_1^{(n)},\phi_2^{(n)},\eta_2^{(n)})$ converges in law to $(\Phi, \eta, \Phi,\eta)$. Here the topology on the curves is that of the uniform distance for continuous curves. As a consequence of the convergence we have that for any $q\in Q$
		\begin{equation*}
		\P(\|\eta_1^{(n)}(q)-\eta_2^{(n)}(q)\| \geq \delta ) \to 0 \ \ \ \text{ as } n\to \infty.
		\end{equation*} 
		Due to the fact that the set $\Lambda_n$ is bounded, we conclude that
		\begin{equation*}
		\E\left[\|\eta_1^{(n)}(q)-\eta_2^{(n)}(q)\|^2 \right] \to 0, \ \ \ \text{ as } n\to \infty.
		\end{equation*}
		This concludes the proof, as it proves \eqref{e.variance_eta}.
	\end{proof}
	
	Corollary \ref{c.Levelines} gives us a explicit way to recover the level line of the GFF given its $e^{iT\phi_n}$. However, this recovery process does not locally depend on the field. We also believe that it is possible to recover the level line via an explicit local function of the $e^{iT(\phi_n+u_n)}$: its own level line.

Now, we let $T$ be small enough such that $T\lambda<\pi$, in this way the imaginary part of $\exp(i T u_n(x))$ has the same as sign as the real part of $x$. We also define $\eta^{(n),T}(\cdot)$, the level line of the imaginary part $\exp(iT(\phi_n+u_n))$. We conjecture the following.
\begin{conjecture}\label{C.level lines of the dGFF}
	There exists a small enough $T_c$ such that for all $T<T_c$, $\eta^{(n),T}$ converges in law to a SLE$_4$. Furthermore, $\eta^{(n)}$ and $\eta^{(n),T}$ converge to the same limit.
\end{conjecture}
A part from Corollary \ref{c.Levelines}, we have two other reasons to believe in this conjecture. The first one is the fact that the gradient of $\phi_n$ in its level line $\eta^{(n)}$ is, in mean, upper and lower bounded (see Lemma 3.1 of \cite{SchSh}). Thus, one could expect that most edges in $\eta$ have corresponding primal edges for which $Im(\exp(iT\phi_n))$ is negative on its left vertex and positive on its right one.

The second reason is that level lines do not get close to each other, neither to itself. This can be seen in Section 3.4 and 3.5 of \cite{SchSh}, or by understanding their scaling limit as in Remark 1.5 of \cite{WaWu}.

As we said before, we conjecture that we have a similar result for the Villain model.
In fact, Fröhlich and Spencer  conjectured that the Villain model at low temperature $T$ is close to the imaginary exponential of a GFF with a slightly different  temperature $T_{\text{Vil}}:=T_{\text{Vil}}(T)>T$ (see Section 8.1 of \cite{FS1983}). This allows us to interpret Conjecture \ref{C.level lines of the dGFF} as follows.
\begin{conjecture}\label{c.SLE4}
	Take $T$ small enough and let $\psi_n$ be a Villain model in $\Lambda_n$ with temperature $T$ and boundary values given by $\exp(-i\lambda\sqrt{T_{\text{Vil}}'})$ in the left side of the boundary, i.e. $\partial \Lambda_n\cap \{x:\Re(x) < 0\}$, and $\exp(i\lambda\sqrt{T_{\text{Vil}}'})$ in $\partial \Lambda_n\cap \{x:\Re(x)\geq0\}$. If we take $\eta^{(n)}$ to be the level line of the imaginary part of $\psi$, then $\eta^{(n)}$ converges in law to an SLE$_4$ (see Figure \ref{F.Villain}). 
\end{conjecture}

\begin{figure}[h!]
	\includegraphics[width=0.44\textwidth]{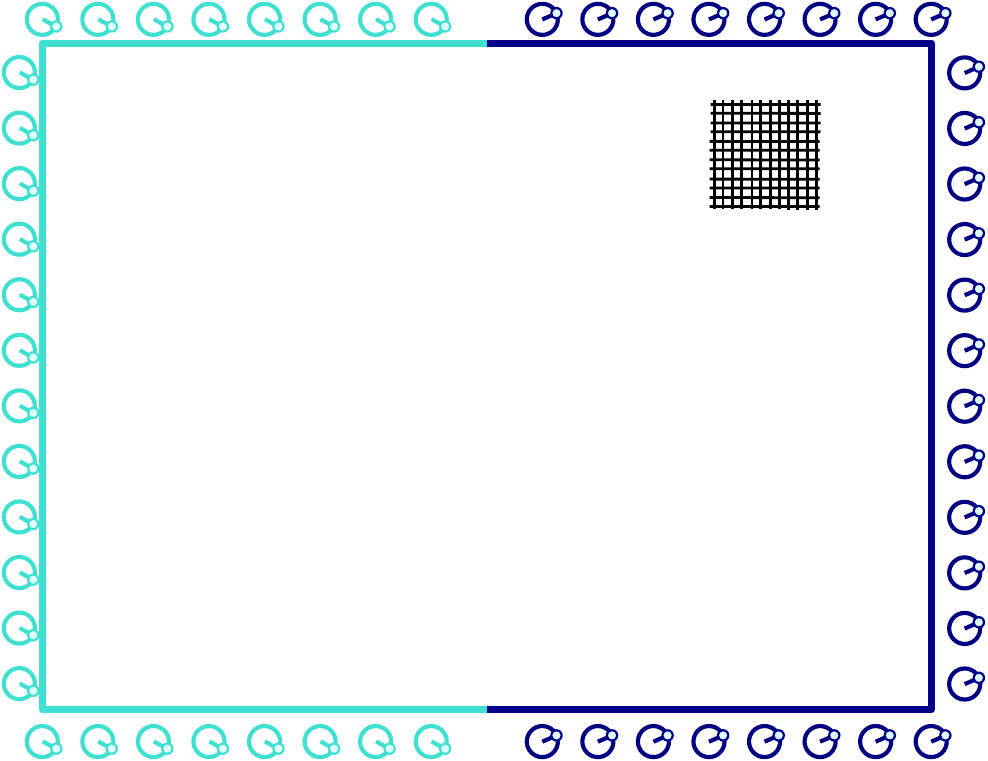}
	\includegraphics[width=0.44\textwidth]{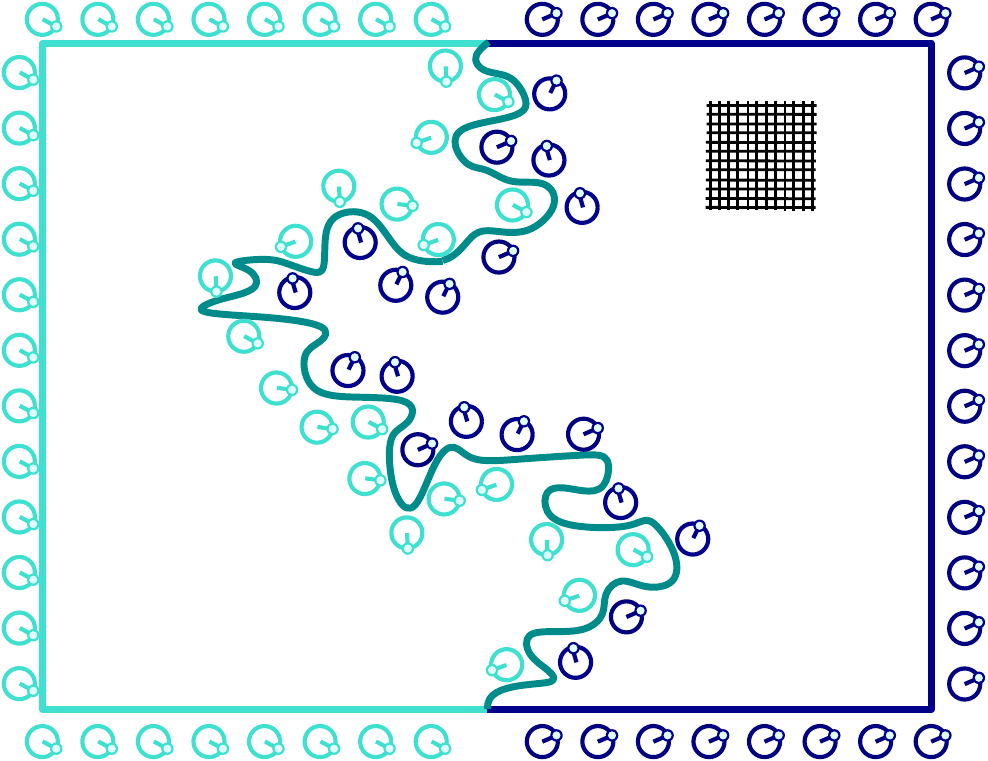}
	\caption{The left image depicts the boundary values of the Villain model. The right image represents the level line of the imaginary part of this Villain model. We believe that this lines converges in law to an SLE$_4$ when the temperature of the system is low enough.}
	\label{F.Villain}
\end{figure}

In fact, the result should hold for a more generally boundary values.
\begin{conjecture}\label{c.SLErho}
	Take $T$ small enough and let $\psi_n$ be a Villain model in $\Lambda_n$ with temperature $T$ and boundary values given by $\exp(-ia)$ in the left side of the boundary, i.e. $\partial \Lambda_n\cap \{x:\Re(x)< 0\}$, and $\exp(i a)$ in $\partial \Lambda_n\cap \{x:\Re(x)\geq0\}$. Then for $a$ small enough, we take $\eta^{(n)}$ to be the level line of the imaginary part of $\psi$, then $\eta^{(n)}$ converges in law to an SLE$_4(\rho)$, with $\rho=a/(\lambda\sqrt t)-1$. 
\end{conjecture}
\subsubsection{Full set of interfaces} Instead of studying a single interface of the GFF, one could also study the whole set of interfaces arising from a $0$ boundary condition GFF. This sets are called ALE, and were introduced in \cite{ASW} and further studied in \cite{AS2,QWCoupling,ALS1}. 

ALE's are characterised as the only random set $\A_{-\lambda,\lambda}$ such that a continuum GFF $\Phi$ can be written as
\begin{equation}
\Phi:=\sum_{O} \Phi^O+\sigma^O\lambda\,,
\end{equation}
where the sum is over
connected components $O$ of the complement of $\A_{-\lambda,\lambda}$, i.e., $[-1,1]^2\backslash \A_{-2\lambda,2\lambda}$. Furthermore, $\sigma^O\in \{-1,1\}$ and where conditionally on $\A_{-\lambda,\lambda}$, $\Phi^O$ is a $0$-boundary GFF in $O$ (conditionally )independent of $(\Phi_{O'})_{O'\neq O}$. The existence and uniqueness of such a set was proven in \cite{ASW}. Furthermore, as it was shown in Lemma 3.6 of \cite{AS2}, this set can be thought as the union of the $0$-level lines of the continuum GFF $\Phi$.

In fact, for this discussion it is useful to define $\eta$ the $0$-level line of a discrete GFF going between $x\in \partial \Lambda_n$ and $y\in \partial \Lambda_n$. $\eta$ is then a dual path connecting an edge containing $x$ to an edge containing $y$ such that for all vertices in $\Lambda_n\backslash \partial \Lambda_n$ to the left of $\eta$, one has that $\phi_n(x)<0$ and that for all vertices to the right of $\eta$, $\phi_n(x)>0$. Furthermore, let us define the discrete ALE, $\A_{-\lambda,\lambda}^n$ as the union over all starting points and end points of its associated $0$-level line.

The $0$-level line is known to converge for the Hausdorff topology by Theorem 1.3 of \cite{SchSh}, and furthermore the techniques of \cite{SchSh2} allow us to see that it converges to the $0$-level line of a continuum GFF. These techniques, together with the above-mentioned Lemma 3.6 of \cite{AS2}, allow to show that $\A_{-\lambda,\lambda}^n$ converges for the Hausdorff topology to the ALE\footnote{The exact argument is not written anywhere, even though this proof has been known to a small community. As the main focus of this section is not this result, but rather to shade light on this interesting direction we will not formalise this result further here.}.

We can now discuss similar results as the ones we did for the level lines. In particular, as before we have that
\begin{proposition}
	We have that for $T<T_{rec}^{-}$, there exists a deterministic function $A_T(\cdot)$ such that when $\phi_n\to \Phi$, we have that $A_T(e^{iT\phi_n})\to \A_{-\lambda,\lambda}$.
\end{proposition}

The problem, as before, is that we do not know whether this function $A_T$ can be taken to be the discrete ALE associated to the imaginary part of $e^{iT\phi_n}$. This is the content of the next conjecture.
\begin{conjecture}
	Take $T$ small enough and $\phi_n$ a $0$-boundary GFF converging to $\Phi$. One has that the ALE associated to the imaginary part of $e^{iT\phi_n}$ converges to $\A_{-\lambda,\lambda}$.
\end{conjecture}

\begin{figure}[h]
	\includegraphics[width=0.44\textwidth]{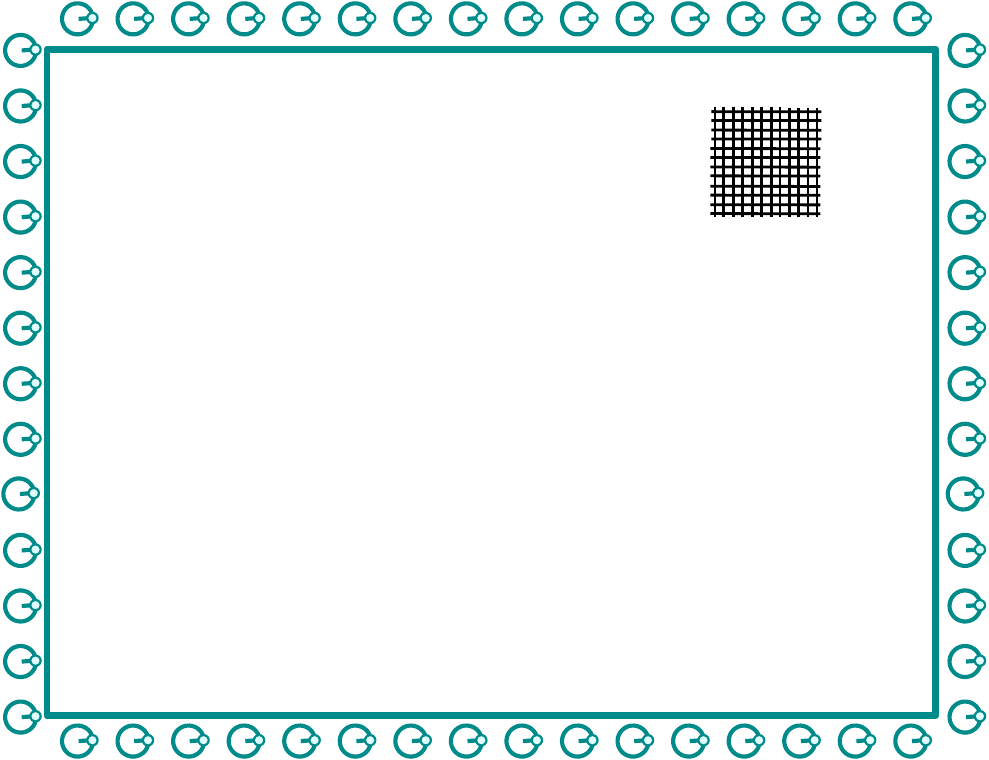}
	\includegraphics[width=0.44\textwidth]{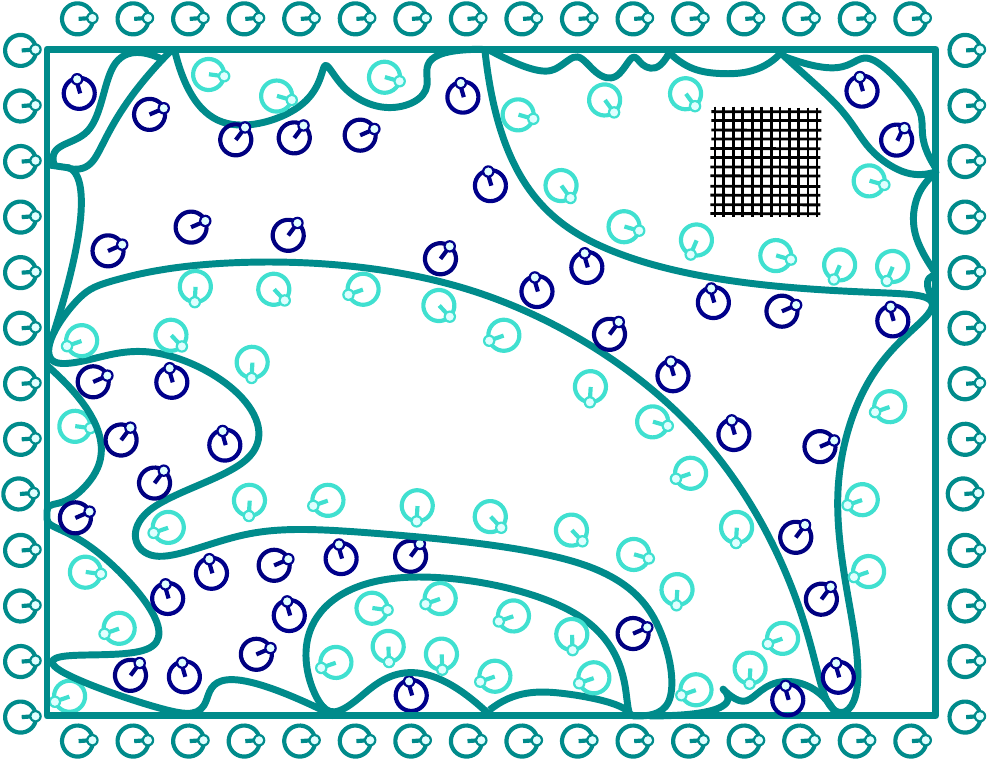}
	\caption{The left image depicts the boundary values of the Villain model for this case. The right image represents the ALE associated to the imaginary part of this Villain model. We believe that this set converges in law to the ALE $\A_{-\lambda,\lambda}$ when the temperature of the system is low enough. A striking consequence of this conjecture is that the law of the limiting set does not depend on the temperature, as long as the system is cold enough.}
	\label{F.Villain_ALE}
\end{figure}

An even more daring conjecture, proposes that the same is true for a Villain model at small enough temperature.
\begin{conjecture}\label{c.ALE}
	Take $T$ small enough and let $\psi_n$ be a Villain model in $\Lambda_n$. One has that as $n\to \infty$ the discrete ALE associated to the imaginary part of $\psi_n$ converges to $\A_{-\lambda,\lambda}$.
\end{conjecture}
It is interesting to note that we expect the set of interfaces of the Villain model at low temperature resemble a lot to each other at various $\beta$. That is to say that this geometry will not distinguish the temperature from which the ALE arises. However, we expect that the law inside each connected component of the complement of this ALE's will  look pretty different. To be more precise, we expect that the boundary conditions generated by this ALE get closer and closer to $1=e^{i0}$ as the temperatures goes to $0$.

\subsection{Upper bound on the value of $T_{\rec}^+$.} 
In fact, the analysis of level lines of the GFF, makes us believe the following conjecture.
\begin{conjecture}
	We have that $T_{\rec}^+\leq 2\sqrt{2\pi}$.
\end{conjecture}
Let us note that the value $2\sqrt{2\pi}$, it is the smallest value of $T$ so that $\exp(i T \lambda)=\exp(-iT\lambda)$. This is to say that this is the value for which we could not expect to recognize the macroscopic difference between the left and the right side of the level line $\eta$ introduced in Section \ref{ss.LL}. 

The level line $\eta$ is fundamental to be able to recover the GFF. This is shown, for example, in the construction of the free-boundary GFF given in \cite{QWCoupling}.

There is another reason why we believe that one cannot recover $\phi$ when $T=2\sqrt{2\pi}$. It has to do with the level set of the GFF.

Although the GFF is not a function, one can still define $\mathbb A_{-a,b}$. This is informally, the (connected component connected to the boundary of the)  preimage of $[-a,b]$. These sets were introduced\footnote{See \cite{SSV} to better understand the relationship between $\mathbb A_{-a,b}$ and the imaginary chaos.} in \cite{ASW,ALS1} and their existence is conditional on the size of the interval $[-a,b]$
\begin{itemize}
	\item[] \textit{The set $\mathbb A_{-a,b}$ if and only if $a,b>0$ and $a+b\geq 2\lambda$.}
\end{itemize}
The case $a+b=2\lambda$ is special. These are the values such that $\exp(-i Ta)=\exp(iTb)$. Furthermore, in \cite{AS2}, it is shown that these are the only values of $a$ and $b$ such that the following happens
\begin{itemize}
	\item[] \textit{Fix two-points $x,y\in [-1,1]$ and let $O(x)$ and $O(y)$ be the connected component of $[-1,1]^2\backslash \mathbb A_{-a,b}$ containing $x$ and $y$ respectively. Then, there is a positive probability that $O(x)\neq O(y)$ and $\partial O(x)\cap \partial O(y)$ is a continuous curve.}
\end{itemize} 
This property implies that the places where the GFF takes values $-a$ and the ones where it takes values $b$ are mesoscopically separated, i.e. they are not macroscopically far apart. As the function $x\mapsto \exp(i2\sqrt{ 2\pi} x)$ cannot distinguish between $-a$ and $b$, we believe it is not possible to recover $\mathbb A_{-a,b}$ just by knowing $\exp(i2\sqrt 2 \pi \phi)$. This would make impossible to recover all the macroscopic information of the GFF.

\appendix
\section{Viewing the shift $\ba=\{a_i\}_{i\in \Lambda}$ as an exterior magnetic field}\label{a}

The goal of this appendix is to provide a different proof of Proposition \ref{c.magical}. 
The idea of this proof was inspired to us by an inspection of this exact identity in the simplest possible case of a Gaussian free field on a single point $\{x\}$ with Dirichlet boundary condition, namely  a Gaussian $\calN(0,\frac 1 \beta)$. The appendix is organized as follows, first we investigate the case of one point, then we make a link with Riemann-theta functions (thus explaining the name {\em modular invariance}) and finally we give a second proof of Proposition \ref{c.magical}.

\subsection{Warm up: GFF with one point and Jacobi-theta function.}

%\begin{figure}[!htp]
%\begin{center}
%\includegraphics[width=0.8\textwidth]{simus3}
%\end{center}
%\caption{}\label{}
%\end{figure}

Let us consider the GFF on a graph with two points $\{x,y\}$ with 0-boundary conditions in $y$. The partition function of the $a$-shifted integer valued field (here the vector $\ba$ is just one parameter which we call $a$) reads as follows:
\begin{align*}\label{}
Z(\beta,a)& = \int \big(\sum_{n\in \Z} \delta_{2\pi n +a}(\phi)\big)  \frac 1 {\sqrt{2\pi /\beta}}  e^{-\frac {\beta} {2}  \phi^2}   d \phi \\
& = \frac 1 {\sqrt{2\pi /\beta}}  \sum_{n\in \Z} \exp(-\frac \beta 2 (2\pi n + a)^2)\,.
\end{align*}

In the limiting case where we plug the following infinite Fourier series 
\[
1+2\sum_{q=1}^\infty \cos(q(\phi-a)) \equiv  2\pi \sum_{n\in \Z} \delta_{2\pi n + a}(\phi)
\] into the Fröhlich-Spencer expansion on one point, it can be checked that the identity~\eqref{e.magical} reads
as follows
%\begin{align*}\label{}
%T_N(\beta,a)&:=
%\frac
%{
%\sum_{n=-N}^N \exp(-\frac \beta 2 (2\pi n + a)^2)\cdot(2\pi n +a)
%}
%{
%\sum_{n\in \Z} \exp(-\frac \beta 2 (2\pi n + a)^2)
%} \\
%\tilde T_M(\beta,a) &:=
%\frac
%{
%\frac 1 \beta 
%\sum_{q=1}^{M}
%2(1-\frac q {1000})
%e^{-\frac {q^2} {2 \beta}}
%\sin(q \cdot a) \cdot q
%}
%{
%1+ 
%\sum_{q=1}^{M}
%2(1-\frac q {1000})
%e^{-\frac {q^2} {2\beta}}
%\cos(q\cdot a)
%}\,.
%\end{align*}
\begin{align}\label{e.theta}
\frac
{
\sum_{n\in \Z} \exp(-\frac \beta 2 (2\pi n + a)^2)\cdot(2\pi n +a)
}
{
\sum_{n\in \Z} \exp(-\frac \beta 2 (2\pi n + a)^2)
}
&=
\frac
{
\frac 1 \beta 
\sum_{q\in \Z}
e^{-\frac {q^2} {2 \beta}}
\sin(q \cdot a) \cdot q
}
{
\sum_{q\in \Z}
e^{-\frac {q^2} {2\beta}}
\cos(q\cdot a)
}\,,
\end{align}
which is correct for any $\beta>0$ and any real $a\in(-\pi,\pi)$.  (Note interestingly that it is degenerate for the L.H.S as $\beta \to 0$ but not for the R.H.S!) 

One way to prove this identity is to notice its link with Jacobi's theta function. Indeed the later is classically defined as follows (see for example \cite{Mumford}).
\begin{align*}\label{}
\theta(z \md \tau):= \sum_{n\in \Z} \exp(i \pi n^2 \tau + 2 i \pi n z)\,, %\,\,\, z\in \C, \tau\in \H
\end{align*}
defined for all $z\in\C, \tau\in \H$. %\margin{D. Mumford says p28, the behaviour in $\tau$ is also extremely beautiful, but rather deeper and more subtle. In our setting $\tau$=inverse temperature}
Now if one plugs 
\[
z:=i \beta a\,, \,\, \tau:= 2 i \pi \beta 
\]
into $\theta$, we find 
\begin{align*}\label{}
\theta(z\md \tau)= e^{\frac \beta 2 a^2} \sum_{n\in \Z} \exp(-\frac \beta 2 (2\pi n + a)^2)\,.
\end{align*}
Jacobi's first modular identity states that
\begin{align}\label{e.J1}
\theta(\frac z \tau \md \frac {-1} \tau) = \alpha \, \theta(z \md \tau),
\end{align}
where $\alpha=(-i\tau)^{1/2} \exp(\frac \pi \tau i z^2) = \sqrt{2\pi \beta} \exp(-\frac \beta 2 a^2)$. 
This identity gives us: 
\begin{align*}\label{}
\sum_{q\in \Z} e^{-\frac {q^2} {2\beta}}\cos(q a) = \sqrt{2\pi \beta} \sum_{n\in \Z} \exp(-\frac \beta 2 (2\pi n + a)^2)
= \sqrt{2\pi \beta} Z(\beta,a)\,,
\end{align*}
from which one can prove the identity~\eqref{e.theta} by taking a  log-derivative in $a$. Note that one may also avoid using Jacobi's identity and reprove things using a Poisson summation formula. We indicate the link here as our shift-parameter $\ba$ which is central to our work is naturally associated to the first argument $z$ of the theta function $\theta$ (while the second argument $\tau$ is related to the inverse temperature).

%
%The identity we numerically checked is whether for $\beta>0$ small enough, 
%\begin{align*}\label{}
%\lim _{N\to \infty} T_N(\beta,a)
%&= 
%\lim_{M\to \infty} \tilde T_M(\beta,a)\,.
%\end{align*}
%See the plot we obtained in Figure \ref{f.??} for the values $N=100$ and $M\in \{10,20,100\}$. 

%\begin{align*}\label{}
%\frac
%{
%\sum_{n=-200}^{200} \exp(-\frac \beta 2 (2\pi n + a)^2)\cdot(2\pi n +a)
%}
%{
%\sum_{n=-200}^{200} \exp(-\frac \beta 2 (2\pi n + a)^2)
%}
%\approx
%\frac
%{
%\frac 1 \beta 
%\sum_{q=1}^{1000}
%2(1-\frac q {1000})
%e^{-\frac {q^2} {2 \beta}}
%\sin(q \cdot a) \cdot q
%}
%{
%1+ 
%\sum_{q=1}^{1000}
%2(1-\frac q {1000})
%e^{-\frac {q^2} {2\beta}}
%\cos(q\cdot a)
%}\,,
%\end{align*}
%where we noticed, somewhat surprisingly, that the convergence of the series on the R.H.S seems to be very slow. 

\medskip
The argument we just outlined bares some resemblance with the fact that, in an Ising model with an exterior magnetic field $h$, one can compute the average magnetization as a derivative w.r.t $h$ of the free energy $\log Z$. In our context, we used that 
\begin{align*}\label{}
\frac
{
\sum_{n\in \Z} \exp(-\frac \beta 2 (2\pi n + a)^2)\cdot(2\pi n +a)
}
{
\sum_{n\in \Z} \exp(-\frac \beta 2 (2\pi n + a)^2)
}
&
= -\frac 1 {\beta} \p_a \log(Z(\beta,a))\,.
\end{align*}

This suggests that our key identity~\eqref{e.magical} for a general domain $\Lambda\subset \Z^2$ should be reminiscent of the way to recover the average magnetic field of the Ising model from the derivative in $h$ of its Free energy $\log Z$. We implement this idea in the rest of the appendix by viewing the vector shift $\ba=\{a_i\}_{i\in \Lambda}$ acting as an external magnetic field. 
We prove Proposition \ref{c.magical} along these lines in two steps: %\margin{I don't understand what we are doing here. Is it a justification for the next subsection? In A.3, we work also with the non-limiting case so 1) does not feel really true?}:
\bi
\item[A.2)] First, as in the case of one-point, we work in the limiting case of infinite Fourier series at each vertex $x\in \Lambda$. This makes the analogy with an Ising-model clearer and makes a connection with the modular invariance of certain Riemann-theta functions. From the intuition gathered here, we notice that the key identity \eqref{e.magical} is an appropriate log-derivative w.r.t $\ba$, namely $-\<{\sigma, \nabla_\ba \log Z} = -\<{\frac 1 \beta (-\Delta)^{-1} f, \nabla_\ba \log Z}$. 
\item[A.3)] In the second part, we work in the finite cut-off case. Here it is not so clear how to recognise the integral against $\<{f,\phi}$ on the R.H.S of the identity~\eqref{e.magical}. The reason comes from the fact that expansion into charges from \cite{FS} (and particularly the effect of the complex translation under spin-waves) somehow obfuscates the readability of $\EFK{\beta,\Lambda,\lambda_\Lambda,v}{\ba}{\<{\phi,f}} $. To end the proof, we first get around the blurring effect caused by the expansion into charges from \cite{FS} (using the matching of partition functions before and after expansions into charges) and then connect to an actual average of $\<{\phi, f}$ by running Gaussian integration by parts. 
\ei

%\subsection{Second proof of the magical identity~\eqref{e.magical} through gradients in $\ba$ of suitable partition functions.}
\subsection{Riemann-theta function and $\ba$-shifted integer valued GFF.}
$ $

\smallskip
In this section, we implicitly rely on expansions into infinitly many charge configurations in \cite{FS,RonFS} by attaching to each vertex $i\in \Lambda$ the following infinite trigonometric series 
\[
\lambda_i(\phi_i) = 1 + 2 \sum_{k=1}^\infty \cos(k(\phi_i -a_i))\,.
\]
We will not properly justify here %\margin{You told me FS properly justify it?}
 that the series are well defined as our goal is to justify properly in the next section the key identity~\eqref{e.magical} which holds in the finite cut-off case 
\[
\lambda_i(\phi_i) = 1 + 2 \sum_{k=1}^N \cos(k(\phi_i -a_i))\,,
\]
with $N$ large. 
\smallskip

Let us introduce the following two partition functions in the general case of $\Lambda\subset \Z^2$ with, say, Dirichlet boundary conditions. 
\begin{align*}\label{}
%Z(\beta, \ba) & := \sum_{m\in \Z^\Lambda} \exp(-\frac \beta 2 \sum_{i\sim j} (2\pi (m_i -m_j) +a_i-a_j)^2 )\\
%Z(\beta, \ba) & := \frac{1}{\sqrt{2\pi \beta^{-1}\det \Delta^{-1} }}\sum_{m\in \Z^\Lambda} \exp(-\frac \beta 2 \sum_{i\sim j} (2\pi (m_i -m_j) +a_i-a_j)^2 )\\
Z(\beta, \ba) & := \frac{1}{ \sqrt{(2\pi \beta^{-1})^{|\Lambda|}\det (-\Delta)^{-1} }}\sum_{m\in \Z^\Lambda} \exp(-\frac \beta 2 \sum_{i\sim j} (2\pi (m_i -m_j) +a_i-a_j)^2 )\\
\tilde Z(\beta,\ba) & = \sum_{\calN \in \calF} c_\calN Z_\calN^\ba(0) \\
& = \sum_{\calN\in \calF} c_\calN \, \int \prod_{\rho\in\calN}[1+z(\beta,\rho,\calN) \cos(\<{\phi,\bar \rho}-\<{\ba,\rho})]
d\mu_{\beta,\Lambda,v}^\GFF(\phi).
\end{align*}
(N.B. As hinted above $\calF$ must be an infinite set of charge configurations here). 

The expansion from Fröhlich-Spencer (in this limiting case) reads as follows: for all $\beta<\beta_0$ and $\ba \in \R^\Lambda$
\begin{align}\label{e.dual}
Z(\beta, \ba) = \tilde Z(\beta,\ba).
\end{align}

Inspired by the analogy with Ising, we now compute for any $g : \Lambda \to \R$, %\margin{Sorry, I change $\exp$ by $e^x$ because that way it fits in the line.}
\begin{align*}\label{}
& \<{g, \nabla_\ba  \log Z(\beta, \ba)} \\
& = \sum_{i\in \Lambda} g_i \p_{a_i} \log Z(\beta, \ba) \\
& =  -\beta\frac{\sum_{i\in \Lambda} g_i  \sum_{m\in \Z^\Lambda} e^{-\frac \beta 2 \sum_{i\sim j} (2\pi (m_i -m_j) +a_i-a_j)^2}  (\sum_{i\sim j}  (2\pi (m_i -m_j) +a_i-a_j) ) }   {\sum_{m\in \Z^\Lambda} e^{-\frac \beta 2 \sum_{i\sim j} (2\pi (m_i -m_j) +a_i-a_j)^2} } \\
& = \beta
\frac
{
\sum_{i\in \Lambda} \sum_{m\in \Z^\Lambda} e^{-\frac \beta 2 \sum_{i\sim j} (2\pi (m_i -m_j) +a_i-a_j)^2}    [\Delta(2\pi m + \ba)]_i g_i } 
 {\sum_{m\in \Z^\Lambda} e^{-\frac \beta 2 \sum_{i\sim j} (2\pi (m_i -m_j) +a_i-a_j)^2}  )}  \\
& =\beta
\frac
{
\sum_{m\in \Z^\Lambda} e^{-\frac \beta 2 \sum_{i\sim j} (2\pi (m_i -m_j) +a_i-a_j)^2} \<{g,\Delta(2\pi m + \ba)} } 
 {\sum_{m\in \Z^\Lambda} e^{-\frac \beta 2 \sum_{i\sim j} (2\pi (m_i -m_j) +a_i-a_j)^2} }.
\end{align*}
Choose (as in \cite{FS,RonFS}), 
\[
g=\sigma := \frac 1 \beta \red{-}\Delta^{-1} f .
\]
This gives us
\begin{align*}\label{}
 \<{\sigma, \nabla_\ba \log Z(\beta,\ba)} &= \red{-} 
\frac
{
\sum_{m\in \Z^\Lambda} e^{-\frac \beta 2 \sum_{i\sim j} (2\pi (m_i -m_j) +a_i-a_j)^2} \<{f ,2\pi m + \ba} } 
 {\sum_{m\in \Z^\Lambda}e^{-\frac \beta 2 \sum_{i\sim j} (2\pi (m_i -m_j) +a_i-a_j)^2}}  \\
& = \red{-} \EFK{\beta,\Lambda,\ba}{\mathrm{IV}}{\<{\phi,f}}.
\end{align*}

Now, from ~\eqref{e.dual}, we know that for any function $g :\Lambda \to \R$,  we have 
\begin{align*}\label{}
\<{g, \nabla_\ba  \log Z(\beta, \ba)} = \<{g, \nabla_\ba \log \tilde Z(\beta, \ba)}.
\end{align*}

As such, this implies with $g=\sigma$ the following formula for $\EFK{\beta,\Lambda,\ba}{\mathrm{IV}}{\<{\phi,f}}$:
\begin{align*}\label{}
\EFK{\beta,\Lambda,\ba}{\mathrm{IV}}{\<{\phi,f}} = \red{-} \<{\sigma,\nabla_\ba  \log \tilde Z(\beta,\ba)}.
\end{align*}

%where 
%\begin{align*}\label{}
%\tilde Z(\beta,\ba) & = \sum_{\calN \in \calF} c_\calN Z_\calN^\ba(0) \\
%& = \sum_{\calN\in \calF} c_\calN \, \int \prod_{\rho\in\calN}[1+z(\beta,\rho,\calN) \cos(\<{\phi,\bar \rho}-\<{\ba,\rho})]
%d\mu_{\beta,\Lambda,v}^\GFF(\phi)
%\end{align*}
%
%
%(CHECK signs, but it seems fine). \purple{\bf I think both partition functions are the same up to a constant (which does not matter when differentiating w.r.t $\ba$) BUT the only difference is that there is intrinsically  $-\ba$ in Frolich-Spencer way of rewriting things while this is naturally $+\ba$ on the side of the Integer-valued partition function.}

Let us then compute this gradient $\nabla_\ba$ and check that it gives us the desired identity: 
\begin{align*}\label{}
& \<{\sigma,\nabla_\ba  \log \tilde Z(\beta,\ba)} \\
& = \frac 1 {\tilde Z(\beta,\ba)}\sum_i \sigma_i  \sum_{\calN\in\calF} c_\calN 1_{i\in \rho=\rho_i \in \calN} \int z_i(\beta,\rho,\calN)
[\red{-}\sin(\<{\phi,\bar \rho}-\<{\ba,\rho})\red{-}\red{\rho_i}] \\
& \;\;\;\;\times
\prod_{\rho\in\calN\setminus \{\rho_i\}}
[1+z(\beta,\rho,\calN) \cos(\<{\phi,\bar \rho}-\<{\ba,\rho})]
d\mu_{\beta,\Lambda,v}^\GFF(\phi) \\
& = 
 \frac 1 {\tilde Z(\beta,\ba)} \sum_{\calN\in\calF} c_\calN  \int  
 \left( 
 \sum_{\rho\in\calN} \frac
 { z(\beta,\rho,\calN) \sin(\<{\phi,\bar \rho}-\<{\ba,\rho})\<{\sigma,\rho}
 }
 {
 1+z(\beta,\rho,\calN) \cos(\<{\phi,\bar \rho}-\<{\ba,\rho})
 }
 \right) \\
& \;\;\;\;\times
\prod_{\rho\in\calN}
[1+z(\beta,\rho,\calN) \cos(\<{\phi,\bar \rho}-\<{\ba,\rho})]
d\mu_{\beta,\Lambda,v}^\GFF(\phi)
\end{align*} 
and we thus recover the R.H.S of~\eqref{e.magical} in the limiting case of infinite trigonometric polynomials at each site. 

\smallskip

Let us briefly highlight now the link with Riemann-theta functions which we believe illustrates what is beneath the identity~\eqref{e.magical}. It is not hard to rewrite the partition function 
\[
Z(\beta, \ba) = \frac{1}{\sqrt{(2\pi \beta^{-1})^{|\Lambda|}\det (-\Delta)^{-1} }}\sum_{m\in \Z^\Lambda} \exp\left(-\frac \beta 2 \sum_{i\sim j} (2\pi (m_i -m_j) +a_i-a_j)^2 \right)
\]
as a theta function in several variables (i.e, the Riemann-theta function). The later generalized theta functions may be defined as follows (see for example \cite{Mumford}): for any $g\geq1$, $\bz = (z_1,\ldots,z_g) \in \C^g$ and $\Omega$ a symmetric $g\times g$ complex matrix whose imaginary part is positive definite, set
\begin{align}\label{e.RT}
\theta(\bz \md \Omega):= \sum_{\bm \in \Z^g}  \exp(\pi i \, \bm^T \Omega \bm + 2 i \pi \,  \bm \cdot \bz).
\end{align}
The Riemann-theta functions therefore match exactly with our model when 
\[
\begin{cases}
\bz &:= i \beta (-\Delta) \ba  \\ 
\Omega& := 2 i \, \pi \beta (-\Delta) .
\end{cases}
\]
We claim that the identity~\eqref{e.magical} is reminiscent of the suitable  $\log$-derivative (i.e. taking $F(\ba) \mapsto - \<{\sigma, \nabla_\ba F(\ba)}$) of the modular invariance identity for Riemann-theta function (see for example 5.1 in \cite{Mumford}) which states that 
%The extension of Jacobi's modular identity~\eqref{e.J1} to Riemann-theta function is the following modular identity:
\begin{align}\label{e.J2}
\theta(\Omega^{-1} \bz, - \Omega^{-1}) = \sqrt{\det(-i \Omega)} \exp(i \bz^T \Omega^{-1} \bz) \theta(\bz,\Omega)\,.
\end{align}

\subsection{Blurring effect of the decomposition into charges.}
In this subsection, we work with finite cut-off Fourier series (and therefore do not need to worry with convergences of series) and we end our alternative proof of Proposition \ref{c.magical}. 

By running the same computation as the one outlined above for the infinite trigonometric series, we have that the R.H.S in the identity~\eqref{e.magical} is given by 
\[
-\<{\sigma, \nabla_\ba  \log \tilde Z_N(\beta, \ba)} 
\]
where  $\tilde Z_N(\beta, a)$ denotes here the rewriting by Fröhlich-Spencer of the finite-cutoff partition function (i.e. $\tilde Z_N(\beta,\ba) =  \sum_{\calN \in \calF} c_\calN Z_\calN^\ba(0)$). 
The key point here is that it is not clear how to recognise the expectation $\EFK{\beta,\Lambda,\lambda_\Lambda,v}{\ba}{\<{\phi,f}}$ from the above $\log$-derivative. To make that identification easier, one should rely instead on an easier expression of the partition function $\tilde Z_N(\beta, \ba)$. Indeed we will instead work with the initial expression of the partition function before subtle expansions into charges are made. Namely, we consider
\begin{align*}\label{}
\hat Z_N(\beta,\ba) 
& := \int \prod_{x\in \Lambda} \left(1 + 2 \sum_{k=1}^N \cos(k (\phi_x-a_x)) \right) d \P_{\beta,\Lambda}^\GFF
\end{align*}
and we then compute
\begin{align*}\label{}
 -&\<{\sigma, \nabla_\ba  \log \tilde Z(\beta, \ba)}  \\
 & = 
-\<{\sigma, \nabla_\ba  \log \hat Z_N(\beta, \ba)}  \\
& = -
\sum_{i\in \Lambda} \sigma_i  
\frac
{\EFK{\beta}{\GFF}{
2 \left(\sum_{k=1}^N  k \sin (k (\phi_i - a_i)) \right) \prod_{x\in \Lambda\setminus i} \left(1 + 2 \sum_{k=1}^N \cos(k(\phi_x-a_x)) \right) 
}}
{\hat Z_N(\beta, \ba)}\,.
\end{align*}
We now wish to compare this with an expression for
$\EFK{\beta,\Lambda,\lambda_\Lambda,v}{\ba}{\<{\phi,f}}$:
\begin{align*}\label{}
& \EFK{\beta,\Lambda,\lambda_\Lambda,v}{\ba}{\<{\phi,f}} \\
& = \EFK{\beta,\Lambda}{\GFF}{\prod_{x\in \Lambda} \left( 1 +  2 \sum_1^N \cos(k(\phi_x-a_x))\right) \purple{\<{\phi,f}} } \\
& = \sum_{i\in \Lambda} f_i \EFK{\beta,\Lambda}{\GFF}{
\prod_{x\in \Lambda} \left( 1 +  2 \sum_1^N \cos(k(\phi_x-a_x))\purple{\phi_i}\right)} \\
& =   \sum_{i\in \Lambda} f_i  \sum_{j\in \Lambda} \<{\purple{\phi_i} \phi_j}_\beta^\GFF 
\EFK{\beta,\Lambda}{\GFF}{ \p_j
\prod_{x\in \Lambda} \left( 1 + 2 \sum_1^N \cos(k(\phi_x-a_x)) \right)}\,,
\end{align*}
by Gaussian integration by parts. Continuing, this gives us
\begin{align*}\label{}
& \EFK{\beta,\Lambda,\lambda_\Lambda,v}{\ba}{\<{\phi,f}} \\
&= \sum_{i\in \Lambda} f_i  \sum_{j\in \Lambda} \frac 1 \beta (-\Delta)^{-1}(i,j)
\EFK{\beta,\Lambda}{\GFF}{ - 2 \sum_{k=1}^N \sin(k(\phi_j -a_j)) k 
\prod_{x\in \Lambda \setminus j} \left( 1 +  2 \sum_1^N \cos(k(\phi_x-a_x)) \right)} \\
& = \purple{-} 
\sum_{i\in \Lambda} f_i   \frac 1 \beta (-\Delta)^{-1}(i,j) \Psi(j)\,,
\end{align*}
where 
\[
\Psi(j):= 2 \EFK{\beta,\Lambda}{\GFF}{ \sum_{k=1}^N \sin(k(\phi_j -a_j)) k 
\prod_{x\in \Lambda \setminus j} \left( 1 +  2 \sum_1^N \cos(k(\phi_x-a_x)) \right)}\,.
\]
\begin{align*}\label{}
 \EFK{\beta,\Lambda,\lambda_\Lambda,v}{\ba}{\<{\phi,f}} & = \purple{-} 
\sum_{i\in \Lambda} f_i   \frac 1 \beta [-\Delta^{-1}]\Psi(i) \\
& = \purple{-} \frac 1 \beta \<{f, (-\Delta^{-1}) \Psi}  = \purple{-}  \<{\frac 1 \beta (-\Delta^{-1})f,  \Psi} \\
& = \purple{-} \sum_{i\in \Lambda} \sigma_i 
 \EFK{\beta,\Lambda}{\GFF}{ 2 \sum_{k=1}^N \sin(k(\phi_i -a_i)) k 
\prod_{x\in \Lambda \setminus i} \left( 1 +  2 \sum_1^N \cos(k(\phi_x-a_x)) \right)}\,,
\end{align*}
which ends our proof as we obtained, as desired, the same expression as for $- \<{\sigma, \nabla_\ba \log \hat Z_N(\beta,\ba)}$. \qed

\section{Link with the random-phase Sine-Gordon model}\label{a.SG}

As it was pointed out to us by Tom Spencer, our work turns out to be closely related to the  \textbf{random-phase Sine-Gordon} model which has been studied extensively in the physics literature (\cite{cardy1982,doussal2007}) and which we now introduce.

\begin{definition}\label{}
Let $\Lambda\subset \Z^2$ be a finite domain. Let $z\in[0,\infty]$ (this is called the \textbf{activity}) and $\ba=\{a_i\}_{i\in \Lambda}$ be a \textbf{quenched disorder} on the vertices given by i.i.d random variables $a_i$ uniform in $[0,2\pi)$. 

We equip the domain $\Lambda$ with either Dirichlet or free boundary conditions. The \textbf{random-phase Sine-Gordon model} is the following quenched disorder probability measure on fields $\{\phi_i\}_{i\in \Lambda}$
\begin{align*}
\P^{\ba-\SG}_{\beta,z,\Lambda}[d\phi] := \frac 1 {Z^{\ba-\SG}_{\beta,z,\Lambda}}  \; 
\exp(z \sum_{i\in \Lambda}  \cos(\phi_i - a_i))
\P_{\beta,\Lambda}^{\GFF}[d\phi]\,.
\end{align*}
\end{definition}

\begin{remark} $ $
\bnum
\item 
Note that if we let the {\em activity} $z\to \infty$, the measure $\P^{\ba-\SG}_{\beta,\Lambda}$ converges to the $\ba$-shifted IV-GFF on $\Lambda$.  (With $\ba\in[0, 2\pi)^\Lambda)$). 
\item If $z\to \infty$ and if the disorder $\ba$, instead of being uniform in $[0,2\pi)^\Lambda$, is sampled as follows
\[
\ba:= \varphi \Mod{2\pi}, \text{ with }\varphi\sim \P_{\beta}^\GFF\,,
\] 
then,  the annealed law $\int \P(d\ba) \P^{\ba-\SG}_{\beta,\infty,\Lambda}[d\phi]$ is very simple and is given by $\P_{\beta}^\GFF[d\phi]$. 
\enum
\end{remark}

When the \textbf{disorder} $\ba$ is uniform, it turns out that the annealed law is very different from the law of a GFF. Indeed in a series a works including \cite{cardy1982,doussal2007}, the following {\em roughening/super-roughening} phase transition has been predicted: 
\bi
\item If the temperature is high enough, it is predicted that on large domains $\Lambda_n:=\{-n,\ldots,n\}^2$, the random phase Sine-Gordon model will fluctuate as the GFF, namely for any fixed $z\in[0,\infty]$ and $\beta$ small enough, 
\begin{align*}\label{}
\EFK{\ba}{}{\EFK{\beta,z,\Lambda_n}{\ba-\SG}{\phi^2(0)} } \asymp_{n\to \infty} \log n\,. 
\end{align*}
\item On the other hand, if the temperature is low enough, fluctuations are predicted to be larger! The following super-roughening behaviour is predicted (see \cite{cardy1982,doussal2007}): for any fixed positive activity $z\in(0,\infty]$ (note that here $z>0$ is required) and $\beta$ high enough, 
\begin{align*}\label{}
\EFK{\ba}{}{\EFK{\beta,z,\Lambda_n}{\ba-\SG}{\phi^2(0)} } \asymp_{n\to \infty} (\log n)^2\,.
\end{align*}
\ei

Our present work does not allow us to investigate the more surprising low temperature phase with expected $(\log n)^2$ variances (see for example Remark \ref{r.LowTemp}). Yet, it enables us to prove rigorously that the fluctuations for the {\em random-phase Sine-Gordon model} in the the high temperature regime are at least as large as for the GFF. (N.B. note that with the quenched disorder $\ba$, one cannot rely on classical correlation inequalities such as Ginibre). Namely a very mild generalization of the proof of Theorem \ref{th.IVgff} implies the following result (which also clarifies the link between high enough temperature and the choice of activity $z$).

\begin{theorem}\label{th.SG}
For simplicity, we state our result for 0-boundary conditions around $\Lambda_n$ (but the analogous statement also holds for free boundary conditions by considering $\phi(a)-\phi(b)$ for two distant points $a,b$ in the bulk). 

There exists $\beta_0$ s.t. for all $\beta<\beta_0$ and all activity $z\in[0,\infty]$, then uniformly in the disorder $\ba \in [0,2\pi)^{\Lambda_n}$, we have 
\begin{align*}\label{}
\mathrm{Var}_{\beta,z,\Lambda_n}^{\ba-\SG}\left[\phi(0)\right] \geq \Omega(1) \log n\,.
\end{align*}
This implies in particular the following lower-bound for the fluctuations of random phase Sine-Gordon (i.e. with a quenched disorder $\ba\sim$ i.i.d) when $\beta<\beta_0$:
\begin{align*}\label{}
\EFK{\ba}{}{\EFK{\beta,z,\Lambda_n}{\ba-\SG}{\phi^2(0)} } \geq \Omega(1) \log n\,. 
\end{align*}
\end{theorem}

\ni
{\em Sketch of proof.}
Since Theorem \ref{th.IVgff} is stated uniformly in the disorder $\ba$, the latter theorem implies the limiting case $z=\infty$. It remains to notice that the proof also handles the case of finite activities $z\in[0,\infty)$ using the following minor modifications.  Indeed, recall that we wrote the proof of Theorem \ref{th.IVgff} for general trigonometric polynomials 
\[
\lambda_i(\phi) = 1 + 2 \sum_{q=1}^N \hat \lambda_i(q) \cos(q \phi(i))\,.
\]
where the set of weights $\lambda_\Lambda=(\lambda_i)_{i\in \Lambda}$ is assumed to satisfy the same hypothesis as in (5.35) in \cite{FS} (or equivalently (1.9) in \cite{RonFS}). 

In our present setting, at any site $i\in \Lambda$, we need to work with the following periodic function:
\begin{align*}\label{}
\phi_i \mapsto e^{z \cos(\phi_i - a_i)} = \sum_{q\in \Z} \alpha(q) \cos(q(\phi_i-a_i))\,,
\end{align*}
with
\[
\alpha(q)= \frac 1 {2\pi} \int_0^{2\pi} e^{- i q \theta }e^{z \cos \theta} d\theta\,.
\]
It is sufficient to notice that $0< \alpha(0)\leq e^z$ and that for each $q\in \Z$, $|\alpha(q)| \leq \alpha(0)$. Indeed, we may rewrite our periodic function as follows:
\[
e^{z \cos(\phi_i - a_i)} = \alpha(0) \left( 1 + 2\sum_{q\geq 1} \hat \lambda(q) \cos(q(\phi_i-a_i)) \right)\,,
\]
where $|\hat \lambda(q)| \leq 1$. Since the conditions on $\hat \lambda(q)$ in \cite{FS,RonFS} are conditions on the growth of these coefficients, it is immediate to see that this ``Sine-Gordon'' trigonometric polynomial satisfies the conditions required to run the same proof as in Theorem \ref{th.IVgff}. To fully match with the setup in that proof, note that one can absorb the multiplicative constant $\alpha(q)$ in the trigonometric polynomial into the partition function without any impact on the fluctuations. Also, we wrote the proof as in \cite{RonFS}, with a cut-off $N$ on large frequencies (i.e. looking at $1+2\sum_{q=1}^N \hat \lambda(q)$). The same limiting argument $N\to \infty$) as in \cite{RonFS} applies here.  The rest of the proof (in particular the analysis of the first and second moments in Subsection \ref{ss.proof12}) is identical in this setting. This ends this extension of Theorem \ref{th.IVgff} to the case of the random-phase Sine-Gordon model.
\qed

\bibliographystyle{alpha}
\bibliography{biblio}

\end{document}